\documentclass[aap, reqno] {imsart}
\usepackage{color}
\usepackage{amssymb}
\usepackage{comment}

\usepackage{amsthm,amsmath}
\RequirePackage[numbers]{natbib}
\RequirePackage[colorlinks=true, pdfstartview=FitV, linkcolor=blue,  citecolor=blue, urlcolor=blue]{hyperref}
\RequirePackage{hypernat}
\usepackage{paralist}
\usepackage{graphicx}
\usepackage{amsfonts}
\usepackage{amsmath}
\usepackage{amssymb}
\usepackage{amsbsy}
\usepackage{fullpage}
\usepackage{epsfig}
\usepackage{natbib, mathrsfs} 
\usepackage{verbatim}
\usepackage[latin1]{inputenc}
\usepackage{mhequ}

\numberwithin{equation}{section}

\usepackage{enumerate}
\usepackage{hyperref}

\newtheorem{theorem}{Theorem}[section]
\newtheorem{lemma}[theorem]{Lemma}
\newtheorem{remark}[theorem]{Remark}
\newtheorem{Example}[theorem]{Example}
\newtheorem{remarks}[theorem]{Remarks}

\newtheorem{assumptions}[theorem]{Assumptions}

\newcommand{\kpo}{_{k+1}}
\newcommand{\kk}{_{k}}
\newcommand{\kpon}{_{k+1}^{i,N}}
\newcommand{\kn}{_{k}^{i,N}}
\newcommand{\bn}{^{(N)}}
\newcommand{\nors}[1]{\|#1\|_s}

\newcommand{\ra}{\rightarrow}
\newcommand{\lra}{\longrightarrow}

\newcommand{\kjn}{^{j,N}_k}
\newcommand{\kkn}{_k^N}
\newcommand{\inn}{^{i,N}}
\newcommand{\jnn}{^{j,N}}

\newcommand{\lv}{\left\vert}
\newcommand{\rv}{\right\vert}

\newcommand{\pa}{\partial}
\newcommand{\dd}{,\dots,}
\newcommand{\be}{\begin{equation}}
\newcommand{\ee}{\end{equation}}
\newcommand{\PP}{\mathcal{P}}

\newcommand{\lan}{\langle}
\newcommand{\ran}{\rangle}

\newcommand{\les}{\lesssim}
\newcommand{\hsint}{\mathcal{H}^s_{\cap}}
\newcommand{\hsintt}{\mathcal{H}^s_{\cap\cap}}

\newcommand{\tr}{\textup{Trace}}

\newcommand{\cl}{\mathcal{L}}

\newcommand{\N}{\mathbb{N}}

\newcommand{\R}{\mathbb{R}}
\newcommand{\hs}{\mathcal{H}^s}
\newcommand{\h}{\mathcal{H}}
\newcommand{\beq}{\begin{equs}}
\newcommand{\eq}{\end{equs}}
\newcommand{\al}{A_{\ell}}
\newcommand{\dl}{D_{\ell}}
\newcommand{\Gl}{\Gamma_{\ell}}
\newcommand{\el}{_{\ell}}
\newcommand{\C}{\mathcal{C}}
\newcommand{\EE}{\mathbb{E}}
\newcommand{\cN}{\mathcal{N}}
\newcommand{\dist}{\stackrel{\mathcal{D}}{\sim}}
\newcommand{\exo}{\EE_{x_0}}

\newcommand{\xs}{x^{\sharp}}
\newcommand{\xdag}{x^{\dag}}

\newcommand{\Ws}{W^{\sharp}}
\newcommand{\Wdag}{W^{\dag}}

\newcommand{\Ss}{\mathfrak{S}^{\sharp}}
\newcommand{\Sdag}{\mathfrak{S}^{\dag}}

\newcommand{\ws}{w^{\sharp}}
\newcommand{\wdag}{w^{\dag}}

\begin{document}

\begin{frontmatter}
\title{Diffusion Limit For The  Random Walk Metropolis Algorithm  Out Of stationarity\protect}

\runtitle{RWM out of stationarity}

\begin{aug}
\author{ \snm{Juan Kuntz,}\corref{}\ead[label=e2]{juan.kuntz08@imperial.ac.uk}}
  \author{ \snm{Michela Ottobre}\corref{}\ead[label=e1]{m.ottobre@hw.ac.uk}}
\and
  \author{ \snm{Andrew M. Stuart}\thanksref{t5}\ead[label=e3]{astuart@caltech.edu}}

\thankstext{t5}{Supported by ERC and EPSRC}

\runauthor{J. Kuntz, M.Ottobre and A.M. Stuart}

  \affiliation{Imperial College, Heriot Watt University and Warwick University}

 \address{ Michela Ottobre\\ Heriot Watt University, 
 Mathematics Department, \\
Edinburgh, EH14 4AS, UK \\
\printead{e1}}
  \address{ Juan Kuntz\\
Imperial College London,\\
   SW7 2AZ, London,  UK,\\
          \printead{e2}}
 \address{ Andrew M. Stuart\\
Department of Computing and Mathematical Sciences, \\
California Institute of Technology, CA 91125, USA
          \printead{e3}}
\end{aug}

\begin{abstract}

 The Random Walk Metropolis (RWM) algorithm is a Metropolis-Hastings Markov Chain Monte Carlo algorithm designed to sample from a given target distribution $\pi^N$ with Lebesgue density on $\R^N$.  Like any other Metropolis-Hastings algorithm, RWM constructs a Markov chain by randomly proposing a new position (the ``proposal move"), which is then accepted or rejected according to a rule which makes the chain reversible with respect to $\pi^N$.  When the dimension $N$ is large a key 
question is to determine the optimal scaling with $N$ of the proposal variance: if the proposal variance is too large, the algorithm will reject the proposed moves too often; if it is too small, the algorithm will explore the state space too slowly. Determining the optimal scaling of the proposal variance gives a measure of the cost of the algorithm as well.  One approach to tackle this issue, which we adopt here, is to 
derive diffusion limits for the algorithm. Such an approach has been proposed in the seminal papers \cite{MR1428751,Robe:Rose:98}; in particular in \cite{MR1428751} the authors derive a diffusion limit for the RWM algorithm under the two following assumptions: i) the algorithm is started in stationarity; ii) the target measure $\pi^N$ is in product form. The present paper considers the situation of practical interest in which both assumptions i) and ii) are removed. That is a) we study the case (which occurs in practice) in which the algorithm is started out of stationarity and b) we consider target measures which are in non-product form. In particular,  we work in the setting in which 
families of measures on spaces of increasing dimension are found by 
approximating a measure, on an infinite dimensional Hilbert space, which is defined by its density with respect to a Gaussian. {The target measures that we consider arise in Bayesian nonparametric statistics and in the study of conditioned diffusions.} 
We prove that, out of stationarity, the optimal scaling for the proposal variance
is $O(N^{-1})$, as it is in stationarity. In this optimal scaling a diffusion 
limit is obtained and the cost of reaching and exploring the invariant measure
scales as $O(N)$. Notice that the optimal scaling
in and out of stationatity need not be the same in general, and indeed they differ e.g. 
in the case of the MALA algorithm \cite{Kuntz2016}.
\end{abstract}

\begin{keyword}[class=MSC]
\kwd[Primary ]{60J22}
\kwd[; secondary ]{60J20, 60H10}
\end{keyword}

\begin{keyword}
\kwd{Markov Chain Monte Carlo, Random Walk Metropolis algorithm, diffusion limit, optimal scaling}
\end{keyword}

\end{frontmatter}

\maketitle


\section{Introduction}

\subsection{Setting and Main Result}  
Metropolis-Hastings algorithms are popular MCMC methods used to sample from a given target measure, $\pi^N$, defined via its density with respect to Lebesgue
measure on $\R^N$  (with abuse of notation, we
often denote both the measure and the density with the same letter).  The basic mechanism consists of employing a proposal transition density $q(x,y)$ in order to produce a reversible chain $\{x_k\}_{k=0}^{\infty}$ which has the target measure as invariant distribution \cite{MR1620401}. At step $k$ of the chain, a proposal move $y_{k+1}$ is generated by using a proposal kernel $q(x,y)$, i.e. 
$y_{k+1} \sim q(x_k, \cdot)$. Then such a move is accepted with probability $\alpha(x_k, y_{k+1})$, where
$$
\alpha(x,y)= \min\left\{1, \frac{\pi(y) q(y,x)}{\pi(x) q(x,y)}  \right\}\,.
$$
If the move is accepted then the chain is updated  to the state $x\kpo:=y\kpo$, otherwise $x\kpo:=x_k$. When the proposal kernel $q(x,y)$ is symmetric in its variables,  the expression for the acceptance probability simplifies to 
$$
\alpha(x,y)= \min\left\{1, \frac{\pi(y)}{\pi(x)}  \right\}\,.
$$
Random Walk Metropolis (RWM) belongs to the family of Metropolis-Hastings algorithms with symmetric proposal, as the proposal move is generated according to a random walk. A key question for Metropolis-Hastings methods in general, and for
RWM in particular, is to determine the cost of the algorithm
as a function of the dimension $N$. The present paper aims at studying the cost of the RWM algorithm by the
use of diffusion limits. Precisely, we identify scalings of the
proposal variance with resepct to the dimension $N$ which lead to a diffiusion
limit. Since the inverse proposal variance has the interpretation as a time-step
in a discretization of the limiting diffusion, this scaling determines the
number of steps required to reach and explore the desired target distribution.
 We study the situation of practical interest where the algorithm
 is started out of stationarity and the target measure is in non-product form.

 In what follows we first introduce the class of target measures that we will be considering and we then specify the RWM algorithm for such a class of targets (more details on the algorithm and on the class of target measures  can be found in Section \ref{sec:algor} and in Section \ref{sec:2}, respectively). We then clarify the problem that is the  subject of the paper, we present our main result and, immediately after (see Remark \ref{rem:onsindepofx}), we explain the practical implications of such a result in terms of cost of the algorithm (in this context we will specify what we mean by ``cost of the algorithm"). 

The class of target measures  that we consider  are determined
by approximations of a measure on an infinite dimensional Hilbert space. In particular, let $\pi$ be a probability measure  defined on an infinite dimensional separable Hilbert space ($\h, \langle \cdot, \cdot \rangle, \| \cdot \|$)  and absolutely continuous with respect to a Gaussian measure $\pi_0$ with mean zero and covariance operator $\C$:
\be\label{targetmeasure}
\frac{d\pi}{d\pi_0}  \propto \exp({-\Psi}), \qquad \pi_0 \dist \cN(0,\C),
\ee
where $\Psi: \tilde{\h}\ra \R$ is some real valued functional with domain $\tilde{\h} \subseteq \h$ and $\pi_0(\tilde{\h})=1.$ In Section \ref{sec:2} we will detail our assumptions on $\Psi$ and give the precise definition of the space $\tilde{\h}$ and identify it with an appropriate Sobolev-like subspace of $\h$ (denoted by $\h^s$ in Section \ref{sec:2}).    The covariance operator $\C$ is a positive, self-adjoint, trace class operator on 
$\h$, with eigenbasis $\{\lambda_j^2, \phi_j\}_{j \in \N} $:
\be\label{cphi}
\C \phi_j= \lambda_j^2 \phi_j, \quad \forall j \in\mathbb{N},
\ee
where $\{\phi_j\}_{j \in \N}$ is an orthonormal basis of $\h$. 
We will analyse the RWM algorithm designed to sample from the finite dimensional projections $\pi^N$
 of the measure \eqref{targetmeasure} {on the space} 
\be\label{eq:subspace}
\h \supset X^N:=\textrm{span}\{\phi_j\}_{j=1}^N 
 \ee 
spanned by the first $N$ eigenvectors of the covariance operator. {Notice that the space $X^N$ is isomorphic to $\R^N$.  To clarify this further, we need to introduce some notation. 
Given a point $x \in \h$,  $x^N:=\PP^N(x)$ is the projection of $x$
onto the space $X^N$; $x^{i,N}$ will be the $i$-th component of the vector $x^N\in \R^N$, i.e. $x\inn=\langle \phi_i, x^N\rangle$. \footnote{Notice that if $x^N=\PP^N(x)$ and $1 \leq i \leq N$ then $x\inn=\langle \phi_i, x^N\rangle= \lan \phi_i, x\ran$.} Similar 
notation is also used for $y, \xi$ and other vectors; we do not give details.
We will also denote  $\Psi^N(x):=\Psi(\PP^N(x))$ and $\C_N$ will be, effectively,  an $N\times N$ diagonal matrix with $i$-th diagonal component equal to $\lambda_i^2$. More formally,}
\be\label{defpsiNCN}
\Psi^N:=\Psi\circ \PP^N \quad \mbox{and} \quad \C_N:=\PP^N\circ \C\circ \PP^N. \ee
With this notation in place,  our target measure is the measure $\pi^N$ (on $X^N \cong \R^N $) defined as 
\be\label{piN}
\frac{d\pi^N}{d\pi_0^N}(x)=M_{\Psi^N}e^{-\Psi^N(x)},  \qquad \pi_0^N \sim\cN(0,\C_N), 
\ee
where $M_{\Psi^N}$ is a normalization constant. Notice that the sequence of measures $\{\pi^N\}_{N\in \N}$ approximates the measure $\pi $ (in particular, the sequence $\{\pi^N\}_{N\in \N}$ convereges to $\pi$ in the Hellinger metric \cite{Stua:10}).

 Letting $\ell>0$ denote a positive parameter,  consider the RWM algorithm with proposal 
\be\label{proposal}
y=x+\sqrt{\frac{2 \ell^2}{N}}\C_N^{1/2}\xi^N,  \qquad  \xi^N= \sum_{j=1}^N \xi\jnn \phi_j, \quad 
\xi\jnn \dist \cN(0,1) \,\, \mbox{i.i.d.}.
\ee
  The current position $x$ and the proposal $y$ belong to $\h$; however, because the noise is finite dimensional, effectively only the first $N$ components of $x$ are modified when a proposal is accepted, 
namely the components belonging to $X^N$.  

Using the proposal \eqref{proposal} we construct the RWM - Markov chain $\{x\kk\}_{k}\subset \h$, through the ``accept-reject" mechanism described earlier. In computational practice one uses the projected chain $x^N_k=\PP^N(x_k)$, which samples from the measure $\pi^N$, i.e.  for any fixed $N \in \N$, the chain $\{x_k^N\}_{k\in \N}\subset X^N$ can be used to sample from the measure $\pi^N$. However  we often work in $\h$ rather than in $X^N$ (and therefore consider the chain $\{x_k\}_{k\in \N }$ rather than the chain $\{x_k^N\}_{k\in \N}$) only because in $\h$ the analysis is cleaner.  

To explain the problem at hand consider for a moment, instead of the proposal \eqref{proposal}, the following proposal:
\be\label{propbeta}
y=x+\sqrt{\frac{2 \ell^2}{N^{\beta}}}\C_N^{1/2}\xi^N,  
\ee
where $\beta>0$ is a positive parameter to be chosen. As is well known, if $\beta$ is ``too large" then the proposal variance (that is, informally, the size of the jumps of the chain) is ``too small", therefore the algorithm will move in state space very slowly. On the other hand, if $\beta$ is ``too small" then the proposal variance is too large and the algorithm will tend to reject the proposed moves too frequently (and this is more and more the case as the dimension $N$ increases). We will show that the value of $\beta$ that strikes the balance between these two opposing scenarios is $\beta=1$.  

We are now in a position to present our main result: let $x\bn (t)$  be the 
continuous interpolant of the chain $\{x_k\}$, namely
\be\label{interpolant}
x\bn (t)=(Nt-k)x\kpo+(k+1-Nt)x_k, \quad t_k\leq t< t_{k+1},  
\mbox{ where } t_k=k/N.
\ee
The main result of this paper is the diffusion limit for the RWM algorithm started out of stationarity. We informally state such a result below,  
with the functions  $\dl, \Gl$ and $\al$  defined immediately
after the statement.  The rigorous statement of the result, with
precise conditions, appears in  Theorem \ref {thm:weak conv of Skn} and Theorem \ref{thm:mainthm1}.
 Below we denote by $C([0,T]; \tilde{\h})$ the space of $\tilde{\h}$- valued continuous  functions on $[0,T]$, endowed with the uniform topology. 
\smallskip

{\bf Main Result.} {\em  
  Let $\{x_k\}_{k \in \N}$ be the Markov chain constructed using the RWM proposal \eqref{proposal} and starting from the (deterministic) initial datum $x_0\in \tilde{\h}$. 
Assume 
\begin{equation}
\label{eq:sic}
S_0:=\lim_{N \ra \infty}
 \frac{1}{N}\sum_{j=1}^N \frac{\lv x_0^{j,N} \rv^2}{\lambda_j^2} < \infty.
\end{equation}
 Then  the continuous interpolant of the chain $x_k$, i.e. the sequence of processes 
$x\bn (t)$ defined in \eqref{interpolant},  converges weakly in $C([0,T];\tilde{\h})$ (as $N \ra \infty$)  to the  solution of the SDE
\be\label{informalSPDe}
dx(t)=[-x(t)-\C\nabla\Psi(x(t)) ] \dl (S(t)) \, dt+\sqrt{\Gl (S(t))}  \, dW(t)  , \quad x(0)=x_0\, ,
\ee
where $S(t) : \R_+ \ra \R_+:=\{s\in \R: s\geq0\}$ is a deterministic function which solves the ODE
\be\label{ODE}
dS(t)=A_{\ell}(S(t))\, dt, \qquad S(0)=S_0\,, 
\ee
and $W(t)$ is a $\tilde{\h}$-valued $\tilde{\C}$-Brownian motion. \footnote{The operator that here we denote generically by $\tilde{\C}$, to avoid getting in too much notation at this stage,  will be more clearly defined in Section \ref{sec:2} and there denoted by $\C_s$. More precisely, as we will explain, $W(t)$ is a Brownian motion with covariance $\C_s$, see Section  \ref{sec:2}. }
}
\smallskip

If we denote by $\Phi(x)$  the cdf of a standard Gaussian, the functions $\dl, \Gl, \al : \R_+ \ra \R $ that appear in the above statement are defined as follows:  for $x> 0$ and $\ell>0$ a positive parameter, we define
\begin{align} 
\dl(x) & :=2\ell^2 e^{\ell^2 (x-1)} \Phi\left( \frac{\ell (1-2x)}{\sqrt{2x}}\right), \label{defD}\\
\Gl(x) & := \dl(x)+ 2\ell^2 \Phi\left( -\frac{\ell}{\sqrt{2x}} \right),  \label{defGamma}\\
\al(x) & := (1-2x) \dl (x)+ 2\ell^2 \Phi\left( -\frac{\ell}{\sqrt{2x}} \right) = -2x \dl(x)+ \Gl(x) \label{defAl}
\end{align}
 and for $x=0$ and $\ell>0$ we set
\be\label{at0}
\dl (0)=\Gl (0)=\al(0)= 2\ell^2 e^{-\ell^2}.
\ee

\begin{remark}\label{rem:onsindepofx}
\textup{
We make several remarks concerning the main result.}
\begin{itemize}
\item  \textup{ The effective time-step implied by the interpolation \eqref{interpolant}
is $N^{-1}$ so, in this sense,  the main result indicates that, started out of stationarity, the RWM
algorithm will take ${\cal O}(N)$ steps to reach and  explore target measures found by
approximating $\pi$ in $\R^N$. In this respect, we say that the computational cost of the algorithm is of order $N$. To put it differently, our result proves that the proposal variance which
delivers a diffusion limit scales like $N^{-1}$ with dimension and that, 
therefore, the cost of the algorithm is of order $N$.
 \item Notice that equation \eqref{ODE} evolves independently of equation \eqref{informalSPDe}.  
Once the RWM chain $\{x_k\}_k$  is introduced (see \eqref{chaininh} for a precise description of the chain) and an initial state $x_0\in \tilde{\h}$ is given such that $S(0)$ is finite, the  real valued (double) sequence $S_k^N$, }
\be\label{skn}
S_k^N:=\frac{1}{N} \sum_{i=1}^N \frac{\lv x_k^{i,N}\rv^2}{\lambda_i^2}
\ee
\textup{started at 
$S_0^N:=\frac{1}{N} \sum_{i=1}^N \frac{\lv x_0^{i,N}\rv^2}{\lambda_i^2}$ is well defined. We can then consider the continuous interpolant $S\bn (t)$ of the chain $\{ S_k^N\} \subset \R_+$, namely
\be\label{interpolantofsk}
S\bn (t)=(Nt-k)S\kpo^N+(k+1-Nt)S_k^N, \quad t_k\leq t< t_{k+1},  
\mbox{ where } t_k=k/N.
\ee
In Theorem \ref{thm:weak conv of Skn} we  prove that $S\bn(t)$ converges in probability in $C([0,T];\R)$  to the solution of \eqref{ODE} with initial condition 
$S_0:=\lim_{N\ra\infty}S_0^N$. Once such a result is obtained, we can prove that $x\bn(t)$ converges to $x(t)$. We want to stress that the convergence of 
$S\bn(t)$ to $S(t)$ can be obtained independently of the convergence of $x\bn(t)$ to $x(t)$. Moreover, notice that $S_k^N$ is not a Markov Chain in general (unless e.g. $\Psi = 0$.)}
\item  \textup{
Let $S(t):\R_+\rightarrow \R_+$ be the solution of the ODE \eqref{ODE}. We will prove (see Theorem \ref{thm:existenceuniquenessforODE}) that   $S(t) \ra 1$  as $t\ra \infty$. With this in mind,   notice that
$\dl(1)= 2\ell^2 \Phi(-\ell / \sqrt{2})=:h\el$ and $\Gl(1)=2\dl(1)= 2 h\el$.  Heuristically one can then argue that the asymptotic behaviour of the law of $x(t)$, 
solution of \eqref{informalSPDe}, 
is described by the law of the following infinite dimensional SDE:
\be\label{longlimSPDEinf}
dz(t)=-h\el (z+\C \nabla \Psi(z))+ \sqrt{2h\el} dW.
\ee
It was proved in \cite{Hair:etal:05, Hair:Stua:Voss:07}  that \eqref{longlimSPDEinf} is ergodic with unique invariant measure given by  our target measure \eqref{targetmeasure}. 
Our deduction concerning computational cost is made on the assumption
that the law of \eqref{informalSPDe} does indeed tend to the law of
\eqref{longlimSPDEinf}, although we will not prove this here as it would take
us away from the main goal
of the paper which is to establish the diffusion limit of the RWM algorithm. } 
\end{itemize}
\end{remark}
{\hfill$\Box$}

\subsection{Relation to the Literature}\label{subs:rellit}

As already explained, in this paper we consider target measures in 
non-product form, when the chain is started out of stationarity. 
When the target measure is in product form, a diffusion limit for the resulting Markov chain was studied in the seminal paper \cite{MR1428751}, where it is assumed that
\be\label{targetproduct}
p(x^N)=\Pi_{i=1}^N e^{-V(x^{i,N})}, \quad x^N=(x^{1,N} \dd x^{N,N}) \in \R^N,  
\ee
and the potential $V$ is such that the measure $p$ is normalized. 
That work assumed that the chain is started in stationarity, leading to the 
conclusion that, in stationarity, ${\cal O}(N)$ steps are required to explore 
the target distribution. In \cite{MR2137324} the same question was addressed
in the case where $p$ is the density of an isotropic Gaussian,
when the chain is started out of stationarity. Recently the papers
\cite{JLM12MF,JLM12LT} made the significant extension of considering
the product case \eqref{targetproduct} for quite general potentials $V$,
again out of stationarity. The work in \cite{MR2137324,JLM12MF} demonstrates
that the same scaling of the proposal variance is required both in and out
of stationarity, in the product case, and that then ${\cal O}(N)$ steps are required to explore
the target distribution. Again recently, diffusion limits for RWM started in stationarity have 
also been considered  for measures in non-product form \cite{Matt:Pill:Stu:11}, using families
of target measures found 
by approximating \eqref{targetmeasure}, as we consider in this paper; 
once again the conclusion is that ${\cal O}(N)$ steps are required to explore 
the target distribution. In the present paper we combine the settings of
\cite{Matt:Pill:Stu:11} and \cite{JLM12MF} and make a significant extension
of the analysis to consider
measures in non-product form, when the chain is  started out of stationarity,
again showing that ${\cal O}(N)$ steps are required to explore the target 
distribution.

In \cite{MR1428751} the diffusion limit is for a single coordinate of the Markov
chain and takes the form
\be\label{McKean0}
dX(t)=-h_{\ell}V'(X(t)) dt+\sqrt{2h_{\ell}}dB(t),
\ee
with $X_t\in \R$ and $B(t)$ a one dimensional Brownian motion. 
Each coordinate of the Markov chain has the same weak limit.
In \cite{JLM12MF,JLM12LT} a similar limit is obtained for each
coordinate, but because the system is out of stationarity the
coordinates are coupled together, leading to a  
one dimensional nonlinear (in the sense of McKean) diffusion process
\be\label{McKean}
dX(t)=-d_{\ell}(t)V'(X(t)) dt+\sqrt{2g_{\ell}(t)}dB(t),
\ee
with $X_t\in \R$ and $B(t)$ a one dimensional Brownian motion and
$$d_{\ell}(t)=\mathcal{G}\el \left( \EE\left[V'(X(t))\right]^2,\EE\left[V''(X(t))\right]\right), \quad
g_{\ell}(t)=\frac12\tilde{\Gamma}\el \left( \EE\left[ V'(X(t))\right]^2,\EE\left[ V''(X(t))\right] \right).$$
 The  definition of the  functions $\mathcal{G}_{\ell}$ and $\tilde{\Gamma}_{\ell}$  can be found in \cite[(1.7) and (1.6)]{JLM12MF}, respectively. While we don't repeat the full definition here, we point out the two main facts which are relevant in the present context: i) 
in stationarity $d_{\ell}(t)=h_{\ell}$ and $g_{\ell}(t)=h_{\ell}$
and so \eqref{McKean} is identical to \eqref{McKean0}, but out of stationarity
the variation of these quantities reflects what remains
of the coupling between different coordinates in the limit of large $N$;
 ii) regarding the functions $\dl(x)$ and $\Gl(x)$ (defined in \eqref{defD} and \eqref{defGamma}, respectively), notice that $\dl(x)=\mathcal{G}_{\ell \sqrt{2}}(x,1)$,  $\Gl(x)=\tilde{\Gamma}_{\ell \sqrt{2}}(x,1)$.

In \cite{Matt:Pill:Stu:11}, since the target measure is no longer of product form,
the continuous interpolant of the RWM chain $x_k$ defined in \eqref{chaininh} 
has diffusion limit given by the solution of the infinite dimensional SDE \eqref{longlimSPDEinf}, when the chain is started in stationarity. In contrast, in this
paper where we study the same target measure as in \cite{Matt:Pill:Stu:11}, but
started out of stationarity, the limiting diffusion is \eqref{informalSPDe}, with
$S(t)$ solving \eqref{ODE}. 
The relationship between \eqref{McKean0} and \eqref{McKean} is entirely
analogous to the relationship between \eqref{longlimSPDEinf} and \eqref{informalSPDe}.
It is natural to ask, then, why we do not obtain an infinite dimensional
nonlinear (in the McKean sense) diffusion process as the limit in this paper?
The reason for this is related to the fact that our underlying reference
measure is Gaussian. Indeed in the case of Gaussian product measure
the limiting diffusion \eqref{McKean} simplifies in the sense that the
the equations for $d_{\ell}(t)$ and $g_{\ell}(t)$ depend only
on the process $X$ through the quantity $M(t):=\EE(X_t)^2$ and it is 
explicitly noted in   \cite{JLM12MF} that $M(t)$ solves precisely the ODE \eqref{ODE}. 
It is also relevant to observe at this point that the weak limit 
$S\bn \stackrel{d}{\lra}S $ (in $C([0,T], \R_+)$) has already been proven in \cite{MR2137324} in the
Gaussian case where all the components $x\kn$ are identically distributed.

On a technical note, we observe that in \cite{JLM12MF,JLM12LT} the symmetry of the target measure   allows the authors to employ propagation of chaos techniques so that 
these  two papers have brought together two thus far distant worlds: MCMC  and  
probabilistic  methods for nonlinear PDEs.  
In our case,  due to the lack of symmetry in the proposal,  the propagation of chaos point of view cannot be used so we base our analysis on the more ``hands on" approach used in \cite{Matt:Pill:Stu:11}.  As already mentioned, the latter paper is devoted to the study of the diffusion limit for the same chain that we are analysing here and in the same infinite dimensional context as well. 
The difference with our paper is that the chain in  \cite{Matt:Pill:Stu:11} is started in stationarity. As a consequence, the quantity that here we call $S(t)$ is, in their case, equal to 1 for every $t\geq 0$; to better phrase it, if we start the chain in stationarity,  then 
\be\label{skstat}
 S_k^N=\frac{1}{N}\sum_{i=1}^N \frac{\lv x\inn_k\rv^2}{\lambda_i^2}\stackrel{N \ra \infty}{\lra} 1\, ,\quad \mbox{almost surely, for all}\, k \ge 0.
\ee
Recalling that $S(t)\ra 1$ as $t \ra \infty$, this is coherent with our results.
Although the approach we use here is similar to the one developed in \cite{Matt:Pill:Stu:11},
significant extensions of that work are required in order to handle the
technical complications introduced by the non stationarity of the chain.
  Throughout the paper we will flag up the main steps where our analysis differs from 
that in \cite{Matt:Pill:Stu:11} (see in particular Section \ref{subs52}, the comments at the end of Section \ref{sec:heurdiffcoeffforchainx} and  Remark \ref{rem:differ1}). Let us just say for the moment that if we start the chain in stationarity then $x_k^N \sim \pi^N$ for all $k\geq 0$. Because $\pi^N$ is a change of measure from a Gaussian measure, all the almost sure properties of the chain only need to be shown for $x \sim \pi_0$. In the non stationary case we cannot reduce the analysis to the Gaussian case and therefore some of the estimates become  more involved.
The above discussion motivates our interest in the problem studied in this paper: on the one hand we want to extend the analysis of \cite{JLM12MF} away from the non-practical i.i.d. product form for the target; on the other hand we drop the assumption of stationarity in   \cite{Matt:Pill:Stu:11}.

 We mention for completeness  that  the non stationary case  has also  been considered in \cite{Pillai2014,Ottobre2016}, for the pCN (preconditioned Crank-Nicolson) algorithm and for the SOL-HMC (Second Order Langevin - Hamiltonian Monte Carlo) scheme, respectively. These algorithms are well-defined in the infinite dimensional
limit and hence do not require a scaling of the time-step
which is inversely proportional to a power of the dimension. On a related note, we remark that when we want to sample from measures of the form \eqref{targetmeasure}, RWM is not the optimal choice. Indeed both pCN and the SOL-HMC exactly preserve the  Gaussian measure $\pi_0$ and hence,  in the case $\Psi\equiv 0$, such algorithms are exact; it is for this reason that they are well-defined in the infinite 
dimensional limit, and do not require a scaling of the time-step
with dimension. However it is still of interest to study the behaviour of RWM on
measures of the form \eqref{targetmeasure} because they provide an explicit class
of non-product measures for which analysis is possible and for which the scaling of cost with dimention is the same as in the product case, suggesting broader validity of the conclusions in the papers \cite{{MR2137324}, JLM12MF,JLM12LT}. 

\subsection{Outline of Paper}
The paper is organized as follows. 
In the next Section \ref{sec:algor} we present in more detail the RWM algorithm.  In Section \ref{sec:2} we introduce the notation that we will use in the rest of the paper and the  assumptions we make on the nonlinearity $\Psi$ and on the covariance operator $\C$.   Section \ref{sec:sec4} contains the proof of existence and uniqueness for the limiting  equations \eqref{informalSPDe} and \eqref{ODE}. With these preliminaries  in place, we give, in Section \ref{sec:res and heur}, the precise statement of the main results of this paper,  Theorem \ref{thm:weak conv of Skn} and Theorem    \ref{thm:mainthm1}. In Section \ref{sec:res and heur} we also provide  heuristic arguments to explain how the main results are obtained. Such arguments are then made rigorous in Section \ref{sec:proof of wekconvof Skn} and Section  \ref{sec:proofofmainthm}, which contain the proof of  Theorem \ref{thm:weak conv of Skn} and Theorem    \ref{thm:mainthm1}, respectively. The continuous mapping  argument on which these proofs rely is presented in Section \ref{sec:contmapping}.

\section{The Algorithm}\label{sec:algor}
Once the current state $x$ of the chain is given, the proposed move \eqref{proposal} depends only on the noise $\xi^N$. For this reason, in defining the acceptance probability for our algorithm, we can use the notations $\alpha(x^N,y^N)$ or $\alpha(x^N,\xi^N)$ exchangeably. With this in mind, 
 let us define the acceptance probability
\be \label{accprob}
\alpha(x^N,\xi^N):=1 \wedge \exp{(Q(x^N,\xi^N))}
\ee
where
\be\label{defQ}
Q(x^N,\xi^N):= \frac{1}{2} \| \C^{-1/2} x^N\|^2- \frac{1}{2} \| \C^{-1/2} y^N\|^2 + \Psi(x^N)-\Psi(y^N).
\ee
 Consider the Markov chain $\{x_k\}_{k=0}^{\infty} \subset \h$ constructed as follows
\be\label{chaininh}
x_{k+1}=x_k+\gamma_{k+1} \sqrt{\frac{2 \ell^2}  {N}}\C_N^{1/2}\xi_{k+1}^N \, ,  \footnote{Notice that also the state of the chain $\{x_k\}_{k\in\N} \subset \h$ depends on $N$, as only the first $N$ components are updated. However this is not reflected in the notation to avoid confusion between the finite-dimensional chain $\{x_k^N\} \subset X^N$ and the infinite-dimensional chain $\{x_k\} \subset \h$. }
\ee
where
$$
\gamma_{k+1}\dist \textup{Bernoulli}( \alpha_{k+1}) \qquad \mbox{with} \qquad \alpha\kpo=
\alpha(x\kk^N,\xi\kpo^N).
$$
That is, given $\alpha_{k+1}$, the random variable $\gamma_{k+1}$ is independent of any other source of noise and has  Bernoulli law with mean $\alpha (x^N_k , \xi^{N}_{k+1})$.
Therefore, \eqref{chaininh}  can be spelled out as follows: if the chain is currently in $x_k$, the proposal 
$$
y_{k+1}= x\kk +\sqrt{\frac{2 \ell^2}{N}}(\C_N)^{1/2}\xi\kpo^N
$$
is accepted with probability $\alpha_{k+1}$ and rejected with probability $1-\alpha_{k+1}$.  We specify that in the above 
$$
\xi\kpo^N:=\sum_{i=1}^N \xi\inn\kpo \phi_i, \quad \mbox{ where } \quad \xi\inn\kpo\dist \cN(0,1) \mbox{ i.i.d.},
$$
and therefore  for $\alpha_k$, $Q$ and $\gamma_k$ actually depend on $N$ (we suppress the superscript $N$ in the notation   for convenience). 
In a less compact notation,  \eqref{chaininh} and \eqref{defQ}  can be rewritten as 
\begin{align}
x\kpon &=x\kn+ \gamma\kpo \sqrt{\frac{2 \ell^2}{N}} \, \lambda_i  \, \xi_{k+1}\inn, \qquad \mbox{for }i=1 \dd N  
\label{chainxcomponents}\\
x\kpo &=x_k=x_0 \qquad \mbox{on } \h \setminus X^N \nonumber 
 \end{align}
and
\be\label{extdefQ}
Q_k:=Q(x_k^N, \xi\kpo^N)=  \frac{1}{2}\sum_{i=1}^N\frac{ | x\kn |^2}{\lambda_i^2} - 
\frac{1}{2}\sum_{i=1}^N\frac{ | y_{k+1}^{i,N} |^2}{\lambda_i^2}+\Psi(x_k^N)-\Psi(y_{k+1}^N),
\ee
respectively. As we have already observed in the introduction,  in computational practice the above algorithm is implemented in $\R^N$. 
That is, for any $N$ fixed,   in order to sample from the measure $\pi^N$ (defined in \eqref{piN}), one considers the  projected chain $\{x^N_k=\PP^N(x_k)\}_{k \in \N}$.

\section{Preliminaries} 
\label{sec:2}
In this section we detail the notation and the assumptions (Subsection \ref{subsec:notation}
and Subsection  \ref{sec:assumptions} , respectively) that we will use in the rest of the paper.  
\subsection{Notation}\label{subsec:notation}
Let $\left( \h, \langle\cdot, \cdot \rangle, \|\cdot\|\right)$ 
denote an infinite dimensional separable Hilbert space  with
the canonical norm derived from the inner-product. 
Let $\C$ be a  positive, trace class operator on $\h$ 
and $\{\phi_j,\lambda^2_j\}_{j \geq 1}$ be the eigenfunctions
and eigenvalues of $\C$ respectively, so that
 \eqref{cphi} holds.
We assume a normalization under which $\{\phi_j\}_{j \geq 1}$ 
forms a complete orthonormal basis in $\h$.
Throughout the paper we will use the following notation:
\begin{itemize}
\item The letter $N$ denotes exclusively the dimensionality of the space $X^N$ (defined in \eqref{eq:subspace}) where the target measure $\pi^N$ is supported.
\item As already stressed in the introduction, if $x \in \h$, then $x^N:=\PP^N(x)$ is the projection of $x$ on the space $X^N$ defined in  \eqref{eq:subspace}. 
For every $x \in \h$ we have the representation
$x = \sum_{j} \; x^j \phi_j$, where here $x^j=\langle x,\phi_j\rangle$, i.e. $x^j$ is the $j$-th component of $x$. $x^{j,N}$ denotes the $j$-th component of $x^N$, so that $x^j= x^{j,N}$, for $1 \leq j \leq N$. Similar notation holds for the proposal vector $y$ and the noise vector $\xi$ as well. 
\item $x^N_k$ denotes the $k$-th step of projected  chain $\{\PP^N(x_k)\} \subset X^N$, where  $x_k$ has been defined in \eqref{chaininh}. Accordingly, $x_k^{i,N}$ is the $i$-th component of the vector $x_k^N \in X^N$. 
\end{itemize}
Using this notation, we define Sobolev-like spaces $\h^r, r \in \R$, with the inner-products and norms defined by
\begin{equs}
\langle x,y \rangle_r = \sum_{j=1}^\infty j^{2r}x^j y^j
\qquad \text{and} \qquad
\|x\|^2_r = \sum_{j=1}^\infty j^{2r} \, \lv x^j\rv^{2}.
\end{equs}
$(\h^r, \langle \cdot, \cdot \rangle_r)$ is a Hilbert space. 
Notice that $\h^0 = \h$. Furthermore
$\h^r \subset \h \subset \h^{-r}$ for any $r >0$.  
The Hilbert-Schmidt norm $\|\cdot\|_\C$ is defined as
\be\label{norc}
\|x\|^2_\C =\|\C^{-\frac12}x\|^2= \sum_{j=1}^{\infty} \lambda_j^{-2} \lv x^j\rv^2
\ee
and it is the Cameron-Martin norm associated with the Gaussian $\cN(0,\C)$. 
For $r \in \mathbb{R}$, 
let  $L_r : \h \rightarrow \h$ denote the operator which is
diagonal in the basis $\{\phi_j\}_{j \geq 1}$ with diagonal entries
$j^{2r}$, \textit{i.e.},
$$
L_r \,\phi_j = j^{2r} \phi_j, \qquad  j \in \N,
$$
so that $L^{\frac12}_r \,\phi_j = j^r \phi_j$. 
The operator $L_r$ 
lets us alternate between the Hilbert space $\h$ and the interpolation 
spaces $\h^r$ via the identities:
\beq 
\langle x,y \rangle_r = \langle L^{\frac12}_r x,L^{\frac12}_r y \rangle 
\qquad \text{and} \qquad
\|x\|^2_r =\|L^{\frac12}_r x\|^2. 
\eq
Since $\|L_r^{-1/2} \phi_k\|_r = \|\phi_k\|=1$, 
we deduce that $\{\hat{\phi}_k:=L^{-1/2}_r \phi_k \}_{k \geq 1}$ forms an 
orthonormal basis for $\h^r$. If  $y\sim N(0,\C)$,    then $y$ can be  expressed as
\beq[y1]
y=\sum_{j=1}^{\infty}
\lambda_j \rho_j \phi_j  \qquad \mbox{with } \qquad \rho_j\stackrel{\mathcal{D}}{\sim}N(0,1) \,\,\mbox{i.i.d};
\eq
if $\sum_j \lambda_j^2 j^{2r}<\infty$ then $y$ can be equivalently written as
\beq[y2]
y=\sum_{j=1}^{\infty}
(\lambda_j j^r) \rho_j  \hat{\phi}_j  \qquad \mbox{with } \qquad \rho_j\stackrel{\mathcal{D}}{\sim}N(0,1) \,\,\mbox{i.i.d}.
\eq
For a positive, self-adjoint operator $D : \h \mapsto \h$, its trace in $\h$ is defined as
\begin{equs}
\tr_{\h}(D) := \sum_{j=1}^\infty \langle \phi_j, D  \phi_j \rangle.
\end{equs}
We stress that in the above $ \{ \phi_j \}_{j \in \mathbb{N}} $ is an orthonormal basis for $(\h, \langle \cdot, \cdot \rangle)$. Therefore, if 
$\tilde{D}:\h^r \rightarrow \h^r$, its trace in $\h^r$ is
\begin{equs}
\tr_{\h^r}(\tilde{D}) := \sum_{j=1}^\infty \langle L_r^{-\frac{1}{2}} \phi_j, \tilde{D} L_r^{-\frac{1}{2}} \phi_j \rangle_r.
\end{equs}
Since $\tr_{\h^r}(\tilde{D})$ does not depend on the orthonormal basis,
the operator $\tilde{D}$ is said to be trace class in $\h^r$ if $\tr_{\h^r}(\tilde{D}) < \infty$ for
some, and hence any, orthonormal basis of $\h^r$. 

Because $\C$ is defined on $\h$, the covariance operator
$$
\C_r=L_r^{1/2} \C L_r^{1/2}
$$
is defined on $\h^r$. With this definition, 
for all the values of $r$ such that $\tr_{\h^r}(\C_r)=\sum_j \lambda_j^2 j^{2r}< \infty$, we can think of $y$ as a mean zero Gaussian random variable with covariance operator $\C$ in $\h$ and $\C_r$ in $\h^r$ (see \eqref{y1} and \eqref{y2}).  In the same way, 
if $\tr_{\h^r}(\C_r)< \infty $ then 
$$
W(t)= \sum_{j=1}^{\infty} \lambda_j w_j(t) \phi_j= \sum_{j=1}^{\infty}\lambda_j j^r w_j(t) \hat{\phi}_j,
$$
with $\{ w_j(t)\}_{j\in{\N}}$  a collection of  i.i.d. standard Brownian motions on $\mathbb{R}$, 
can be equivalently  understood as an $\h$-valued $\C$-Brownian motion or as an  $\h^r$-valued $\C_r$-Brownian motion.  

\noindent
Throughout we use the following notation.
\begin{itemize}
\item Two sequences $\{\alpha_n\}_{n \geq 0}$ and $\{\beta_n\}_{n \geq 0}$ satisfy $\alpha_n \lesssim \beta_n$ 
if there exists a constant $K>0$ (independent of $n$), such that $\alpha_n \leq K \beta_n$ for all $n \geq 0$.
The notations $\alpha_n \asymp \beta_n$ means that $\alpha_n \lesssim \beta_n$ and $\beta_n \lesssim \alpha_n$.

\item 
Two sequences of real functions $\{f_n\}_{n \geq 0}$ and $\{g_n\}_{n \geq 0}$ defined on the same set $\Omega$
satisfy $f_n \lesssim g_n$ if there exists a constant $K>0$ (independent of $n$) satisfying $f_n(\omega) \leq K g_n(\omega)$ for all $n \geq 0$
and all $\omega \in \Omega$.
The notations $f_n \asymp g_n$ means that $f_n \lesssim g_n$ and $g_n 
\lesssim f_n$. Similarly, for  two functions $f(x)$ and $g(x)$, we write $f(x)\les g(x)$ if there exists a constant $K>0$ (independent of $x$) such that $f(x)\leq K g(x)$ for all $x$ where the two functions are defined. 

\item
The notation $ \EE_x \left[ f(x,\xi) \right]$ denotes expectation 
with variable $x$ fixed, while the randomness present in $\xi$
is averaged out.
\end{itemize}
As customary,  $\R_+:=\{s\in \R : s \geq 0\}$ and for all $b \in \R_+$ we let $[b]=n$ if $n\leq b < n+1$ for some integer $n$. Finally,  for time dependent functions we will  use both the notations $S(t)$ and $S_t$ interchangeably.  

\subsection{Assumptions} \label{sec:assumptions}
In this section we describe the assumptions on the covariance 
operator $\C$ of the Gaussian measure $\pi_0 \dist \cN(0,\C)$ and the functional $\Psi$. We fix a distinguished exponent 
$s\geq 0$ and assume that $\Psi: \mathcal{H}^s\rightarrow \R$ and $\tr_{\mathcal{H}^s}(\mathcal{C}_s)<\infty$. In other words the space $\h^s$ is the one that we were denoting by $\tilde{\h}$ in the introduction. 
For each $x \in \h^s$ the derivative $\nabla \Psi(x)$
is an element of the dual $(\h^s)^*$ of $\h^s$, comprising 
the linear functionals on $\h^s$.
However, we may identify $(\h^s)^* = \h^{-s}$ and view $\nabla \Psi(x)$
as an element of $\h^{-s}$ for each $x \in \h^s$. With this identification,
the following identity holds
\begin{equs}
\| \nabla \Psi(x)\|_{\mathcal{L}(\h^s,\R)} = \| \nabla \Psi(x) \|_{-s}; 
\end{equs}
furthermore, the second derivative $\partial^2 \Psi(x)$ 
can be identified with an element 
of $\mathcal{L}(\h^s, \h^{-s})$.
To avoid technicalities we assume that $\Psi(x)$ is quadratically bounded, 
with first derivative linearly bounded at infinity and second derivative globally 
bounded. 
\begin{assumptions} \label{ass:1}
The functional $\Psi$ and covariance operator $\C$ satisfy the following assumptions.
\begin{enumerate}
\item {\bf Decay of Eigenvalues $\lambda_j^2$ of $\C$:}
there exists a constant $\kappa > \frac{1}{2}$ such that
\begin{equs}
\lambda_j \asymp j^{-\kappa}.
\end{equs}

\item {\bf Domain of $\Psi$:}
there exists an exponent $s \in [0, \kappa - 1/2)$ such 
that $\Psi$ is defined everywhere on $\h^s$.

\item {\bf Size of $\Psi$:} 
the functional $\Psi:\h^s \to \R$ satisfies the growth conditions
\begin{equs}
0 \quad\leq\quad \Psi(x) \quad\lesssim 1+  \|x\|_s^2
\end{equs}

\item {\bf Derivatives of $\Psi$:} 
The derivatives of $\Psi$ satisfy
\begin{equs}\label{der-s}
\| \nabla \Psi(x)\|_{-s} \quad\lesssim  \|x\|_s^{\varsigma} \vee  \|x\|_s  \qquad \text{and} \qquad
\|{\partial^2 \Psi(x)}\|_{\cl(\h^s;\h^{-s})} \quad\lesssim\quad 1,
\end{equs}
for some $ 1/2 \leq \varsigma<1$. 

\end{enumerate}
\end{assumptions}

\begin{remark} \label{rem:one}\textup{
Regarding the first of Assumptions \ref{ass:1}, the condition $\kappa > \frac{1}{2}$ ensures that $\tr_{\h^s}(\C_s)  < \infty$
for any $0 \leq s  < \kappa - \frac{1}{2}$; this implies that
$\pi_0(\h^s)=1$ 
for any  $ 0 \leq s < \kappa - \frac{1}{2}$.  As for the first of the requirements in \eqref{der-s}, this is slightly less general than the corresponding condition imposed in \cite{Matt:Pill:Stu:11} (there it is required that $\| \nabla \Psi(x)\|_{-s} \quad\lesssim  1 + \|x\|_s$). This is to avoid excessive technicalities (particularly  in the proof of \eqref{lembound26}, which is the only place where this simplification is actually used, see Remark \ref{rem810} and Remark \ref{remvalue} on this point). }
 {\hfill$\Box$}
\end{remark}

\begin{Example}\textup{
The functional $\Psi(x)  = \frac{1}{2}\|x\|_s^2$ is defined on $\h^s$ and its derivative at $x \in \h^s$
is given by $\nabla \Psi(x) = \sum_{j \geq 0} j^{2s} x^j \phi_j \in \h^{-s}$ with  
$\|\nabla \Psi(x)\|_{-s} = \|x\|_s$. The second derivative $\partial^2 \Psi(x) \in \mathcal{L}(\h^s, \h^{-s})$ 
is the linear operator that maps $u \in \h^s$ to $\sum_{j \geq 0} j^{2s} \langle u,\phi_j \rangle \phi_j \in \h^{-s}$:
its norm satisfies $\| \partial^2 \Psi(x) \|_{\mathcal{L}(\h^s, \h^{-s})} = 1$ for any $x \in \h^s$. }  {\hfill$\Box$}
\end{Example}

\noindent
The Assumptions \ref{ass:1} ensure that the functional $\Psi$ behaves 
well in a sense made precise in the following lemma.  We set 
\be\label{defFz}
F(z)=-z-\C \nabla \Psi(z). 
\ee

\begin{lemma} \label{lem:lipschitz+taylor}
Let Assumptions  \ref{ass:1} hold.
\begin{enumerate}
\item  The function $\C \nabla \Psi(z)$ is globally Lipshitz on $\hs$ and hence the same holds for
the function  $F(z)$:
\begin{equs} 
\|F(x) - F(y)\|_s \;\lesssim \; \|x-y\|_s
\qquad \qquad \forall x,y \in \h^s.  
\end{equs}
\item 
The second order remainder term in the Taylor expansion of $\Psi$ satisfies
\begin{equs} \label{e.taylor.order2}
\big| \Psi(y)-\Psi(x) - \langle\nabla \Psi(x), y-x \rangle \big| \lesssim \|y-x\|_s^2
\qquad \qquad \forall x,y \in \h^s.
\end{equs}
\end{enumerate}
\end{lemma}
\begin{proof}
See \cite{Matt:Pill:Stu:11}.
\end{proof}
We would also like to recall  that because of our assumptions on the covariance operator, for all $p\geq 0$ there is a constant $c=c(p)$ such that
\be\label{c1/2xi}
\EE\nors{(\C_N)^{1/2} \xi^N}^p \leq c, \quad \mbox{uniformly in $N$, }
\ee
if $\xi^N$ is the Gaussian defined in \eqref{proposal}. We will prove this inequality in Appendix A. For the moment we just stress that $c>0$ is  a constant independent of $N$ but that does depend on $p$.

\section{Existence And Uniqueness For the Limiting SDE} \label{sec:sec4}
The main statements of this section are Theorem \ref{thm:existenceuniquenessforODE}, Theorem \ref{Thm:SPDe} and Theorem \ref{contofupsilon}. In Theorem \ref{thm:existenceuniquenessforODE} and Theorem \ref{Thm:SPDe}   we prove existence and uniqueness for  the solution to equation \eqref{ODE} and equation \eqref{informalSPDe}, respectively. Theorem \ref{contofupsilon} is a ``continuous mapping" result and it is crucial for  the arguments of Section \ref{sec:contmapping} (and, ultimately, it is the backbone of the proof of our main results).

\begin{theorem}\label{thm:existenceuniquenessforODE}
For any initial datum $S(0) \in \R_+$,  there exists a unique  solution  $S(t)$ to the ODE \eqref{ODE}. Such a solution  is strictly positive for every $t>0$. Furthermore, $S(t)$ is  bounded with continuous first derivative for all $t\geq 0$.   In particular                                   
\be\label{Stto1}
\lim_{t\ra \infty} S(t) =1 \,
\ee
and 
\be\label{solofODEisbdd}
0\leq \min\{S(0),1\}\leq S(t) \leq \max\{S(0), 1\} \, , \qquad \mbox{for } t \geq 0. 
\ee
\end{theorem}
Before proving the above theorem we state Lemma \ref{lem:propofDandGamma}, which gathers all the properties of the  real valued functions $\dl, \Gl$ and $\al$, defined in 
\eqref{defD}-\eqref{at0}.

\begin{lemma}\label{lem:propofDandGamma}
The functions $\dl(x)$,  $\Gl(x)$ and  $\sqrt{\Gl (x)}$ are positive, globally Lipshitz continuous and bounded, with bounded first derivative.    $\al(x)$  is bounded above but not below; it  has continuous first derivative on the whole  of $\R_+$ and it is globally Lipshitz.
Moreover, for any $\ell>0$,  $\al(x)$ is strictly positive for $x\in[0,1)$, strictly negative for $x>1$ and $\al(1)=0$. 
\end{lemma}
\begin{proof}[Proof of lemma \ref{lem:propofDandGamma}]
The proof of the above Lemma \ref{lem:propofDandGamma} follows from the same arguments used in \cite[Proof of Lemma 2]{JLM12MF}. We sketch the proof in Appendix A for completeness. A plot of the function $\al(x)$ for various values of $\ell$ can be found in \cite[page 258]{MR2137324}. Figure \ref{fig:al} contains a plot of $\al(x)$ for $\ell=1$ and $\ell=2$. Plots of the functions $\dl, \Gl$ and of the derivative of $\al$ can be found in Appendix A. 
\end{proof}

\begin{figure}[ht]
\centering
\includegraphics[height=7cm]{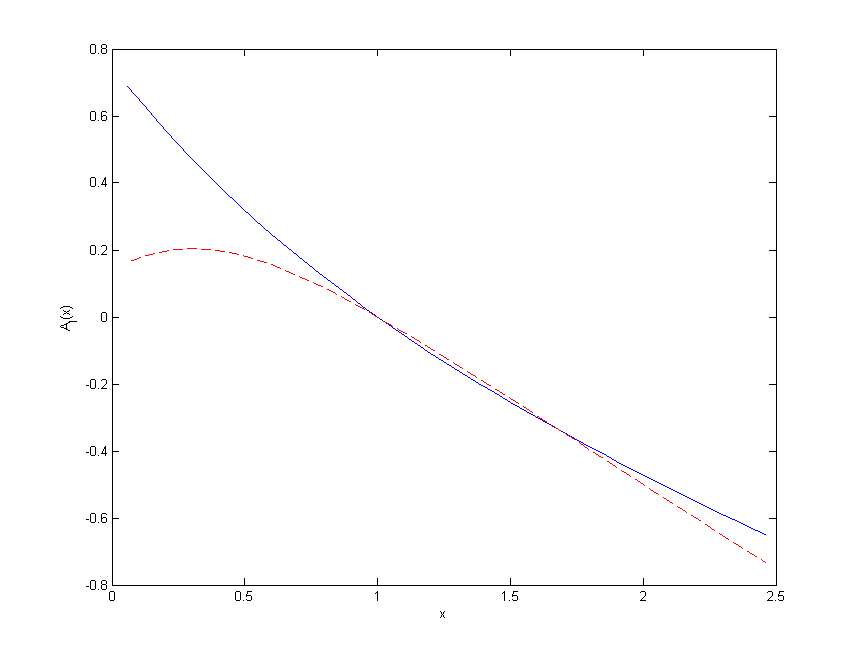}
\caption{Plots of the function $\al(x)$ for $\ell=1$ and $\ell=2$ (dashed line).}
\label{fig:al}
\end{figure}

\begin{proof}[Proof of Theorem \ref{thm:existenceuniquenessforODE}]
Existence and uniqueness  for \eqref{ODE} is standard, since $\al$ is globally Lipshitz. The limit \eqref{Stto1} and the bound \eqref{solofODEisbdd} are a consequence of the last statement of Lemma \ref{lem:propofDandGamma}.  Indeed, if we start with an initial datum $S_0 \in [0,1)$ then $S(t)$ will increase towards 1. If $S(0)>1$ then $S(t)$ will decrease towards 1. \\
\end{proof}
We now come to existence and uniqueness for equation 
\eqref{informalSPDe}, which we rewrite  as 
$$
dx(t)=F(x(t))\dl(S(t)) dt+\sqrt{\Gl(S(t))} dW(t),
$$
where $W(t)$ is an $\hs$ valued 
$\C_s$-Brownian Motion and the function $F$ has been defined in  \eqref{defFz}.  The above  is a short notation for the integral form
\be\label{xsolofSPDE}
x(t)=x(0)+\int_0^t F(x(v))\dl(S(v)) dv+ \int_0^t \sqrt{\Gl(S(v))} dW(v) 
\, .
\ee
In view of the next statement we emphasize that throughout the paper the spaces $C([0,T];\hs)$ and $C([0,T];\R_+)$ are  endowed with the uniform topology. 
\begin{theorem}\label{Thm:SPDe}
Let Assumption \ref{ass:1} hold. Then,  for any initial condition
$  x(0)  \in \hs$, any $T>0$ and  every
$\h^s$-valued ${{\C}}_s$-Brownian motion $W(t)$, there exists a
unique solution of  equation \eqref{informalSPDe}  (with $S(t)$ given by \eqref{ODE}) in the space
$C([0,T]; \h^s)$.
\end{theorem}
Before proving the above theorem, let us make a remark on the statement.
\begin{remark}\textup{
In the statement of Theorem \ref{Thm:SPDe} we refer to equation \eqref{informalSPDe} or, equivalently, to equation \eqref{xsolofSPDE}. With the notation introduced so far, the function $S(t)$ appearing in \eqref{xsolofSPDE} is the solution of the ODE \eqref{ODE}. However the proof of Theorem \ref{Thm:SPDe} is still valid if  $S(t): \R_+ \rightarrow  \R_+$ is any continuous and  bounded function with continuous first derivative.
}{\hfill$\Box$}
\end{remark}

\begin{proof}[Proof of Theorem \ref{Thm:SPDe}] 
Once we prove continuity of the map 
\begin{align*}
\mathcal{J}\,:\,\hs \times C([0,T]; \hs)  & \lra C([0,T];\hs)\\
(x(0), W(t))               & \lra   x(t) 
\end{align*}              
where $x(t)$ is defined by \eqref{xsolofSPDE}, 
 existence and uniqueness for equation \eqref{xsolofSPDE}  for a small enough time interval follow from a standard contraction mapping argument, see e.g. \cite{Matt:Pill:Stu:11}.  To show the continuity of such a map, let $x^{\sharp}(t)$ and $x^{\dag}(t)$ be the images through the map $\mathcal{J}$ of the pairs 
$(x^{\sharp}(0),\Ws(t))$ and $(\xdag(0), \Wdag(t))$,  respectively. For all $0 \leq t \leq T$,   we then have
\begin{align*}
\nors{\xs (t)-\xdag (t)} &\leq \int_0^t \nors{F(\xs_v)\dl(S_v)-F(\xdag_v) \dl(S_v)} dv  \\
&+\nors{\xs (0)-\xdag (0)}+
\Big\| \int_0^t \sqrt{\Gl(S_v)} (d\Ws_v - d\Wdag_v ) \Big\|_s.
\end{align*}
Thanks to the Lipshitzianity of $F$,  Lemma \ref{lem:lipschitz+taylor},   and the boundedness of $\dl$, Lemma  \ref{lem:propofDandGamma}, 
  the  drift coefficient of \eqref{xsolofSPDE}, i.e.
  $$
\Theta(x, S):= F(x)\dl(S), \quad (x, S) \in \hs \times \R_+,
$$
   is globally Lipshitz, uniformly in time. Therefore, integrating by parts in the stochastic integral,   we get 
\begin{align*}
\nors{\xs (t)-\xdag (t)} &\les \int_0^t \nors{\xs (v)-\xdag (v)} 
+ \nors{\xs (0)-\xdag (0)}\\
&+ \Bigl\|\Gl^{1/2}(S_t) (\Ws_t- \Wdag_t) 
- \int_0^t \frac{d}{dv}\left(\Gl^{1/2}(S_v) \right) (\Ws_v - \Wdag _v)\Bigr\|_s.
\end{align*}
We now further work on the right hand side of the above as follows:
\begin{align*}
\Bigl\|\int_0^t \frac{d}{dv}\left(\Gl^{1/2}(S_v) \right) (\Ws_v - \Wdag _v)\Bigr\|_s & \leq
\sup_{v \in [0,t]}\nors{\Ws_v - \Wdag _v}  \int_0^t \lv \frac{d}{dv}\left(\Gl^{1/2}(S_v) \right)\rv.
\end{align*}
Clearly,
$$
\lv \frac{d}{dv}\left(\Gl^{1/2}(S_v) \right) \rv= \frac{1}{2}\lv \left.\frac{\frac{d}{dx}\Gl(x)}{\Gl^{1/2}(x)} \right\vert_{x=S_v} (\al (S_v))  \rv.
$$
From the definition of $\Gl$ (see \eqref{defGamma} and \eqref{at0}), for any $x\geq 0$,   $\Gl$ is bounded below away from zero. Moreover, $\Gl$ has bounded derivative (see Lemma \ref{lem:propofDandGamma}) and $\al$, being continuous, is bounded on compacts. These facts, together with \eqref{solofODEisbdd},  imply the bound 
 \be\label{ulip2}\int_0^t \lv \frac{d}{dv}\left(\Gl^{1/2}(S_v) \right)\rv \les t, 
\ee
hence  
\be\label{ulip}
\nors{\xs (t)-\xdag (t)} \les \int_0^T \sup_{v \in [0,t]}\nors{\xs (v)-\xdag (v)} 
+ \nors{\xs (0)-\xdag (0)}+ T\sup_{t \in [0,T]} \nors{\Ws_t - \Wdag_t}.
\ee
Taking the supremum over $t \in [0,T]$ on the left hand side of the above gives the desired contractivity, thanks to the Gronwall Lemma, for a small enough time interval, say $[0,T_0]$ and hence a unique solution can be constructed for $t \in [0,T_0]$. Such a solution can then be extended  to  $t \geq 0$, thanks to the specific form of  \eqref{ulip}, which, we stress again, is a consequence of \eqref{solofODEisbdd}. Indeed, thanks to the fact that the drift of the equation is Lipshitz uniformly in time and to \eqref{ulip2}, the time dependence of the RHS of \eqref{ulip} will stay the same when we try and construct a solution starting from $T_0$. We will therefore be able to construct a solution over the interval $[T_0, 2T_0]$. Continuing inductively we can cover the whole real axis. This concludes the proof. 
\end{proof}
Consider now the following    equation
\be\label{SPDe1}
d{z} (t)=[-z(t)-\C\nabla\Psi(z(t))]  \dl (S(t)) \, dt + d\eta(t),
\ee
where $S(t)$ is the solution of \eqref{ODE} and  $\eta(t)$ is any time-continuous function taking values in $\hs$. Also, let 
$\mathfrak{S}(t):\R_+ \ra \R$ be the solution of 
\be\label{SPDe2}
d\mathfrak{S}(t)=A_{\ell}(\mathfrak{S}(t)) \, dt+ a\, dw(t),
\ee
where  $w(t)$ is a real valued standard Brownian motion and $a\in \R_+$ is a constant.
 \begin{remarks}\label{rems:deceqn}\textup{ Before stating the next theorem we need to be more precise  about equations \eqref{SPDe1} and \eqref{SPDe2}. 
\begin{itemize}
\item We consider equation \eqref{SPDe2}, which is \eqref{ODE} perturbed by noise, in view of the contraction mapping argument  (explained in Section \ref{sec:contmapping}) that we will use to prove our main results.  Observe that  \eqref{SPDe2} still admits a unique solution (by the Lipshitzianity of $\al$).  Analogous observations hold for equation \eqref{SPDe1}, which has the same structure as equation   \eqref{informalSPDe}.  
 \item The solution to \eqref{SPDe2} might not stay positive if started from a positive initial datum (as opposed to the solution to \eqref{ODE}, which preserves positivity). However $\al$ is only defined for positive arguments,   (see \eqref{defAl}). To make sense of the notation  in \eqref{SPDe2}, we extend  $\al$ to the negative semiaxis. In other words, the function $\al$ appearing in \eqref{SPDe2} is not the same $\al$ defined in \eqref{defAl}; we should use a different notation  for such a function but we refrain from doing so for simplicity. In conclusion, the function $\al(s)$ in \eqref{SPDe2} is intended to be a strictly positive function for any $\R \ni  s< 1$ and we fix it equal to 
1 if  $s\leq -1/2$; it smoothly interpolates between -1/2  and  $\al(0)$ 
    if  $-1/2 <s<0$ and it coincides with $\al(s)$ as defined in \eqref{defAl}
    if $s\geq 0$. Therefore such an $\al$ will still be globally Lipshitz.
\item We emphasize that \eqref{SPDe1} and  \eqref{SPDe2} are decoupled as the function $S(t)$ appearing in 
\eqref{SPDe1} is the solution of \eqref{ODE}. This fact will be  particularly  relevant in the remainder of this section as well as in  Section \ref{subs:cma1} and Section \ref{subs:cma2}.  
\end{itemize}
 } {\hfill$\Box$}
\end{remarks}
The  statement of the following theorem is  crucial to the proof of the main results of this paper, Theorem \ref{thm:weak conv of Skn} and Theorem \ref{thm:mainthm1}, stated in the next section.
\begin{theorem}\label{contofupsilon} With the notation introduced so far (and in particular with the clarifications of Remarks \ref{rems:deceqn}) 
let $z(t)$ and $\mathfrak{S}(t)$ be  solutions of  \eqref{SPDe1} and \eqref{SPDe2},  respectively.
  Then under Assumption \ref{ass:1} the maps
\begin{align*}
\mathcal{J}_1: \hs \times  C([0,T]; \hs)  & \lra C([0,T];\hs \times \R) \\
 (z_0, \eta(t))  & \lra z(t)
\end{align*}
and 
\begin{align*}
\mathcal{J}_2: \R_+ \times  C([0,T]; \R)  & \lra C([0,T]; \R) \\
 (\mathfrak{S}_0, w(t))  & \lra \mathfrak{S}(t)
\end{align*}
are  continuous maps. 
\end{theorem}
\begin{proof} Continuity of the map $\mathcal{J}_1$ can be shown with a calculation in the same spirit of the one done for the map $\mathcal{J}$, so we only sketch the proof of the continuity of the map $\mathcal{J}_2$. To this end we will use
\eqref{solofODEisbdd} and the Lipshitzianity of $\al$. Let $\Ss(t)$ and $\Sdag(t)$ be the images through the map $\mathcal{J}_2$ of the pairs 
$(\Ss_0, \ws(t))$ and $(\Sdag_0, \wdag(t))$,  respectively. Then
\begin{align}
\lv \Ss_t-\Sdag_t \rv &\leq \lv \Ss_0 -\Sdag_0 \rv + \int_0^t 
\lv \al(\Ss_v)-\al(\Sdag_v) \rv dv + a \lv \ws_t-\wdag_t \rv\nonumber\\
& \les  \lv \Ss_0 -\Sdag_0 \rv + \int_0^t 
\lv \Ss_v-\Sdag_v \rv dv  + a \lv \ws_t-\wdag_t \rv. \label{nlip}
\end{align}
Now we can conclude by Gronwall's Lemma. 
\end{proof}

\section{Statement of Main Theorems and Heuristics of Proofs}\label{sec:res and heur}
In this section we give a precise statement of the main results of the paper, Theorem \ref{thm:weak conv of Skn} and Theorem \ref{thm:mainthm1} below,  and outline the heuristic arguments which are at the basis of the proof of such results.  The rigorous proofs of Theorem \ref{thm:weak conv of Skn} and Theorem \ref{thm:mainthm1}  are detailed in Section \ref{sec:proof of wekconvof Skn} and Section \ref{sec:proofofmainthm}, respectively, and they 
consist in quantifying the formal  approximations  presented in this section. 
 The structure of such proofs relies on the    continuous mapping argument which  is presented in Section \ref{sec:contmapping}.

While describing the main intuitive ideas of the proof, we will also try and  emphasize the differences with  
the analysis presented in \cite{Matt:Pill:Stu:11} in the stationary case. 
Here and throughout the paper we will   use a notation analogous to the one used in \cite{Matt:Pill:Stu:11}. 
\subsection{Statement of Main Results}

Let us define the set $\hsint$ as follows:
\be\label{spaceint}
\hsint:=\left\{x \in \hs : \lim_{N \ra \infty} \frac{1}{N}\sum_{i=1}^N  \frac{\lv x^i \rv^2} { \lambda_i^2}< \infty \right\}\, .
\ee
 \begin{theorem}\label{thm:weak conv of Skn}
Let Assumption \ref{ass:1}  hold and let $x_0\in \hsint$. Let $\{S_k^N\}\subset \R_+$ be the double sequence defined in \eqref{skn} and started at $S_0^N=\frac{1}{N}\sum_{i=1}^N  \lv x_{0}^{i,N} \rv^2 / \lambda_i^2 $. Let $S\bn(t)$, defined in  \eqref{interpolantofsk},  be the continuous interpolant of $S_k^N$.  Then, as $N \ra \infty$,  $S\bn(t) $    converges weakly in $C([0,T]; \R)$ to the  solution $S(t)$ of the ODE \eqref{ODE} started at  $S_0:=\lim_{N\ra \infty}S_0^N$.
\end{theorem}
We will prove Theorem \ref{thm:weak conv of Skn}  in Section \ref{sec:proof of wekconvof Skn}. For the time being, let us make the following observations. 
\begin{remark}\label{rem:convprob}\textup{
Notice that the weak limit of the double sequence $S_k^N$ is a deterministic function, therefore the above theorem also implies convergence in probability in $C([0,T];\R)$ of  $S^{(N)}(t)$ to $S(t)$.}
{\hfill$\Box$}
\end{remark}
Let us now introduce the piecewise constant interpolant of the (double) sequence $S_k^N$, i.e. the (sequence of) functions $\bar{S}\bn(t)$ defined as follows:
\be\label{piecewiseconstinterS}
 \bar{S}\bn(t)=S_k^N, \quad \mbox{for }\,\, t_k\leq t < t_{k+1}, \,\, t_k=k/N\,.
\ee 
\begin{lemma} \label{lem:asconv} Under the assumptions of Theorem \ref{thm:weak conv of Skn}, for every fixed $t>0$, 
$$
{S}^{(N)}(t) \rightarrow S(t) \quad \mbox{ almost surely}
$$ 
and 
$$
\bar{S}^{(N)}(t) \rightarrow S^{(N)}(t) \quad \mbox{ almost surely}.
$$ 
Therefore, 
$$
\bar{S}^{(N)}(t) \rightarrow S(t) \quad \mbox{ almost surely}.
$$ 
\end{lemma}
\begin{proof}[Proof  of Lemma \ref{lem:asconv}] The proof of this lemma can be found in Appendix B. 
\end{proof}
Consider now the set $\hsintt$ defined as the set of  $x \in \hsint$ such that 
\begin{itemize}
\item for all  $p\geq 0$, 
\be\label{firstcondp}
\lim_{N \ra \infty} \sum_{i=1}^N  i^{2s}\lambda_i^2\frac{\lv x^i \rv^{2p}} { \lambda_i^{2p}}< \infty,
\ee
\item there exists some $\epsilon >0$, such that 
\be\label{boundbelowepsilon}
 \lim_{N \ra \infty} \frac{1}{N}\sum_{i=1}^N  \frac{\lv x^{i} \rv^2 }{ \lambda_i^2} \geq \epsilon >0.
\ee
\end{itemize}

\begin{theorem}\label{thm:mainthm1}
Let Assumption \ref{ass:1}  hold and  $x_0\in \hsintt$. Then, as $N \ra \infty$ the continuous interpolant $x\bn(t)$ of the chain 
$\{x_k\}_k \subset \hs$ (defined in \eqref{interpolant} and \eqref{chainxcomponents}, respectively) and started at $x_0$, converges weakly in $C([0,T]; \hs)$ to the  solution $x(t)$ of equation \eqref{informalSPDe} started at $x_0$. We recall that the time-dependent function $S(t)$ appearing in \eqref{informalSPDe} is the solution of the ODE \eqref{ODE}, started at $S_0:= \lim_{N \ra \infty} \frac{1}{N}\sum_{i=1}^N  \lv x_{0}^{i} \rv^2 / \lambda_i^2$.
\end{theorem}
We will prove  Theorem \ref{thm:mainthm1} in Section \ref{sec:proofofmainthm}.  Note that in the above statement we are picking a deterministic initial condition. However it is worth noting that $x_0\in \hsintt$ almost surely if $x_0$ is drawn at random from the stationary measure \eqref{targetmeasure} . We will make some remarks on condition \eqref{boundbelowepsilon} at the end of Subsection \ref{subs52}. As for condition \eqref{firstcondp}, strictly speaking this does not need to be satisfied for all $p\geq 0$;  a finite, sufficiently large $p$ would suffice. However we refrain from determining the optimal $p$, which would distract from the main goals of the paper, and 
we state the result as it is, based on \eqref{firstcondp}.

\subsection{Formal Analysis of the Acceptance Probability}\label{subs52}
Gaining an intuition about the behaviour of the  acceptance probability $\alpha(x, \xi)$, defined in \eqref{accprob}, is at the core of the proof of the main result of this paper,   Theorem \ref{thm:mainthm1}. 
We present here a formal calculation that helps impart such  intuition.  We stress again that the calculations of this section are purely formal and will be made rigorous from Section \ref{sec:proof of wekconvof Skn} on . In this spirit we will use the loose notation $A^N \simeq B^N$ when, for $N$ large, $A^N$ is ``approximately equal" to $B^N$, and $A^N \approx B^N$ when, for $N$ large, $A^N$ is ``approximately distributed" according to $B^N$. 
 
Let us recall the notation $\Psi^N:= \Psi \circ \mathcal{P}^N$ (that is,  $\Psi^N(x):=\Psi(\PP^N(x))$) and set
\be\label{defzetakn}
\zeta_k^N:= (\C_N)^{-1/2}x_k^N+(\C_N)^{1/2} \nabla \Psi^N(x_k), \quad \mbox{where }\,\, 
\nabla \Psi^N(x_k)=\PP^N(\nabla \Psi(x_k^N)).
\ee
With these definitions, we can  further rewrite the expression \eqref{extdefQ} for $Q(x_k^N, \xi\kpo^N)$:
\begin{align}
Q(x_k, \xi\kpo)&= - \frac{\ell^2}{N} \sum_{i=1}^N \lv \xi_{k+1}\inn\rv^2- \sqrt{\frac{2\ell^2}{N}} \sum_{i=1}^N
\frac{x\kn \xi\kpo\inn}{\lambda_i}+\Psi(x^N_k)-\Psi(y^N\kpo)\nonumber\\
&= - \frac{\ell^2}{N} \sum_{i=1}^N  \lv \xi_{k+1}\inn\rv^2- \sqrt{\frac{2\ell^2}{N}}
\lan \C^{-1/2}x_k^N, \xi\kpo^N\ran +\Psi(x_k^N)-\Psi(y_{k+1}^N)\nonumber\\
&= - \frac{\ell^2}{N}  \|\xi_{k+1}^N\|^2- \sqrt{\frac{2\ell^2}{N}}
\lan \zeta_k^N, \xi\kpo^N\ran +\Psi(x_k^N)-\Psi(y_{k+1}^N)+ \sqrt{\frac{2 \ell^2}{N}} \lan \C_N^{1/2}\nabla\Psi^N(x_k^N), \xi\kpo^N\ran\nonumber.
\end{align}
Therefore  setting 
$$
r^N(x_k,\xi_{k+1}):= \Psi(x_k^N)-\Psi(y_{k+1}^N)+ \sqrt{\frac{2 \ell^2}{N}}    \lan \C_N^{1/2}\nabla\Psi^N(x_k^N), \xi\kpo^N\ran
$$
and 
\be \label{Q=R+rN}
R(x_k^N, \xi\kpo^N):= - \frac{\ell^2}{N} \sum_{i=1}^N \lv \xi_{k+1}\inn\rv^2- \sqrt{\frac{2\ell^2}{N}}
\lan \zeta_k^N, \xi\kpo^N\ran,
\ee
 we obtain 
\be\label{decompQvera}
Q(x_k^N, \xi\kpo^N)=R(x_k^N, \xi\kpo^N)+ r^N(x_k, \xi\kpo)\,.
\ee
In \cite{Matt:Pill:Stu:11} it is shown that 
\be \label{57p} 
 \lv r^N(x_k,\xi)\rv \les  \frac{\| \C^{1/2} \xi\|_s^2}{N} \,; 
\ee
therefore
\be\label{exprn}
\EE \lv r^N(x_k,\xi_{k+1})\rv \les \frac{1}{N}, 
\ee
 see \cite[eqn. (2.32)]{Matt:Pill:Stu:11}. 
The above \eqref{57p}-\eqref{exprn} are true whether the chain is started in stationarity or not, as they are only a consequence of the properties of $\Psi$ (see \eqref{e.taylor.order2}) and of the noise $\xi_{k+1}$, see \eqref{c1/2xi}. 
Using  \eqref{exprn}, 
\be\label{defofR}
Q(x_k^N, \xi\kpo^N)\simeq R(x_k^N, \xi\kpo^N).
\ee
Looking at the definition of $R$, equation \eqref{Q=R+rN}, and observing that by the Law of Large Numbers
\be\label{LLNapprox}
\frac{1}{N} \sum_{j=1}^N \lv \xi_{k+1}\jnn\rv^2 \lra 1\, ,
\ee
we  deduce that $R\simeq G$ (see Lemma \ref{lem:Wassbound}),  where 
\be\label{defGi}
G:= -\ell^2 - \sqrt{\frac{2\ell ^2}{N}} \sum_{j =1}^N \zeta_k^{j,N}\xi_{k+1}\jnn,
\quad \mbox{ so that, given $x_k$, } \quad 
G \sim \cN\left(-\ell^2,  \frac{2\ell ^2}{N} \sum_{j =1}^N \lv\zeta_k^{j,N} \rv^2\right)\, .
\ee
 We will show 
$$
   \frac{1}{N}\sum_{j = 1}^N \lv\zeta_k^{j,N} \rv^2\simeq 
\frac{1}{N} \sum_{j=1}^N\frac{\lv x_k^{j,N}\rv^2}{\lambda_j^2} = S_k^N.
$$
This can be intuitively understood by observing that in \eqref{defzetakn} the ``dominating contribution" comes from the first addend. 
The above approximation is formalized by  \eqref{z=x+psi} and \eqref{boundonmomentsofx} and it implies   $G \approx Z_{\ell, k}$,  where 
\be\label{approx4} 
 Z_{\ell, k}:= -\ell^2 - \sqrt{\frac{2\ell^2}{N}}\sum_{j=1}^N \frac{x\kjn}
 {\lambda_j} \xi\kpo\jnn \quad \mbox{so that, given $x_k^N$,} \quad  Z_{\ell, k} \sim \cN(-\ell^2, 2\ell^2 S_{k}^N). 
\ee
In conclusion, the  formal analysis  presented so far suggests that we may use the approximations
\be\label{approx5}
Q(x_k^N, \xi\kpo^N) \simeq R \approx \cN \left(  -\ell^2, 2\ell^2\, S_k^N \right). \ee
 
In \cite{Matt:Pill:Stu:11} it is proved that if we start from stationarity then  the sequence
$S_k^N$ converges (for fixed $k$, as $N \ra \infty$) to $1$ almost surely (see \eqref{skstat}).  We will show that if  we start the chain out of stationarity, i.e. $x_0$ is any point in $\h^s$, then 
\be\label{weakcontost}
S_k^N=\frac{1}{N}\sum_{i=1}^N \frac{\lv x_k^{i,N}\rv^2}{\lambda_i^2}\stackrel{d}{\lra} S(t), \quad \mbox{as } N \ra \infty,\, \mbox{ for } t_k\leq t < t_{k+1},
\ee
where $t_k=k/N$ and $S(t)$ is the solution of the ODE \eqref{ODE}. This is the main conceptual difference between our work and \cite{Matt:Pill:Stu:11}, all the other differences are technical consequences of this fact. 

Looking at \eqref{approx5}-\eqref{weakcontost}, we can explain why we are assuming \eqref{boundbelowepsilon}: roughly speaking,  if the initial datum $S_0$ is strictly positive then the limit $S(t)$ is strictly positive for every $t\geq 0$, so  the Gaussian variable on the RHS of  \eqref{approx5} always has a strictly positive variance.  If instead $S_0=0$, then at zero one would have $ Q_0=Q(x_0, \xi_1) \simeq  -\ell^2$ and therefore the acceptance probability at the first step  simply tends to $e^{-\ell^2}$; however this would only be true at zero as, even if $S_0=0$, the solution of the ODE \eqref{ODE} becomes immediately strictly positive for $t>0$ (see Theorem \ref{thm:existenceuniquenessforODE}). To avoid having to take into account also this further possibility (which does not add anything to the overall understanding of the algorithm), and to streamline the analysis, we make the simplifying assumption \eqref{boundbelowepsilon}.

The approximation \eqref{approx5} dictates the behaviour of the acceptance probability. With the present algorithm the average acceptance probability does not tend to one (as $N \ra \infty$, for $t_k\leq t< t_{k+1}$). This is one of the disadvantages 
of using the method analysed in this paper, in comparison to using algorithms 
 which are well defined in infinite dimensions. 

\subsection{Formal Derivation of the Drift Coefficient of Equation \eqref{informalSPDe}}\label{sec:heurdiffcoeffforchainx}
Let us first clarify the use of the notation that we will make in the following. The definition of $x_{k+1}$ \eqref{chaininh} contains two sources of randomness: the Gaussian noise $\xi_{k+1}$ and the Bernoulli random 
variable $\gamma_{k+1}$. With this in mind, 
when we write $\EE_k(\cdot)$ we will mean expectation with respect to $\xi_{k+1}$  and $\gamma_{k+1}$, given $x_k$. In some cases, when we want to emphasize the fact that the expectation is taken with respect to $\xi_{k+1}$ and $\gamma_{k+1}$, we will write explicitly    $\EE_k^{\xi, \gamma}$. In the same way, if we want to stress that expectation is being taken with respect to $\xi_{k+1}$, we write  $\EE_k^{\xi}$. 
According to (\ref{chainxcomponents}), the $i$-th component of the approximate drift is given by
\begin{align}
N \EE_k(x\kpon-x\kn) &= N \EE_k\left( \gamma\kpo \sqrt{\frac{2\ell^2}{N}}\lambda_i \,\xi\kpo\inn \right)
= \sqrt{2N\ell^2} \lambda_i \, \EE_k^{\xi, \gamma}(\gamma\kpo\, \xi\kpo\inn) \nonumber \\
&= \sqrt{2N\ell^2} \lambda_i \, \EE_k^{\xi}(\alpha\kpo\, \xi\kpo\inn)=
 \sqrt{2N\ell^2} \lambda_i \, \EE_k^{\xi}
\left[ \left( 1\wedge e^{Q(x_k^N, \xi\kpo^N)}  \right)  \xi\kpo\inn\right].\label{appdr1}
\end{align}
(We briefly explain at the end of Appendix A how the first equality in \eqref{appdr1} is obtained.)  
For a reason that will be clear in a few lines we further split the RHS of \eqref{Q=R+rN} as follows \footnote{This splitting is standard in the analysis of high dimensional MCMC, see \cite{Matt:Pill:Stu:11}.}
\begin{align}
R(x_k^N, \xi\kpo^N)&= - \frac{\ell^2}{N} \sum_{j\neq i}^N \lv \xi_{k+1}\jnn\rv^2- \sqrt{\frac{2\ell^2}{N}}
\sum_{j\neq i} \zeta_k^{j,N} \xi\kpo^{j,N}-
\frac{\ell^2}{N}  \lv \xi_{k+1}^{i,N}\rv^2- \sqrt{\frac{2\ell^2}{N}}
 \zeta_k^{i,N} \xi\kpo^{i,N}\, \nonumber\\
&= : R^i(x_k^N, \xi\kpo^N)-
\frac{\ell^2}{N}  \lv \xi_{k+1}\inn\rv^2- \sqrt{\frac{2\ell^2}{N}}
 \zeta_k^{i,N} \xi\kpo\inn\, , \label{defRi}
\end{align}
hence
\be\label{approx2}
Q(x_k^N, \xi\kpo^N)\simeq  R^i(x_k^N, \xi\kpo^N)- \sqrt{\frac{2\ell^2}{N}}
 \zeta_k^{i,N} \xi\kpo\inn\,.
\ee
 Using \eqref{approx2} we then have
\be\label{approx3}
\EE_k^{\xi}
\left[ \left( 1\wedge e^{Q(x_k^N, \xi\kpo^N)}  \right)  \xi\kpo\inn\right] \simeq 
\EE_k^{\xi}
\left[ \left( 1\wedge e^{ R^i(x_k^N, \xi\kpo^N)- \sqrt{\frac{2\ell^2}{N}}
\zeta_k^{i,N} \xi\kpo^i } \right) \xi\kpo\inn\right]\,.
\ee
We now use \cite[eqn. (2.36)]{Matt:Pill:Stu:11}, which we recast here for the reader's convenience. 
\begin{lemma}\label{lemma:zexpz}
Let $X$ be a real valued r.v.,  $X \sim \cN(0,1) $. Then for any $a,b \in \R$,
\be\label{Lemma55}
\EE \left[X\left( 1 \wedge e^{aX+b} \right)\right]= a e^{\frac{a^2}{2}+b} \, \Phi\left( -\frac{b}{\lv a \rv }-\lv a \rv \right). 
\ee
\end{lemma}
\begin{proof} See \cite[Lemma 2.4]{Matt:Pill:Stu:11}.
\end{proof}
Now notice that, given $x_k$, $R^i$ is independent of $\xi\kpo^i$ as it only contains the random variables
 $\xi^j\kpo$ for $i\neq j$. Therefore the expected value $\EE_k^{\xi}$ can be calculated by first evaluating  
$\EE_k^{\xi^i}$ and then $\EE_k^{\xi^i_{-}}$, where the latter denotes expectation with respect to $\xi \backslash \xi^i$.
With this observation we can use the above Lemma \ref{lemma:zexpz} with $a= - \sqrt{\frac{2\ell^2}{N}}
\zeta_k^{i,N} $ and $b=R^i$ to further evaluate the RHS of \eqref{approx3}; we get
\begin{align}
\EE_k^{\xi}
\left[ \left( 1\wedge e^{Q(x_k^N, \xi\kpo^N)}  \right)  \xi\kpo\inn\right] & \simeq \EE_k^{\xi}
\left[ \left( 1\wedge e^{ R^i(x_k^N, \xi\kpo^N)- \sqrt{\frac{2\ell^2}{N}}
\zeta_k^{i,N} \xi\kpo^i } \right) \xi\kpo\inn\right] \nonumber\\
&\stackrel{\eqref{Lemma55}}{=}
-\sqrt{\frac{2\ell^2}{N}}\zeta\kn e^{\frac{\ell^2}{N}\lv \zeta\kn \rv^2}
\EE_k^{\xi^i_{-}} e^{R^i}\Phi\left(   
-\frac{R^i}{\sqrt{\frac{2\ell^2}{N}}\lv\zeta\kn \rv}-\sqrt{\frac{2\ell^2}{N}}\lv\zeta\kn \rv\right)\label{eqclar}\\
&\simeq   -\sqrt{\frac{2\ell^2}{N}}\zeta\kn \,
\EE_k^{\xi^i_{-}} e^{R^i}\Phi\left(   
\frac{-R^i}{\sqrt{\frac{2\ell^2}{N}}\lv\zeta\kn \rv}\right)\nonumber\\
&=
  -\sqrt{\frac{2\ell^2}{N}}\zeta\kn \,
\EE_k^{\xi} e^{R^i}\Phi\left(   
\frac{-R^i}{\sqrt{\frac{2\ell^2}{N}}\lv\zeta\kn \rv}\right)\nonumber\\
&\simeq   -\sqrt{\frac{2\ell^2}{N}}\zeta\kn \,
\EE_k^{\xi} e^{R^i} {\bf 1}_{\{R^i<0\}} \\
& \simeq -\sqrt{\frac{2\ell^2}{N}}\zeta\kn \,
\EE_k^{\xi} e^{R} {\bf 1}_{\{R<0\}} \, . \nonumber
\end{align}

Therefore, using the approximation \eqref{approx5} (and the notation \eqref{approx4}), 
\be\label{approx6}
\EE_k^{\xi}
\left[ \left( 1\wedge e^{Q(x_k^N, \xi\kpo^N)}  \right)  \xi\kpo^i\right]  \simeq
-\sqrt{\frac{2\ell^2}{N}}\zeta\kn \,
\EE_k^{\xi} e^{Z_{\ell,k}} {\bf 1}_{\{Z_{\ell,k}<0\}} \, . 
\ee
Now a straightforward calculation shows that if $X\sim \mathcal{N} (\mu, \sigma^2)$ then 
$$
\EE\left( e^X {\bf 1}_{X<0}\right)=e^{\mu+\sigma^2/2} \, \Phi\left(  -\frac{\mu}{\sigma}-\sigma  \right).
$$
In particular this means that if $X \sim \cN(-\ell^2, 2\ell^2 a)$, for some $a>0$,  then 
\be\label{ecarexp}
\EE\left( e^X {\bf 1}_{X<0}\right)=e^{\ell^2(a-1)}\Phi \left(\frac{\ell(1-2a)}{\sqrt{2a}}\right)=\frac{1}{2\ell^2}\dl(a).
\ee 
From  \eqref{ecarexp}, \eqref{approx6} and \eqref{approx4}, we then  get
$$
\EE_k^{\xi}
\left[ \left( 1\wedge e^{Q(x_k^N, \xi\kpo^N)}  \right)  \xi\kpo\inn\right]  \simeq  -\sqrt{\frac{2\ell^2}{N}}
  \zeta\kn \frac{1}{2\ell^2} \dl(S_{k}^N) =
- \frac{1}{\sqrt{2\ell^2N}}  
\zeta\kn \dl(S_{k}^N) \, .
$$
Combining the above with \eqref{appdr1} gives
$$
N \EE_k^{\xi}(x\kpon-x\kn)\simeq - \lambda_i\zeta\kn\,\dl(S_k^N),
$$
which is the desired drift, after observing that $\lambda_i\zeta\kn$ is the $i$-th component of $\C_N^{1/2}\zeta_k^N$ and 
$$
\C_N^{1/2}\zeta_k^N =x_k^N+\C_N \, \nabla\Psi^N(x_k).
$$
 As already mentioned in the introduction, as a consequence of \eqref{skstat}, if we started the chain in stationarity then the approximate drift would not be time dependent and we would have
$$
N \EE_k^{\xi}(x\kpon-x\kn)\simeq - \lambda_i\zeta\kn\,\dl(1),
$$
which is the approximate drift of \eqref{longlimSPDEinf}.

\subsection{Formal Derivation of the Diffusion Coefficient of Equation \eqref{informalSPDe}}
\begin{align}
N \EE_k(x\kpon -x_k^{i,N})(x_{k+1}^{j,N} -x_k^{j,N})&=N \EE_k^{\xi, \gamma}
\left( \gamma\kpo \sqrt{\frac{2\ell^2}{N}}  \lambda_i \xi\kpo\inn\right) \left( \gamma\kpo \sqrt{\frac{2\ell^2}{N}}  \lambda_j \xi\kpo\jnn\right)
\nonumber\\
&= 2\ell^2 \lambda_i \, \lambda_j \EE_k^{\xi} \left(  \xi\kpo^i \, \xi\kpo^j \, 
\left(1\wedge e^{ Q(x_k,\xi\kpo)}\right)  \right), \label{apprdif1}
\end{align}
where the last equality follows analogously to  \eqref{appdr1}.  
We consider \eqref{Q=R+rN} as before,   but this time we split
$$
R(x_k^N, \xi\kpo^N)= R^{ij}(x_k^N, \xi\kpo^N)-
\frac{\ell^2}{N}  \left( \lv \xi_{k+1}\inn\rv^2 + \lv \xi_{k+1}\jnn\rv^2 \right)- \sqrt{\frac{2\ell^2}{N}}
 \left( \zeta_k^{i,N} \xi\kpo\inn + \zeta_k^{j,N} \xi\kpo\jnn \right), 
$$
where
$$
 R^{ij}(x_k^N, \xi\kpo^N) =- \frac{\ell^2}{N} \sum_{h\neq i,j}^N \lv \xi_{k+1}^{h,N}\rv^2- \sqrt{\frac{2\ell^2}{N}}
\sum_{h\neq i,j} \zeta_k^{h,N} \xi\kpo^{h,N}  \, .
$$
As before, $ Q(x_k^N, \xi\kpo^N) \simeq R^{ij}(x_k^N, \xi\kpo^N) $, so that
\begin{align}
 \EE_k^{\xi} \left(  \xi\kpo\inn \, \xi\kpo\jnn \,\left(1\wedge
 e^{ Q(x_k^N,\xi\kpo^N)}  \right) \right)
& \simeq  
 \EE_k^{\xi} \left(  \xi\kpo\inn \, \xi\kpo\jnn \, \left(1\wedge 
e^{R^{ij}(x_k^N,\xi\kpo^N)} \right) \right)\nonumber\\
&= \delta_{ij}  \EE_k^{\xi_{-}^{ij} } \left( 1\wedge e^{R^{ij}(x_k^N,\xi\kpo^N)} \right)   \label{deltaij}\\
&=\delta_{ij}  \EE_k^{\xi} \left( 1\wedge e^{R^{ij}(x_k^N,\xi\kpo^N)} \right).\nonumber
\end{align}
With the same reasoning as in \eqref{approx5},  we have 
$$
Q(x_k^N, \xi\kpo^N) \simeq R^{ij} \approx \cN \left(  -\ell^2, 2\ell^2\, S_k^N \right).
$$
(Again, if we were to consider the stationary regime, then we would have $Q(x_k^N, \xi\kpo^N)  \approx \cN \left(  -\ell^2, 2\ell^2 \right).$)
Now a simple calculation shows that if $X\sim \cN (\mu, \sigma^2)$ then 
\be\label{expmin1ex}
\EE\left( 1 \wedge e^X \right)=e^{\mu+\sigma^2/2} \, \Phi\left(  -\frac{\mu}{\sigma}-\sigma  \right)+ 
\Phi\left( \frac{\mu}{\sigma}\right)  
\ee
and in particular if $X \sim \cN(-\ell^2, 2\ell^2 a)$ for some $a>0$, 
\be\label{GammaGamma}
\EE(1 \wedge e^X)= \frac{1}{2\ell^2}\Gamma\el(a).
\ee
Hence 
\be\label{aboutgamma}
 \EE_k^{\xi} \left( 1 \wedge e^{ R^{ij}(x_k^N,\xi\kpo^N)}  \right) \simeq \frac{1}{2\ell^2} \Gamma\el(S_k^N)\, .
\ee
Putting together  \eqref{apprdif1}, \eqref{deltaij} and \eqref{aboutgamma}  we get
$$
N \EE_k(x\kpon -x_k^{i,N})(x_{k+1}^{j,N} -x_k^{j,N})=\lambda_i \lambda_j \, \delta_{ij} \Gamma\el(S_k^N)\, .
$$
\subsection{Formal Derivation of Equation \eqref{ODE}} \label{subsec:derivofeqnforst}
 We now want to describe the heuristic derivation of the limit \eqref{weakcontost}. Let us start with the drift:
\begin{align}
N \EE_k (S\kpo^N-S_k^N) & =\EE_k \sum_{i=1}^N\left[  \frac{\lv x\kpon \rv^2}{\lambda_i^2}  - 
  \frac{\lv x\kn \rv^2}{\lambda_i^2}  \right]  \nonumber\\
&= \EE_k  \left[  \gamma\kpo  \left( \frac{2\ell^2}{N} \sum_{i=1}^N   \lv  \xi\kpo\inn\rv^2
+2 \sqrt{\frac{2\ell^2}{N} }\sum_{i=1}^N \frac{ x\kn \xi\kpo\inn}{\lambda_i}
  \right) \right]\label{useforsk-skp1}\\
& = \EE_k \left[ \left(1\wedge e^{Q(x_k, \xi\kpo)} \right)\left(-2R(x_k, \xi\kpo) \right)\right] + \EE_k\hat{r}^N \label{useinlemma}
\end{align}
where 
\be\label{rhatNdef}
\hat{r}^N:= -2 \sqrt{\frac{2\ell^2}{N}} \left[ \gamma\kpo
\langle  (\C_N)^{1/2} \nabla\Psi^N (x_k^N), \xi_{k+1}^N \rangle \right] \,.
\ee
We will show (as a consequence of \eqref{c12dpsi} and \eqref{boundonmomentsofx}) that $\hat{r}^N$ is negligible for large $N$. So, by \eqref{defofR} and \eqref{useinlemma},
\be \label{A1A1}
N \EE_k (S\kpo^N-S_k^N) \simeq  \EE_k \left[ \left(1\wedge e^{R(x_k^N, \xi\kpo^N)} \right)\left(-2R(x_k^N, \xi\kpo^N) \right)\right] . 
\ee
Now observe that if $X \sim \cN(\mu, \sigma^2)$ then, 
$$ 
\EE\left[-2X \left( 1\wedge e^{X} \right) \right]\stackrel{\eqref{Lemma55}, \eqref{expmin1ex}}{=}e^{\mu+\sigma^2/2} \Phi\left( -\frac{\mu}{\sigma}-\sigma\right)
(-2\mu-2\sigma^2)-2\mu\, \Phi\left( \frac{\mu}{\sigma} \right),
$$
so that, if $X \sim \cN(-\ell^2, 2\ell^2 a)$ for some $a>0$, we have 
\be\label{useinlemma1}
\EE(-2X(1\wedge e^X))=\al(a).
\ee
Therefore, by \eqref{approx5}, \eqref{A1A1} and the above,  we conclude
$$ N \EE_k (S\kpo^N-S_k^N)\simeq \al(S_k^N).  $$ 
Showing that the diffusion  coefficient  for $S_k^N$ vanishes is a consequence of the calculation that we have just done, indeed
\begin{align}
N\EE_k (S\kpo^N-S_k^N)^2&=\frac{1}{N}\EE_k  \left[  \sum_{i=1}^N \frac{\lv x\kpon \rv^2}{\lambda_i^2}  - 
  \frac{\lv x\kn \rv^2}{\lambda_i^2}  \right]^2 \nonumber\\
& \simeq \frac{1}{N} \EE_k\left[ (1\wedge e^{R(x_k^N,\xi\kpo^N)})^2  R^2(x_k^N,\xi\kpo^N)  \right] \nonumber\\
&\leq \frac{1}{N} \EE_k\left[  R^2(x_k^N,\xi\kpo^N)  \right] \simeq 
\frac{1}{N}\EE_k \lv Z_{\ell,k}\rv^2\simeq\frac{2 \ell^2 \, S_k^N}{N}\nonumber.
\end{align}
We will prove that   $S_k^N$'s  are uniformly bounded in $N$ and $k$ (in the sense of Lemma \ref{lem:bound1}),  hence $(2 \ell^2 \, S_k^N) / N \ra 0$.
\subsection{Suboptimal Scalings for the Proposal Variance}
\label{rem:betaneq1}
Consider the Random Walk algorithm with  proposal \eqref{propbeta}, for $\beta \neq 1$. 
In this case the acceptance probability becomes
$$
\alpha^{\beta}(x, \xi):= 1 \wedge \exp Q^{\beta}(x,\xi), 
$$
where, with the same reasoning leading to \eqref{approx5}, 
\be\label{approxbeta}
Q^{\beta}(x_k, \xi\kpo)=:Q_k^{\beta} \simeq R_k^{\beta} \sim \cN (-\ell^2 N^{1-\beta}, 2\ell^2 N^{1-\beta} S_{k}^N). 
\ee
Assuming that $S_0$ is finite, one can show that $S_k^N$ remains bounded (uniformly in $k$ and $N$). Therefore, if we look at the average acceptance probability, we have
\begin{align*}
\EE(1 \wedge e^{ Q^{\beta}(x_k,\xi_{k+1})})  \stackrel{\eqref{expmin1ex}}{=}  &
\Phi \left( \frac{-\ell^2 N^{(1-\beta)/2}}{\sqrt{2\ell^2 S_k^N}} \right)\\
 + &  e^{\ell^2 N^{1-\beta}(S_k^N -1)} \Phi \left( \frac{\ell^2 N^{(1-\beta)/2} (1-2S_k^N)}{\sqrt{2\ell^2 S_k^N}} \right).
\end{align*}
Therefore, if $\beta>1$ the acceptance probability tends to one as $N \ra \infty$, if $0 \leq \beta<1$ it tends to zero.

\section{Continuous Mapping Argument}\label{sec:contmapping}
In this section we explain the continuous mapping  arguments that the proofs of 
Theorem \ref{thm:weak conv of Skn} and Theorem \ref{thm:mainthm1}  rely on.  The continuous mapping argument that we use here is analogous to the one used in \cite{Matt:Pill:Stu:11,MR3024970}. The only difference is that the drift and diffusion coefficient of \eqref{informalSPDe} are time dependent. 

 Section \ref{subs:cma1} and Section \ref{subs:cma2}  contain the outline of the mapping argument that we will use in  the proof of Theorem \ref{thm:weak conv of Skn} and Theorem \ref{thm:mainthm1}, respectively.

\subsection{Continuous Mapping Argument for \eqref{ODE} (used in the  Proof of
 Theorem \ref{thm:weak conv of Skn})} \label{subs:cma1}
 Consider the chain  $S_k^N$, defined in \eqref{skn} and let $S\bn(t)$ and $\bar{S}\bn (t)$ be the continuous and piecewise constant interpolants of such a chain, respectively;  we recall that $S\bn(t)$  and $\bar{S}\bn(t)$ have been 
defined in \eqref{interpolantofsk}  and \eqref{piecewiseconstinterS}, respectively.
Decompose the chain $S_k^N$ into its drift and martingale part:
$$
S\kpo^N=S_k^N+\frac{1}{N}\al^N(x_k^N)+\frac{1}{\sqrt{N}}M_{k}^{2,N},
$$
where
\be\label{approximatedrifts}
\al^N(x_k^N):=N \EE_k \left[ S\kpo^N-S_k^N \right]
\ee
and 
\be\label{m2kn}
M_{k}^{2,N}:=\sqrt{N} \left[S\kpo^N- S_k^N - \frac{1}{N}\al^N(S_k^N)  \right].
\ee
We will show in Lemma \ref{lem:AlN-Al}
and Lemma \ref{lem:Alx-Aly} that $\al^N(x_k^N)$ converges to $A_{\ell}(S(t))$. \footnote{While the approximate drift $\al^N(x_k^N)$ of the chain $S_k^N$ depends only on $x_k^N$, the limiting drift $A_{\ell}$ depends only on $S(t)$. This is coherent with the fact that  $S_k^N$ depends only on $x_k^N$: in the limit, the  dependence of the drift on  $S_k^N$ appears explicitly. } 
Now a straightforward calculation (completely analogous to the one in  \cite[Appendix A]{Ottobre2016}) shows that 
$$
S\bn(t)=S_k^N+\int_{t_k}^t \al^N(\bar{x}^{(N)}(v)) dv+ \sqrt{N}(t-t_k)M_k^{2,N}, \quad \mbox{when } t_k\leq t <t\kpo \, ,
$$
where $\bar{x}^{(N)}$, the piecewise constant interpolant of the chain $\{x_k\}_k$,   is defined in \eqref{piecewiseconstinterx} below. Therefore 
$$
S\bn(t)=S_0^N+\int_{0}^t \al^N(\bar{x}\bn (v)) dv+\frac{1}{\sqrt{N}} 
\sum_{j=0}^{k-1}M_j^{2,N} +\sqrt{N}(t-t_k)M_k^{2,N}, 
\quad \mbox{for any }t \in [0,T].  
$$
Setting
\be\label{wN}
w^N(t):=\frac{1}{\sqrt{N}} \sum_{j=0}^{k-1}M_j^{2,N} +\sqrt{N}(t-t_k)M_k^{2,N},
\ee
we can rewrite the above as
\begin{align}
S\bn(t)&=S_0^N+ \int_{0}^t \al^N(\bar{x}\bn (v)) dv+ w^N(t)\nonumber\\
&=S_0^N+\int_{0}^t \al(S\bn(v)) dv+\hat{w}^N(t) \, ,\label{contmapwaSN}
\end{align}
where, for all $t \in [0,T]$,
\begin{align}
\hat{w}^N(t)&:= \int_0^t \left[ \al^N(\bar{x}\bn (v))-\al(S\bn(v)) \right] dv + w^N(t)\nonumber \\
&= \int_0^t \left[ \al^N(\bar{x}\bn (v))-\al(\bar{S}\bn(v)) \right] dv 
+\int_0^t \left[ \al (\bar{S}\bn (v))-\al(S\bn(v)) \right] dv + 
w^N(t). \label{whatN}
\end{align}
Equation \eqref{contmapwaSN} shows that $S\bn(t)=\mathcal{J}_2(S_0^N, \hat{w}^N)$, where $\mathcal{J}_2$ is the map defined in Theorem \ref{contofupsilon}. By the continuity of the map $\mathcal{J}_2$, if we show that  $\hat{w}^N$ converges weakly to zero in $C([0,T]; \R)$,  then $S\bn(t)$ converges weakly to the solution of the ODE \eqref{ODE}. The weak convergence of $\hat{w}^N$ to zero will be proved in Section \nolinebreak \ref{sec:proof of wekconvof Skn}.

 Now we outline the continuous mapping argument for the chain $x\kkn$ and in doing so we shall fix some more notation.

\smallskip

\subsection{Continuous Mapping Argument for \eqref{informalSPDe} (used in the Proof of Theorem \ref{thm:mainthm1})}\label{subs:cma2}
We now consider the chain that we are actually interested in, i.e. the chain $ \{x_k\}_k \subset \hs$, defined in \eqref{chainxcomponents}.  We act analogously to what we have done for the chain $S\kkn$. So we start by recalling the definition of the continuous interpolant $x\bn(t)$, equation \eqref{interpolant}, and we define the piecewise constant interpolant of the chain to be 
\be\label{piecewiseconstinterx}
\bar{x}\bn(t)=x_k^N \quad \mbox{for }\,\, t_k\leq t < t_{k+1}.
\ee
We also recall the notation $\Theta(x,S)$ for  the drift of equation \eqref{informalSPDe}, i.e.
\be\label{Theta}
\Theta(x, S)= F(x)\dl(S). \quad (x, S) \in \hs \times \R_+.
\ee
The drift-martingale decomposition of the chain  $x_k^N$  is as follows:
\be\label{driftmartdecompx}
 x\kpo^N=x_k^N+\frac{1}{N}\Theta^N(x_k^N)+\frac{1}{\sqrt{N}}M_{k}^{1,N}.
\ee
where $\Theta^N(x)$  is
\be\label{approximatedriftd}
\Theta^N(x_k^N):=N \EE_k \left[ x\kpo^N-x_k^N \right]
\ee
and 
\be\label{Mkn}
M_{k}^{1,N}:=\sqrt{N} \left[x\kpo^N- x_k^N - \frac{1}{N} \Theta^N(x_k^N)  \right].
\ee
Notice that $\Theta^N(x)$ is just a function of $x$; we will show (see Lemma \ref{lem:err2tozero} and \eqref{etahatN}) that the approximate drift $\Theta^N(x)$ converges to $\Theta(x,S)$, the drift of the SDE \eqref{informalSPDe}; that is, in the limit the dependence on $S$ becomes explicit (this should not surprie since, as already remarked, $S_k^N$ depends only on $x_k^N$).
Using again  \cite[Appendix A]{Ottobre2016} we 
obtain  
$$
x\bn(t)=x_k^N+\int_{t_k}^t \Theta^N(\bar{x}\bn (v)) dv+ 
\sqrt{N}(t-t_k)M_k^{1,N}, \quad \mbox{when } t_k\leq t <t\kpo
$$
and therefore, for all $t \in [0,T]$,
$$
x\bn(t)=x_0^N+\int_{0}^t \Theta^N(\bar{x}\bn (v)) dv+\frac{1}{\sqrt{N}} \sum_{j=1}^{k-1}M_j^{1,N} +\sqrt{N}(t-t_k)M_k^{1,N}.  
$$
Setting
\be\label{etaN}
\eta^N(t):=\frac{1}{\sqrt{N}} \sum_{j=0}^{k-1}M_j^{1,N} +\sqrt{N}(t-t_k)M_k^{1,N}, 
\quad \mbox{when } t_k\leq t <t\kpo, 
\ee
we can rewrite the above as
\begin{align}
x\bn(t)&=x_0^N+ \int_{0}^t \Theta^N(\bar{x}\bn (v)) dv+ \eta^N(t)\nonumber\\
&=x_0^N+\int_{0}^t \Theta(x\bn(v), S(v)) dv+\hat{\eta}^N(t) \, ,\label{contmapetahatx}
\end{align}
where, for all $t \in [0,T]$,
\begin{align}
\hat{\eta}^N(t)&:= \int_0^t \left[ \Theta^N(\bar{x}\bn (v))-\Theta(x\bn(v), S(v)) \right] dv + \eta^N(t)\nonumber \\
&= \int_0^t \left[ \Theta^N(\bar{x}\bn (v))-
\Theta(\bar{x}\bn(v), \bar{S}\bn (v)) \right] dv \nonumber\\
&+\int_0^t \left[ \Theta(\bar{x}\bn(v), \bar{S}\bn (v)) - \Theta({x}\bn(v), {S}\bn (v)) \right] dv \nonumber\\
&+\int_0^t \left[ \Theta ({x}\bn(v), {S}\bn (v))-\Theta({x}\bn(v), 
{S(v)}) \right] dv + \eta^N(t). \label{etahatN}
\end{align}
If we can prove that $\hat{\eta}^N(t)$ converges  weakly in $C([0,T]; \hs)$ to 
\be\label{limint}
\eta(t):= \int_0^t \Gl^{1/2}(S_v)dW_v,
\ee
where $W_v$ is a $\hs$-valued $\C_s$-Brownian motion, then 
\eqref{contmapetahatx} and 
the continuity of the map $\mathcal{J}_1$  allow to conclude that $x\bn(t)$ converges weakly in $C([0,T]; \hs)$ to $x(t)$,  solution of \eqref{xsolofSPDE}.
Such an   argument  is  the backbone of the proof of 
 Theorem 
\ref{thm:mainthm1}. The proof of Theorem 
\ref{thm:mainthm1} can be found in   Section \ref{sec:proofofmainthm}.

\section{Proof of Theorem \ref{thm:weak conv of Skn}}\label{sec:proof of wekconvof Skn}
\begin{proof}[Proof of Theorem \ref{thm:weak conv of Skn}]
Recall the definition of the map $\mathcal{J}_2$ given in Theorem
 \ref{contofupsilon} and observe   that thanks to \eqref{contmapwaSN}, 
 $$
 S\bn(t)=\mathcal{J}_2(S_0^N, \hat{w}^N(t)).
 $$
 Therefore proving the statement of Theorem \ref{thm:weak conv of Skn} amounts to proving that $\hat{w}^N(t)$ converges weakly to zero in 
$C([0,T]; \R)$. This is a consequence of the decomposition \eqref{whatN} together with  Lemma \ref{lem:noise2}, Lemma \ref{lem:AlN-Al} and Lemma \ref{lem:Alx-Aly} below.
\end{proof}
In the following $\EE_{x_0}$ denotes the expected value given $x_0\in \hsint$, the initial value of the chain. We recall once again that the initial value of the chain $x_k^N$ determines the initial value of the chain $S_k^N$.
\begin{lemma}\label{lem:noise2}
 Under  the assumptions of Theorem \ref{thm:weak conv of Skn},   the martingale difference array $w^N(t)$ defined in \eqref{wN} converges weakly to zero  in $C([0,T]; \R)$. 
\end{lemma}
\begin{lemma}\label{lem:AlN-Al}
 Under  the assumptions of Theorem \ref{thm:weak conv of Skn},
\be\label{err2s}
 \EE_{x_0}\int_0^T \lv \al^N(\bar{x}\bn (v))-\al(\bar{S}\bn(v)) \rv^2 \, dv  \lra 0 \quad \mbox{as} \quad N \ra \infty.
\ee
\end{lemma}
\begin{lemma}\label{lem:Alx-Aly}
 Under  the assumptions of Theorem \ref{thm:weak conv of Skn},  for every fixed $T>0$, 
\be\label{lim:Alx-Aly}
\EE_{x_0}\int_0^T \lv \al (\bar{S}\bn (v))-\al(S\bn(v)) \rv^2  \,dv \lra 0 \quad \mbox{as} \quad N \ra \infty.
\ee
\end{lemma}
Before proving the above lemmata, we state  Lemma \ref{lem:bound1},  which we will repeatedly use throughout this section and the next.  The proof of Lemma \ref{lem:noise2} can be found in  Section \ref{sec:annoises},  the proof of Lemma \ref{lem:AlN-Al} and Lemma \ref{lem:Alx-Aly} is the content of Section \ref{sec:andrifts}. 

\begin{lemma}\label{lem:bound1}
 Let the assumptions of Theorem \ref{thm:weak conv of Skn}  hold. Then for every $m\geq 0$ there exists a constant $\bar{c}=\bar{c}(m)$ such that
\be\label{boundonmomentsofx}
\EE_{x_0} \nors{x_k^N}^{m}  < \bar{c}\, ,
\ee
\be\label{boundonmomentsofsk}
 \EE_{x_0} (S_k^N)^m < \bar{c}\, 
\ee
and 
\be\label{lembound2ebounded}  
 \exo e^{\frac{c}{N} \|\zeta_k^N\|^2}< \bar{c}  \quad \mbox{for all }\, c>0. 
\ee
We recall that $\zeta\kkn$  has been defined in \eqref{defzetakn}. The constant $\bar{c}= \bar{c}(m)$ in the above bounds is independent of $N \in \N$ and of $0 \leq k \leq [TN]+1$ (but it depends on $m$).  
 \end{lemma}
\begin{proof}
See Appendix B.
\end{proof}
It is not trivial to prove Lemma \ref{lem:bound1} in the non-stationary regime that we are interested in. We make some more detailed remarks on this point in Remark \ref{rem:differ1}. 
\begin{lemma}\label{corollarylemma}Under the assumptions of Theorem \ref{thm:weak conv of Skn}, 
$$
\bar{S}^{(N)}(t) \rightarrow S(t) \quad \mbox{ in }L^p(\Omega), 
\mbox{ for every fixed } t>0 \mbox{ and any }p>0. 
$$ 
Moreover, 
$$
\exo \int_0^T \lv \bar{S}^{(N)}(t) - S(t) \rv^p dt \longrightarrow 0  \quad \mbox{ as } N \ra \infty, \mbox{ for all } p>0
$$
and 
$$
\exo \int_0^T \lv {S}^{(N)}(t) - S(t) \rv^p dt \longrightarrow 0  \quad \mbox{ as } N \ra \infty, \mbox{ for all } p>0\,.
$$
\end{lemma}
\begin{proof} Using Vitali's convergence theorem, the first statement  is a corollary of \eqref{boundonmomentsofsk} and
 Lemma \ref{lem:asconv}  (indeed $\bar{S}^{(N)}(t) \leq S_k^N+S_{k+1}^{N}$ and the right hand side has bounded moments of any order, so the sequence $\bar{S}^{(N)}(t)$ is uniformly integrable). As for the second statement, it can be obtained from the first by using again the bounded convergence theorem applied to the (deterministic) sequence $\exo \lv \bar{S}^{(N)}(t) - S(t) \rv^p$. Indeed such a sequence tends to zero and is bounded by a multiple of the function $\exo\left[ \lv \bar{S}^{(N)}(t)\lv^p + \rv S(t) \rv^p\right]$, which is bounded again thanks to \eqref{boundonmomentsofsk}. The last statement is obtained similarly and we don't detail the argument. 
This concludes the proof of the lemma.
\end{proof}

\subsection{Analysis of the Drift}\label{sec:andrifts}
Before starting the proof of Lemma \ref{lem:AlN-Al} we observe that because $\sum_{j=1}^N (\xi\kpo\jnn)^2$ has a Chi-squared distribution with $N$ degrees of freedom,  the following bound holds:
\be\label{stirling}
\EE\left[ \sum_{j=1}^N (\xi\kpo\jnn)^2\right]^m = 2^m \frac{\Gamma(m+N/2)}{\Gamma(N/2)}\les N^m,
\ee
by Stirling's formula for the Gamma function $\Gamma$.

\begin{proof}[Proof of Lemma \ref{lem:AlN-Al}]
Set 
\begin{align}
E_k^N&:=\al^N (x_k^N)- \al (S_k^N)\, \label{defEn}.
\end{align}
Then, recalling that for any $b \in \R_+$ we set $[b]=n$ if $n \leq b < n+1$ for some integer $n$, 
\begin{align}
\EE_{x_0} \int_0^T  \lv \al^N(\bar{x}\bn (v))-\al(\bar{S}\bn(v)) \rv^2 \, dv
&=  \EE_{x_0} \frac{1}{N} \sum_{k=0}^{[TN]} \lv E_k^N\rv^2  
\label{driftbigs}\\
&+   \left(T-\frac{[TN]}{N}\right) \,  \EE_{x_0} \lv E_{[TN]}^N\rv^2  .\label{driftsmalls}
\end{align}
From the above equality and observing that $\lv T-\frac{[TN]}{N} \rv < 1/N$, it is clear that  in order to show the limit \eqref{err2s} it is sufficient to prove that
$$
\EE_{x_0}\lv E_k^N \rv^2 \stackrel{N\ra \infty}{\lra} 0\,,
\quad \mbox{uniformly over }\,\, 0\leq k\leq [NT].
$$
To this end, we  write $\al(S\kkn)=\EE_k[(1 \wedge e^{Z_{\ell,k}})(-2 Z_{\ell,k})]$ (which follows from \eqref{useinlemma1})
and use \eqref{useinlemma}  and \eqref{approximatedrifts}, obtaining
 \begin{align}\label{decompee}
 E\kkn=\EE_k\left[(1 \wedge e^Q)(-2R) \right]-\al(S\kkn)+ \EE_k\hat{r}^N=E_{1,k}^N+E_{2,k}^N+\EE_k\hat{r}^N,
 \end{align}
 where
 \begin{align}
  E_{1,k}^N  & := \EE_k
 \left[((1 \wedge e^Q)-(1 \wedge e^R))(-2R) 
 \right],   \nonumber
 \end{align} 
 \begin{align*}
 E_{2,k}^N  & :=  \EE_k \left[ (1 \wedge e^R)(-2R) - (1 \wedge e^{Z_{\ell,k}})(-2 Z_{\ell,k}) \right] ,
\end{align*}
and $\hat{r}^N$ is defined in \eqref{rhatNdef}.
Observe that from \eqref{Q=R+rN} and \eqref{stirling} we have 
\be\label{ekR}
\EE_k\lv R\rv^{2p}
\les 1+\frac{\|\zeta_k^N\|^{2p}}{N^p}, \qquad p\geq 1, 
\ee
as, given $x_k$, the sum $\sum_{i=1}^N \zeta_k^{i,N} \xi_k^{i,N}$ is Gaussian with mean zero and variance $\|\zeta_k^N\|^2.$
 From the definition of $\zeta_k^N$, equation (\ref{defzetakn}),  we have
$$
\frac{\|\zeta_k^N\|^2}{N}\les \frac{1}{N}\sum_{i=1}^N \frac{\lv x\kn \rv^2}
{\lambda_i^2} +\frac{1}{N} \|\C^{1/2}\nabla \Psi^N(x_k^N)\|^2 = S\kkn+ \frac{1}{N} \|\C^{1/2}\nabla \Psi^N (x_k^N)\|^2.
$$
By acting as in \cite[page 915]{Matt:Pill:Stu:11} we obtain
\begin{align}
\|\C^{1/2}\nabla\Psi^N(x)\| &\les \|x\|_s \vee \|x\|_s^{\varsigma}\label{c12dpsi111}\\
& \les (1+ \nors{x}) \label{c12dpsi},
\end{align}
hence
\be\label{genp}
\frac{\|\zeta_k^N\|^{2p}}{N^p}\les (S_k^N)^p +\frac{1}{N^p} (1+ \nors{x_k^N}^{2p}), \quad p\geq 1.
\ee
Combining \eqref{ekR} and \eqref{genp} then gives
\be\label{rknp}
\EE_k \lv R\rv^{2p} \les 1+
(S_k^N)^p +\frac{1}{N^p} (1+ \nors{x_k^N}^{2p}), \quad p\geq 1.
\ee
Therefore, using \eqref{rknp} (with $p=1$), \eqref{boundonmomentsofx}  and (\ref{boundonmomentsofsk}), we obtain
\be\label{ekrfinite}
\exo \lv R\rv^2=\exo \EE_k\lv R\rv^2 \les 1+\frac{\exo \|\zeta_k^N\|^2}{N}\les 1+\exo S_k^N+\frac{\exo \nors{x_k^N}^2}{N} < \infty.
\ee
Using the Lipshitzianity of the function $1\wedge e^x$ and \eqref{decompQvera},   we have
\begin{align}\label{uselater}
\lv E_{1,k}^N  \rv & \les \EE_k \lv r^N R\rv 
\leq \left(\EE_k \lv r^N\rv^2 \right)^{1/2} \left(\EE_k(R)^2 \right)^{1/2}\,.
 \end{align}
By \eqref{ekrfinite} and \eqref{exprn} we then conclude 
\be\label{EEE}
 \exo\lv E_{1,k}^N \rv^2 \leq \frac{1}{N^2}\exo\left( S_k^N+ \frac{\nors{x_k^N}^2}{N}\right)\lra 0, 
\ee
thanks to Lemma \ref{lem:bound1}. 
As for the term $E_{2,k}^N$, 
we use the lipshitzianity of the function $(1 \wedge e^x) (-2x)$  to conclude
\be \label{eek2}
\exo \lv E_{2,k}^N \rv^2 \les \exo \EE_k \lv R- Z_{\ell,k}\rv^2 \stackrel{\eqref{hmrz}}{\les} \frac{1+\exo\nors{x_k}^2}{N} \lra 0.
\ee
  Finally, to estimate $\hat{r}^N$ (defined in \eqref{rhatNdef}), we use    the independence, given $x_k$, of $\Psi(x_k)$ from $\xi_{k+1}$:
\begin{align}
\lv  \EE_k\hat{r}^N\rv^{p} & \leq \left(\EE_k\lv \hat{r}^N\rv^{2}\right)^{p/2} \nonumber \\
& \leq \frac{1}{{N}^{p/2} } \left(  \EE_k \lv \sum_{j=1}^N 
\left[\C_N^{1/2} \nabla \Psi \right]^j \xi\kpo\jnn \rv ^2 \right)^{p/2} \nonumber\\
& = \frac{1}{{N}^{p/2}}   \left(   \sum_{j=1}^N 
\EE_k \lv \left[\C_N^{1/2} \nabla \Psi \right]^j \rv^2  \lv \xi\kpo\jnn \rv ^2 \right)^{p/2}  \nonumber\\
&=  \frac{1}{N^{p/2}} \| \C^{1/2} \nabla \Psi\|^{p}, \label{MM}
\end{align}
where in the above $\left[\C_N^{1/2} \nabla \Psi \right]^j$ denotes the $j$-th component of 
$\C_N^{1/2} \nabla \Psi $. 
Using \eqref{c12dpsi} we have 
\be\label{estrhat}
\lv  \EE_k\hat{r}^N\rv^{p} \leq \left(\EE_k\lv \hat{r}^N\rv^{2}\right)^{p/2} \leq \frac{1+\nors{x_k^N}^p}{N^{p/2}} \quad \mbox{for all } p \geq 1.   
\ee
Hence, \eqref{boundonmomentsofx} gives
\be\label{estrhat1}
\exo \lv \EE_k  \hat{r}^N\rv^{p} \leq \exo \left(\EE_k\lv \hat{r}^N\rv^{2}\right)^{p/2}\leq \frac{1}{N^{p/2}} \quad \mbox{for all } p \geq 1.   
\ee

 This  concludes the proof.
\end{proof}

\begin{proof}[Proof of Lemma \ref{lem:Alx-Aly}] By the Lipshitzianity of $\al$,
$$
\lv \al (\bar{S}\bn(v))  - \al (S\bn(v)) \rv^2 \les  \lv \bar{S}\bn(v) -S\bn(v) \rv^2.
$$
The statement is now a consequence of \eqref{starrina}.

\end{proof}

\subsection{Analysis of the Noise}\label{sec:annoises}
\begin{proof}[Proof of Lemma \ref{lem:noise2}] By the martingale central limit theorem, all we need to prove is that
$$
\frac{1}{N}\sum_{j=1}^{[TN]}\exo \lv  M_j^{2,N} \rv^2 \ra 0 \quad \mbox{as } N \ra \infty.
$$
From the definition of $M_j^{2,N}$,  equation (\ref{m2kn}), 
\begin{align*}
\exo\frac{\lv  M_j^{2,N} \rv^2}{N}& =\exo\lv S\kpo^N-S\kkn-\EE_k[S\kpo^N-S\kkn]\rv^2\\
& \les \exo\lv S\kpo^N-S\kkn \rv^2. 
\end{align*}
With the same calculation as in \eqref{useinlemma}, 
\be\label{red}
S\kpo^N-S\kkn =\frac{\gamma\kpo(-2R)}{N} + \frac{1}{N}\hat{r}^N.
\ee
Therefore, using (\ref{ekrfinite}), \eqref{estrhat1} and $\gamma\kpo\leq 1$, 
\be\label{diffs1s}
\exo\lv S\kpo^N-S\kkn\rv^2 \les \frac{1}{N^2} \exo \EE_k\lv R\rv^2 +  \frac{1}{N^3}\les \frac{1}{N^2}.
\ee
The above implies the bound
$$
\exo\frac{\lv  M_j^{2,N} \rv^2}{N} \les \frac{1}{N^2}.
$$
We can therefore  conclude
$$
\frac{1}{N}\sum_{j=1}^{[TN]}\exo \lv  M_j^{2,N} \rv^2\les \frac{T}{N} \lra 0.
$$
\end{proof}

\section{Proof of Theorem \ref{thm:mainthm1}} \label{sec:proofofmainthm}
Before starting the proof of Theorem \ref{thm:mainthm1}, we state Lemma \ref{lem:Wassbound} below.
We recall the definition of Wasserstain distance between two random variables $X$ and $Y$:
\be\label{def:wassdist}
Wass(X,Y) := \sup_{f \in \mbox{Lip}_1}\EE(f(X)-f(Y))\, ,
\ee
where Lip$_1$ denotes the class of Lipshitz  functions with Lipshitz constant equal to one.  
Notice that from the definition, 
\be\label{wasspractical}
Wass(X,Y)\leq \EE\lv X-Y\rv \, .
\ee
In the next Lemma \eqref{lem:Wassbound} we refer to the Wasserstein distance relative to the marginal $\EE_k$. 
\begin{lemma}\label{lem:Wassbound}
Let the assumptions of Theorem \ref{thm:weak conv of Skn}  hold. Recalling the definitions of $R$, $R^i$,
$G$ and $Z_{\ell,k}$, \eqref{Q=R+rN}, \eqref{defRi}, \eqref{defGi} and \eqref{approx4} respectively, we have
\be\label{wassrir}
Wass(R, R^i) \leq \EE_k \lv R-R^i \rv \les\frac{1+ \lv \zeta\kn \rv}{\sqrt{N}},
\ee
\be\label{wassrigi}
Wass(R, G) \leq \EE_k \lv R-G \rv \les\frac{1}{\sqrt{N}}
\ee
and
\be\label{wassrzl}
Wass(G, Z_{\ell,k}) \leq \EE_k \lv G-Z_{\ell,k} \rv \les \frac{{1+\nors{x_k}}}{\sqrt{N}}.
\ee
Therefore, 
\be\label{wassrzlk}
Wass(R, Z_{\ell,k})\leq \EE_k\lv R-Z_{\ell,k}\rv \les \frac{{1+\nors{x_k}}}{\sqrt{N}}.
\ee
\end{lemma}
\begin{proof}
See Appendix B.
\end{proof}

\begin{proof}[Proof of Theorem \ref{thm:mainthm1}]
If $\mathcal{J}_1$ is the map defined in Theorem \ref{contofupsilon}, then \eqref{contmapetahatx} means that
$$
x\bn(t)=\mathcal{J}_1(x_0^N,\hat{\eta}^N(t)).
$$
From the continuity of $\mathcal{J}_1$,  in order to prove that $x\bn(t)\stackrel{d}{\lra}x(t)$, we just need to prove that 
$\hat{\eta}^N(t)\stackrel{d}{\lra} \eta(t)$, where $\eta(t)$ is the stochastic integral defined in \eqref{limint}.  The weak convergence $\hat{\eta}^N(t)\stackrel{d}{\lra} \eta(t)$ follows from Lemma \ref{lem:noise1},   Lemma \ref{lem:err2tozero},  Lemma \ref{lem:err3tozero}, Lemma \ref{lastbit} and the decomposition \eqref{etahatN}. 
\end{proof}

\begin{lemma}\label{lem:noise1}
 Let the assumptions of Theorem \ref{thm:mainthm1} hold. Then the interpolated martingale difference array $\mathfrak{\eta}^N(t)$ defined in \eqref{etaN} converges weakly in $C([0,T]; \hs)$ to the stochastic integral  $\eta(t)$, equation \eqref{limint}. 
\end{lemma}

\begin{lemma} \label{lem:err2tozero}
Let the assumptions of Theorem \ref{thm:mainthm1} hold. Then for every fixed $T>0$, 
\be\label{err2x}
 \EE_{x_0}\int_0^T \nors{ \Theta^N(\bar{x}\bn (v))-\Theta(\bar{x}\bn(v), \bar{S}\bn (v))}^2  \, dv  \lra 0 \quad \mbox{as} \quad N \ra \infty.
\ee
\end{lemma}

\begin{lemma} \label{lem:err3tozero}
Let the assumptions of Theorem \ref{thm:mainthm1} hold. Then for any fixed $T>0$
\be
\EE_{x_0}\int_0^T \| \Theta (\bar{x}\bn (v), 
\bar{S}\bn (v))-\Theta(x\bn(v), S\bn (v)) \|_s^2  \,dv \lra 0 \quad \mbox{as} \quad N \ra \infty.
\ee
\end{lemma}

\begin{lemma}\label{lastbit}
Let the assumptions of Theorem \ref{thm:mainthm1} hold. Then for any fixed $T>0$
$$
\EE_{x_0}\int_0^T \nors{ 
\Theta ({x}\bn(v), {S}\bn (v))-\Theta({x}\bn(v), 
{S(v)})}^2 dv \lra 0 \quad \mbox{as} \quad N \ra \infty.
$$
\end{lemma}
We will prove Lemma \ref{lem:noise1} in  Section 
\ref{sec:noisex} and Lemma \ref{lem:err2tozero},  Lemma \ref{lem:err3tozero}, Lemma \ref{lastbit} in Section \ref{sec:driftx}. 

\subsection{Analysis of the Drift}\label{sec:driftx}
In what follows we will need some preliminary estimates, which we list in Lemma \ref{lem:bound2} below. 

\begin{lemma}\label{lem:bound2}
Under the assumptions of Theorem \ref{thm:mainthm1}, the following holds: 
\begin{description}
 \item[i)]  Let $Y$ be a positive random variable such that $\exo \lv Y \rv^q < \infty$ for all $q \geq 1$ (should $Y$ depend on $k$ and $N$, all the moments are assumed to be bounded independently of $k$ and $N$).  Then, uniformly over $0 \leq k \leq [TN]+1$, 
\be\label{lembound22}
\limsup_{N\ra \infty}\exo \left[ Y \sum_{i=1}^N i^{2s} \lambda_i^2 \lv \zeta\kn\rv^{p} \right] < \infty, 
\qquad \mbox{for all } \, p\geq 0.
\ee
 \item[ii)] Moreover, 
\be\label{lembound26}
\frac{1}{N}\sum_{k=0}^{[TN]}\exo \sum_{i=1}^N i^{2s} \lambda_i^2\frac {\lv \zeta\kn\rv^{p}}{(N S_k^N)^{\alpha}}  \stackrel{N\ra \infty}{\lra} 0, 
\quad \mbox{for all } \, p\geq \frac{2\alpha}{\varsigma} > 0,
\ee
where we recall that the constant $1/2 \leq  \varsigma<1$ is the one appearing in Assumption \ref{ass:1}.
\item[iii)] Finally, 
\be\label{lembound24}
 \mathbb{E}_k \frac{1}{(1+ \lv R \rv\sqrt{N})^2}\les \left( 1+\nors{x_k^N}^2\right)
 \frac{1}{(N S_k^N)^{1/4}} +
 \frac{1}{\sqrt{N S_k^N}} \, .
\ee
\end{description}
\end{lemma}
\begin{proof}
See Appendix B. \end{proof}

\begin{remark}[On Lemma \ref{lem:bound2} and Lemma \ref{lem:bound1}] \label{rem:differ1}
\textup{The proofs of Lemma \ref{lem:bound2} and Lemma \ref{lem:bound1} bring up some of the main differences between the stationary  and the non-stationary case, so it is worth making some comments.}
\begin{itemize}
\item
\textup{ If we start the chain in stationarity, i.e. $x_0^N\sim \pi^N$, where $\pi^N$ has been defined in \eqref{piN}, then $x_k^N \sim \pi^N$ for every $k \geq 0$. As already observed in the introduction, $\pi^N$ is absolutely continuous with respect to a Gaussian measure; because all the almost sure properties are preserved under this change of measure, in the stationary regime most of the estimates of interest need to be shown only for $x \sim \pi_0$.  
If  $x \sim \pi_0$ then  $x^N= \sum_{i=1}^N \lambda_i \rho^i \phi_i$, where $\rho^i$ are i.i.d.~ $\cN(0,1)$.
 Therefore, recalling \eqref{defzetakn} (see also \eqref{z=x+psi}), one gets
\be\label{shoe}
\lv \zeta\inn \rv^p \les \lv \rho^i\rv^p+ \|x\|_s^p \,.
\ee
With this observation it is then clear that in stationarity the bounds \eqref{boundonmomentsofx} and \eqref{boundonmomentsofsk} are trivially true, 
and  \eqref{lembound22}  follows easily  from  \eqref{shoe} and \eqref{boundonmomentsofx}. 
For the same reason, in the stationary case the estimate \eqref{lembound2ebounded} is  a consequence of Fernique's Theorem,  see \cite[page 916]{Matt:Pill:Stu:11}.  With a similar reasoning and by \eqref{skstat} one can see that also \eqref{lembound26} holds in stationarity.}
\item {\textup{ In our case, i.e. out of stationarity,  proving the bounds of Lemma \ref{lem:bound2} and Lemma \ref{lem:bound1}  requires a bit of an argument. In particular, the reason why the limit \eqref{lembound26} holds can be understood at least heuristically observing that $S_k^N$ converges to $S(t)$ (i.e. to a finite number, which is strictly positive under our assumptions and it converges to 1 if we work in stationarity, see \eqref{skstat}). Combining this observation with  \eqref{lembound22} gives, heuristically, \eqref{lembound26}.  
 \item On a minor note, we point out that the limit \eqref{lembound26} might not hold for $k=0$ if we were to allow $S_0=0$. Indeed, suppose again for simplicity that $\Psi=0$.  If $S_0=0$ and the sequence of partial sums $\sum_{j=1}^N \frac{\lv x_0^{j,N} \rv^2}{\lambda_j^2}$ is convergent  then the quantity on the LHS of \eqref{lembound26} is in general only bounded.  (However if we were to extend the proof to the case $S_0=0$ we would not need \eqref{lembound26} to hold at $k=0$).
}} 
\end{itemize}
\hfill{$\Box$}
\end{remark}

\begin{proof}[Proof of Lemma \ref{lem:err2tozero}]
Set 
\begin{align}
e_k^N&:=\Theta^N (x_k^N)- \Theta (x_k^N, S\kkn) =
N\EE_k[x\kpo^N-x_k^N] - \Theta (x_k^N, S\kkn)\label{defen}
\end{align}
 Then 
\begin{align}
\EE_{x_0} \int_0^T  \| \Theta^N(\bar{x}\bn (v))-\Theta(\bar{x}\bn(v), \bar{S}\bn (v)) \|_s^2  \, dv
&=  \EE_{x_0} \frac{1}{N} \sum_{k=0}^{[TN]}  \nors{e_k^N}^2
\label{driftbig}\\
&+   \left(T-\frac{[TN]}{N}\right) \,  \EE_{x_0} \nors{e_{[TN]}^N}^2.\label{driftsmall}
\end{align}
If $e\kn$ is the $i-$th component of $e_k^N$, the sum on the RHS of \eqref{driftbig} may be rewritten as
$$
 \EE_{x_0} \frac{1}{N} \sum_{k=0}^{[TN]} \sum_{i=1}^N i^{2s}
\lv e\kn\rv^2 .
$$
The statement now follows from Lemma \ref{lem:stimaEkn} below.  
\end{proof}

\begin{proof}[Proof of Lemma \ref{lem:err3tozero}]
From \eqref{Theta} we  have
\begin{align}
\nors{\Theta(\bar{x}\bn(v), \bar{S}\bn(v))  - \Theta(x\bn(v), S\bn(v))}^2 & \les 
\lv\dl (\bar{S}\bn (v)) \rv^2 \nors{F(\bar{x}\bn (v))-F(x\bn (v))}^2 \nonumber\\
& +\nors{F(x\bn (v)) }^2
\lv \dl(\bar{S}\bn) - \dl(S\bn)\rv^2\nonumber\\
& \les  \nors{\bar{x}\bn (v)-x\bn (v)}^2 \nonumber\\
& + (1+ \nors{x\bn (v)}^2) \lv \bar{S}\bn (v) -  S\bn (v) \rv^2,\label{comb1}
\end{align}
having used the boundedness and Lipshitzianity of $\dl$,  the Lipshitzianity of $F$ 
(Lemma \ref{lem:lipschitz+taylor} and Lemma \ref{lem:propofDandGamma}, respectively) and the bound 
\be\label{C}
\| F(z)\|_s^2 \les 1+\nors{z}^2. 
\ee 
The above bound is a consequence of  Assumption \ref{ass:1} and 
$$
\|\C \nabla \Psi (z)\|_s^2 = \sum_{i=1}^{\infty} i^{2s}\lambda_i^2 \lv (\nabla\Psi(z))^i\rv^2\leq \| \nabla\Psi(z)\|_{-s}^2.
$$
Moreover, if $t_k \leq  v \leq t_{k+1}$ then from the definition \eqref{interpolant},  we have
\begin{align}
\exo \nors{\bar{x}\bn(v)-x\bn(v)}^2 &= \exo \nors{(Nv-k)(x\kpo-x_k)}^2  \nonumber\\
&\stackrel{\eqref{chainxcomponents}}{\les} \frac{1}{N}\EE \nors{\C_N^{1/2}\xi^N_{k+1}}^2 \stackrel{\eqref{c1/2xi}}{\les} \frac{1}{N}\, .
\label{comb2}
\end{align}

The statement of the lemma is a consequence of  \eqref{comb1}, \eqref{comb2},  \eqref{boundonmomentsofx}, \eqref{starprime}  and \eqref{v2}.
\end{proof}

\begin{proof}[Proof of Lemma \ref{lastbit}] Analogous to the proof of 
Lemma \ref{lem:err3tozero},  so we only sketch it. 
\begin{align*}
\exo \nors{\Theta(x\bn(t), S\bn(t))-\Theta(x\bn(t), S(t))}^2\les \exo
\nors{F(x\bn (t))}^2 \lv \dl (S\bn(t))-\dl(S(t)) \rv^2\,.
\end{align*}
Now the RHS goes to zero thanks to the Lipshitzianity  of $\dl$,$\eqref{C}$,  Lemma \ref{lem:bound1} and Lemma \ref{corollarylemma}.
\end{proof}

\begin{lemma}\label{lem:stimaEkn}
Let the assumptions of Theorem \ref{thm:mainthm1} hold and recall that $e\kn$ is the $i$-th component of $e_k^N$, defined in \eqref{defen}. 
Then, 
$$
 \EE_{x_0} \frac{1}{N} \sum_{k=0}^{[TN]} \nors{e_k^N}^2 = \EE_{x_0} \frac{1}{N} \sum_{k=0}^{[TN]} \sum_{i=1}^N i^{2s}
\lv e\kn\rv^2  \lra 0 \, \quad \mbox{as } N \ra \infty.
$$
\end{lemma}
\begin{proof}[Proof of Lemma \ref{lem:stimaEkn}] This proof is partly analogous to the proofs of \cite[Lemma 5.5-Lemma 5.11]{Matt:Pill:Stu:11}. The main difference is that here we deal with  time dependent coefficients. The proof will only be  detailed  when it differs from \cite{Matt:Pill:Stu:11}; where it does not we will provide fewer details.  

From the definition of $\Theta$, equation  \eqref{Theta}, the $i$-th component of $\Theta$ calculated at $(x_k^N, S_k^N)$ is
\be\label{exprdldiffcoeff}
\Theta^i(x_k^N, S_k^N)= - \lambda_i \zeta\kn \dl(S_k^N) = - 2 \ell^2 \lambda_i \zeta\kn  \EE e^{Z_{\ell,k}} {\bf 1}_{\{ Z_{\ell,k}<0\}} \, .
\ee
where the second equality is a consequence of \eqref{ecarexp} and \eqref{approx4} .
Therefore the $i$-th component of $e_k^N$ is
\begin{align*}
e\kn &=\sqrt{2N \ell^2} \,  \lambda_i \EE_k^{\xi} \left[  \left( 1 \wedge e^{Q_k} \right) \xi_{k+1}^{i,N}   \right] - \Theta^i(x_k^N, S_k^N)\\
& =  \sqrt{2N \ell^2} \,  \lambda_i \EE_k^{\xi} \left[  \left( 1 \wedge e^{Q_k} \right) \xi_{k+1}^{i,N}   \right]+ \lambda_i \zeta\kn 
\dl(S_k^N)\, .
\end{align*}
Following the reasoning of Section  \ref{sec:heurdiffcoeffforchainx}, we decompose $e\kn$ as  follows:
$$
e\kn=\sqrt{2N \ell^2} \,  \lambda_i \EE_k^{\xi} \left[\left( 1 \wedge e^{R^i- \sqrt{\frac{2\ell^2}{N}} \zeta\kn \xi\kpon   } \right) \xi_{k+1}^{i,N} \right] - \Theta^i(x_k^N, S_k^N)
  + e_{1,k}^{i,N}+e_{2,k}^{i,N} \, ,  \quad \mbox{where}
$$
\begin{align}
e_{1,k}^{i,N} &:= \sqrt{2N \ell^2} \,  \lambda_i      
\EE_k^{\xi} \left[ \left( \left( 1 \wedge e^{Q_k} \right) - \left( 1 \wedge e^R \right)\right) \xi_{k+1}\inn   \right]      \label{ek1n}\\
e_{2,k}^{i,N} &:= \sqrt{2N \ell^2} \,  \lambda_i      
\EE_k^{\xi} \left[ \left( \left( 1 \wedge e^{R} \right) - 
\left( 1 \wedge e^{R^i- \sqrt{\frac{2\ell^2}{N}} \zeta\kn \xi\kpon   } \right)\right) \xi_{k+1}\inn   \right]      \label{ek2n}\, .
\end{align}
 We now use equality \eqref{eqclar},  leading to:
\begin{align*}
e\kn&= -2\ell^2 \lambda_i \zeta\kn \EE_k^{\xi^i_{-}} 
e^{R^i} \Phi\left( \frac{-R^i}{\sqrt{\frac{2\ell^2}{N}}
\lv \zeta_k^{i,N} \rv} \right)-\Theta^i (x_k^N, S_k^N)+ e_{1,k}^{i,N}+e_{2,k}^{i,N}+e_{3,k}^{i,N},
\end{align*}
where
\begin{align}
e_{3,k}^{i,N}&:= -2\ell^2 \lambda_i \zeta\kn \left( e^{\frac{\ell^2}{N} \lv \zeta\kn \rv^2} -1  \right) \EE_k^{\xi^i_{-}}
e^{R^i}
\Phi\left(  \frac{-R^i}{\sqrt{\frac{2\ell^2}{N}}
 \lv \zeta_k^{i,N} \rv } - \sqrt{\frac{2\ell^2}{N}} \lv \zeta_k^{i,N} \rv \right)\nonumber\\
 & -2\ell^2 \lambda_i \zeta\kn  \EE_k^{\xi^i_{-}}
e^{R^i}  \left[ \Phi\left(  \frac{-R^i}{\sqrt{\frac{2\ell^2}{N}}
 \lv \zeta_k^{i,N} \rv } - \sqrt{\frac{2\ell^2}{N}} \lv \zeta_k^{i,N} \rv \right) - 
\Phi\left(  \frac{-R^i}{\sqrt{\frac{2\ell^2}{N}}
 \lv \zeta_k^{i,N} \rv }  \right)  \right].  \label{ek3n} 
\end{align}
Finally, by setting
\begin{align}
e_{4,k}^{i,N}&:=  -2\ell^2 \lambda_i \zeta\kn  \EE_k^{\xi}\left[
e^{R^i}  \Phi\left(  \frac{-R^i}{\sqrt{\frac{2\ell^2}{N}} \lv \zeta_k^{i,N} \rv }  \right)  - 
e^{R^i} {\bf 1}_{\{R^i<0\}}  \right] ,  \label{ek4n} \\
e_{5,k}^{i,N}&:=  -2\ell^2 \lambda_i \zeta\kn
\EE_k^{\xi}  \left[ e^{R^i} {\bf 1}_{\{R^i<0\}}
- e^{Z_{\ell,k}} {\bf 1}_{\{ Z_{\ell,k}<0\}} \label{ek5n}
\right]\,,
\end{align}
 and using \eqref{exprdldiffcoeff},  we obtain
\be\label{6sum}
e\kn=  \sum_{h=1}^5 e_{h,k}^{i,N} \,.
\ee
Now that we have the above  decomposition,  we need to find bounds on each of the $e_{h,k}^{i,N}$'s, 
$h=1 \dd 5$,  which is what we shall do next. 
\begin{description}
\item[$\bullet\, e_{1,k}^{i,N}$ and $e_{2,k}^{i,N}$:] The bounds on $e_{1,k}^{i,N}$ and $e_{2,k}^{i,N}$ are straightforward:
\be \label{boundone1}
\lv  e_{1,k}^{i,N} \rv \les \frac{\lambda_i}{\sqrt{N}} \quad \mbox{and} \quad 
\lv  e_{2,k}^{i,N} \rv \les \frac{\lambda_i}{\sqrt{N}}\, ,
\ee
The first estimate is a consequence of \eqref{Q=R+rN}, \eqref{decompQvera}, \eqref{exprn} and the Lipshitzianity of the function
$f(x)=1 \wedge e^x$; for the second we used  definition  \eqref{defRi}.  
\smallskip
\noindent

\item[$\bullet\, e_{3,k}^{i,N}$:] To study $e_{3,k}^{i,N}$, we set
\be\label{mod11}
{e}_{3,k}^{i,N}: = \bar{e}_{3,k}^{i,N} + \tilde{e}_{3,k}^{i,N} \, ,
\ee
 with
\begin{align*}
\bar{e}_{3,k}^{i,N} & := -2\ell^2 \lambda_i \zeta\kn \left( e^{\frac{\ell^2}{N} \lv \zeta\kn \rv^2} -1  \right) \EE_k^{\xi^i_{-}}
e^{R^i}
\Phi\left(  \frac{-R^i}{\sqrt{\frac{2\ell^2}{N}}
 \lv \zeta_k^{i,N} \rv } - \sqrt{\frac{2\ell^2}{N}} \lv \zeta_k^{i,N} \rv \right)\nonumber\\
\tilde{e}_{3,k}^{i,N} := & -2\ell^2 \lambda_i \zeta\kn  \EE_k^{\xi^i_{-}}
e^{R^i}  \left[ \Phi\left(  \frac{-R^i}{\sqrt{\frac{2\ell^2}{N}}
 \lv \zeta_k^{i,N} \rv } - \sqrt{\frac{2\ell^2}{N}} \lv \zeta_k^{i,N} \rv \right) - 
\Phi\left(  \frac{-R^i}{\sqrt{\frac{2\ell^2}{N}}
 \lv \zeta_k^{i,N} \rv }  \right)  \right].
\end{align*}
To estimate $\tilde{e}_{3,k}^{i,N}$, we use  the boundedness and Lipshitzianity of $\Phi$ together with
\be\label{EeRi}
\EE_k^{\xi}e^{R^i} \les e^{\frac{\ell^2}{N} \| \zeta_k^N\|^2},
\ee
see \cite[(5.20)]{Matt:Pill:Stu:11}. We therefore obtain
\be\label{bounde3}
\lv  \tilde{e}_{3,k}^{i,N} \rv \les  \lambda_i  \frac{\lv \zeta\kn\rv^2}{\sqrt{N}} 
e^{\frac{\ell^2}{N} \| \zeta_k^N\|^2} \, .
\ee
The term $\bar{e}_{3,k}^{i,N}$ will be studied separately later. 
\smallskip
\noindent

\item[$\bullet\, e_{4,k}^{i,N}$:]  We act as in the proof of \cite[Lemma 5.7-Lemma 5.9]{Matt:Pill:Stu:11} and obtain 
\be\label{bounde4}
\lv  e_{4,k}^{i,N} \rv \les \lambda_i \lv\zeta\kn  \rv e^{\|\zeta_k^N\|^2/N} \left(1+ \lv \zeta\kn \rv\right) 
\left( \mathbb{E}_k  \frac{1}{(1+ \lv R \rv \sqrt{N})^2} \right)^{1/4}.
\ee

 \smallskip
\noindent

\item[$\bullet\, e_{5,k}^{i,N}$:]  Let $g(x):=e^x {\bf 1}_{\{ x<0  \}}$;  using the same argument as in  \cite[page 923]{Matt:Pill:Stu:11}, if $X$ and $Y$ are two random variables such that one of them has a density with respect to the Lebesgue measure and such a density is bounded by $M$, then 
\be\label{labM}
\lv \EE g(X) - \EE g(Y) \rv \les \sqrt{ M \,Wass(X,Y)}.
\ee
Such a result is applicable to $ R^i$ and $Z_{\ell,k}$ as $Z_{\ell,k}$ is (conditionally) Gaussian with variance $S_k^N$.  Therefore using \eqref{wassrir}, \eqref{wassrzlk} and \eqref{labM} with $M= 1/\sqrt{2 \pi \,S_k^N}$, we have 
\begin{align*}
\lv \EE_k^{\xi}e^{R^i} {\bf 1}_{\{R^i<0\}}  -  
 \EE e^{Z_{\ell,k}} {\bf 1}_{\{ Z_{\ell,k}<0\}} \rv
& \les  \frac{1}{\left( S_k^N\right)^{1/4}} \sqrt{Wass (R^i, Z_{\ell,k})} \\
& \les \frac{1}{N^{1/4}} \frac{1}{\left( S_k^N\right)^{1/4}}
\left[\left(1+ \nors{x_k^N}\right)^{1/2}+ \sqrt{1+ \lv \zeta\kn \rv}\right]
\end{align*}
The above, together with \eqref{ek5n}, implies
\be\label{bounde6}
\lv e_{5,k}^{i,N} \rv \les \frac{\lambda_i \lv \zeta\kn \rv}{\left( S_k^N\right)^{1/4}}
\frac{1}{N^{1/4}}
\left[\left(1+ \nors{x_k^N}\right)^{1/2}+ \sqrt{1+ \lv \zeta\kn \rv}\right]
\ee
\end{description}
From   the bounds \eqref{boundone1}, \eqref{bounde3},
 \eqref{bounde4} and \eqref{bounde6},  we get
\begin{align*}
\exo \!\! \sum_{h=1, h\neq 3}^{5} \sum_{i=1}^N i^{2s}  \lv e_{h,k}^{i,N} \rv^2 +  \sum_{i=1}^N i^{2s}  \lv \tilde{e}_{3,k}^{i,N} \rv^2  & \les 
\exo \sum_{i=1}^N \frac{i^{2s} \lambda_i^2}{N}\\
&+  \exo \sum_{i=1}^N i^{2s} \lambda_i^2  \frac{\lv \zeta\kn\rv^4}{N}  e^{\frac{2\ell^2}{N}
\|\zeta_k^N\|^2}\\
& +\exo e^{\frac{2}{N}\|\zeta_k^N\|^2}  \left(\mathbb{E}_k   \frac{1}{(1+ \lv R \rv \sqrt{N})^2} \right)^{1/2}
\sum_{i=1}^N i^{2s} \lambda_i^2 \left( \lv \zeta\kn\rv^2
+ \lv \zeta\kn\rv^4 \right)\\
&+ \exo  \sum_{i=1}^N  i^{2s}\lambda_i^2  \frac{\lv \zeta\kn\rv^2+\lv \zeta\kn\rv^3 }{\sqrt{N}}
\frac{1}{\sqrt{S_k^N}}\\
&+ \exo  \sum_{i=1}^N  i^{2s}\lambda_i^2 \lv \zeta\kn\rv^2 \frac{\left( 1+ \nors{x_k} \right)}{\sqrt{N}} 
\frac{1}{\sqrt{S_k^N}}\,.
\end{align*}
After simple manipulations and using  Lemma \ref{lem:bound1} and Lemma \ref{lem:bound2}, we have 
$$
\exo  \frac{1}{N}  \sum_{k=0}^{[TN]} \left( \sum_{h=1, h\neq 3}^{5} \sum_{i=1}^N i^{2s}  \lv e_{h,k}^{i,N} \rv^2 +  \sum_{i=1}^N i^{2s}  \lv \tilde{e}_{3,k}^{i,N} \rv^2\right)   \longrightarrow 0, 
$$
as $N \ra \infty$. 
\begin{remark}\label{remvalue}
\textup{When we apply \eqref{lembound26} of Lemma \ref{lem:bound2} to the above, we need to enforce the condition $p \geq 2\alpha / \varsigma$, under which \eqref{lembound26} holds.  Rewriting such a condition as $\varsigma \geq 2\alpha /p$  and observing that this condition is always applied in the above with $p\geq 2$ and $\alpha \leq 1$, we get the constraint $\varsigma \geq 1/2$ appearing in Assumptions   \ref{ass:1}. } 
\end{remark}
Returning to the proof, if we prove the limit, 
$$
\exo  \frac{1}{N}  \sum_{k=0}^{[TN]} \sum_{i=1}^N i^{2s}  \lv \bar{e}_{3,k}^{i,N} \rv^2   \ra 0, 
$$
 we are done.  To study $\bar{e}_{3,k}^{i,N}$, we use again \eqref{EeRi} and the bound $\Phi \leq 1$, obtaining
$$
\lv \bar{e}_{3,k}^{i,N}\rv \les \lambda_i \lv \zeta_k^{i,N}\rv
 \left( e^{\frac{\ell^2}{N} \lv \zeta\kn \rv^2} -1  \right) e^{\frac{\ell^2}{N} \| \zeta_k^N\|^2}.
$$
Therefore, by the weighted Jentzen inequality and \eqref{lembound2ebounded}, 
\begin{align*}
\exo \sum_{i=1}^N i^{2s} \lv \bar{e}_{3,k}^{i,N}\rv^2
& \les \exo \left[ e^{\frac{2\ell^2}{N} \| \zeta_k^N\|^2} \sum_{i=1}^N i^{2s} \lambda_i^2 \lv \zeta_k^{i,N}\rv^2
\left( e^{\frac{\ell^2}{N} \lv \zeta\kn \rv^2} -1  \right)^2\right] \\
& \les \left( \exo \sum_{i=1}^N i^{2s} \lambda_i^2 \lv \zeta_k^{i,N}\rv^4
\left( e^{\frac{\ell^2}{N} \lv \zeta\kn \rv^2} -1  \right)^4 \right)^{1/2}.
\end{align*}
Using the local Lipshitz property of the function $e^x$, we have 
\begin{align*}
& \exo \sum_{i=1}^N i^{2s} \lambda_i^2 \lv \zeta_k^{i,N}\rv^4
\left( e^{\frac{\ell^2}{N} \lv \zeta\kn \rv^2 } -1  \right)^4\\
& = \exo \sum_{i=1}^N i^{2s} \lambda_i^2 \lv \zeta_k^{i,N}\rv^4
\left( e^{\frac{\ell^2}{N} \lv \zeta\kn \rv^2 } -1  \right)^4  {\bf 1}_{\{ \frac{\ell^2}{N} \lv \zeta\kn \rv^2 < \log \sqrt{N} \}}\\
& + \exo \sum_{i=1}^N i^{2s} \lambda_i^2 \lv \zeta_k^{i,N}\rv^4
\left( e^{\frac{\ell^2}{N} \lv \zeta\kn \rv^2 } -1  \right)^4  {\bf 1}_{\{ \frac{\ell^2}{N} \lv \zeta\kn \rv^2 \geq \log \sqrt{N} \}}\\
& \les e^{\log \sqrt{N}} \exo \sum_{i=1}^N i^{2s} \lambda_i^2 
\frac{\lv \zeta\kn \rv^6}{N} 
+  \exo \sum_{i=1}^N i^{2s} \lambda_i^2 \lv \zeta_k^{i,N}\rv^4
\left( e^{ \frac{\ell^2 \| \zeta_k^N\|^2}{N}} -1  \right)^4  {\bf 1}_{\{ \frac{\ell^2}{N} \| \zeta_k^N\|^2 \geq \log \sqrt{N} \}}.
\end{align*}
We now use Markov Inequality,  \eqref{lembound22} and \eqref{lembound2ebounded}  to estimate the second addend, obtaining 
\begin{align*}
& \exo \sum_{i=1}^N i^{2s} \lambda_i^2 \lv \zeta_k^{i,N}\rv^4
\left( e^{\frac{\ell^2}{N} \lv \zeta\kn \rv^2} -1  \right)^4\\
&  \leq 
\exo \sum_{i=1}^N i^{2s} \lambda_i^2 
\frac{\lv \zeta\kn \rv^6}{\sqrt{N}}  + 
\left( \mathbb{P}\left\{\frac{\ell^2}{N} 
\| \zeta_k^N\|^2 \geq \log \sqrt{N}\right\}\right)^{1/2} 
\\
& \les \frac{1}{\sqrt{N}} +
 \left( \exo e^{\frac{ \ell^2}{N} 
\| \zeta_k^N\|^2} e^{- \log{\sqrt{N}}}  \right)^{1/2} 
\stackrel{\eqref{lembound2ebounded}  }{\longrightarrow} 0. 
\end{align*}

This concludes the proof. 
\end{proof}

\subsection{Analysis of the Noise}\label{sec:noisex}
The proof of Lemma \ref{lem:noise1} is based on   Lemma \ref{lem:finidimdistr+tightness} below. In order to state such a lemma let us introduce the following notation and definitions. Let $k_N:[0,T] \ra \mathbb{Z}_+$ be a sequence of nondecreasing, right continuous functions indexed by $N$, with $k_N(0)=0$ and $k_N(T)\geq 1$. 
Let $\h$ be any Hilbert space and  $\{X\kkn, \mathcal{F}\kkn\}_{0\leq k \leq k_N(T)}$ be a $\h$-valued martingale 
difference array (MDA), i.e. a double sequence of random variables such that $\EE[X\kkn\vert\mathcal{F}_{k-1}^N ]=0$, $\EE[\| X\kkn\|^2\vert\mathcal{F}_{k-1}^N ]< \infty$ almost surely and $\mathcal{F}^{k-1, N} \subset \mathcal{F}\kkn$. Consider the process $\mathcal{X}^N(t)$ defined by
$$
\mathcal{X}^N(t):=\sum_{k=1}^{k_N(t)}X_k^N \,,
$$
if $k_N(t)\geq 1$ and $k_N(t) > \lim_{v\ra 0+} k_N(t-v)$ and by linear interpolation otherwise.   With this set up we state the following result.

\begin{lemma}\label{lem:finidimdistr+tightness}
 Let $\mathfrak{T}:\h \ra \h$ be a self-adjoint positive definite trace class operator on a separable Hilbert space $(\h, \|\cdot \|)$. Suppose
\begin{description}
\item[i)] there exists a continuous and positive function $f(t)$ defined on $[0,T]$ such that
$$\lim_{N\ra \infty} \sum_{k=1}^{k_N(T)} \EE(\|{X\kkn}\|^2\vert \mathcal{F}_{k-1}^N)= \tr(\mathfrak{T}) \int_0^T f(t) dt \,\qquad {\mbox{in probability}} ; $$ 
\item[ii)] if $\{{\phi}_j\}$ is an orthonormal basis of $\h$ then 
$$
\lim_{N\ra \infty} \sum_{k=1}^{k_N(T)} 
\EE(\langle X\kkn,{\phi}_j  \rangle \langle X\kkn,{\phi}_i  \rangle\vert \mathcal{F}_{k-1}^N)=0\, \quad \mbox{for all }\,\, i\neq j\, ;
$$
\item[iii)]for every fixed $\epsilon>0$,
$$
\lim_{N \ra \infty} \sum_{k=1}^{k_N(T)}
\EE(\|{X\kkn}\|^2) {\bf 1}_{\{\|{X\kkn}\|^2\geq \epsilon \}}
\vert \mathcal{F}_{k-1}^N )=0,  \qquad\mbox{in probability}.
$$
\end{description}
Then the sequence $\mathcal{X}^N$ converges weakly in $C([0,T]; \hs)$ to the stochastic integral $\int_0^T \sqrt{f(t)} dW_t$,  where $W_t$ is a 
$\h$-valued $\mathfrak{T}$-Brownian motion.
\end{lemma}
\begin{proof} This lemma is in the same spirit as   
\cite[Proposition 4.1 and Remark 4.2]{Matt:Pill:Stu:11}. As observed in \cite[Proof of Theorem 5.1]{MR833268}, the statement just needs to be proved for a finite dimensional Hilbert space, i.e. in finite dimensions. The first two conditions are needed to ensure the weak convergence of the finite dimensional 
distributions of $\mathcal{X}^N$,  the last condition guarantees tightness of the sequence, see \cite[Theorem 3.2]{MR668684} and 
\cite[Remark 4.2]{Matt:Pill:Stu:11}.  One may also consult the more compact \cite[Section 5.5]{Ottobre2016}.
\end{proof}

\begin{proof}[Proof of Lemma \ref{lem:noise1}] We apply 
Lemma \ref{lem:finidimdistr+tightness}   with $k_N(t)=[tN]$,  $X_k^N=M_k^{1,N}/\sqrt{N}$ and  $\mathcal{F}_k^N$  the sigma-algebra generated by $\{ 
\gamma_{h+1}^N, \xi_{h+1}^N, \, 0\leq h\leq k\}$ to study the sequence $\eta^N(t)$, defined in (\ref{etaN}), in  the Hilbert space $\hs$. We now check that the three conditions of Lemma \ref{lem:finidimdistr+tightness} hold in the present case.  
 \begin{description}
\item[i)] 
 We need to show that 
\be\label{m1goestoint}
\frac{1}{N}\exo\sum_{k=0}^{[TN]} \EE_k \nors{M_k^{1,N}}^2 \lra \tr(\C_s) \int_0^T \Gl(S(u))du \, .
\ee
From the definition of $M_k^{1,N}$,  equation \eqref{Mkn}, we have
\be\label{useiniii}
\frac{1}{N} \nors{M_k^{1,N}}^2= \nors{x\kpo^N-x\kkn -\EE_k(x\kpo^N-x\kkn)}^2,
\ee
hence
\begin{align}
\frac{1}{N}\EE_k \nors{M_k^{1,N}}^2&=\EE_k \nors{x\kpo^N-x\kkn}^2 -\nors{\EE_k(x\kpo^N-x\kkn)}^2 \nonumber\\
&= \frac{2\ell^2}{N} \EE_k\nors{\gamma\kpo \C^{1/2}\xi\kpo^N}^2-
\nors{\EE_k(x\kpo^N-x\kkn)}^2, \label{diff+err}
\end{align}
where the above equality holds thanks to  $\eqref{chaininh}$. 
We will show that
\be\label{meanerrorterm}
\exo\sum_{k=1}^{[TN]}\nors{\EE_k(x\kpo^N-x_k^N)}^2\lra 0 \quad \mbox{as } N \ra \infty.  
\ee
Assuming the above for the moment, let us focus on the first addend in 
(\ref{diff+err}):
\begin{align*}
\frac{2\ell^2}{N} \EE_k\nors{\gamma\kpo \C_N^{1/2}\xi\kpo^N}^2&=
\frac{2\ell^2}{N} \sum_{j=1}^N \lambda_j^2j^{2s}\EE_k\lv  
\gamma\kpo \xi\kpo\jnn\rv^2\\
&= \frac{2\ell^2}{N} \sum_{j=1}^N \lambda_j^2j^{2s}\EE_k
  \left[(1\wedge e^Q) \lv\xi\kpo\jnn\rv^2\right]\\
  &=\frac{2\ell^2}{N} \sum_{j=1}^N \lambda_j^2j^{2s}\EE_k^{\xi}
  \left[(1\wedge e^{R_j}) \lv\xi\kpo\jnn\rv^2\right]+\frac{a_{1,k}^N}{N},
\end{align*}
where
\be\label{a1kn}
a_{1,k}^N:=2\ell^2 \sum_{j=1}^N \lambda_j^2j^{2s}\EE_k
  \left[\left((1\wedge e^{Q})-(1\wedge e^{R_j})\right) \lv\xi\kpo\jnn\rv^2\right].
\ee
We now use the same technique that we used for the drift coefficient (that is, we first take expectation with respect to $\xi^i$ and then with respect to $\xi\setminus \xi^i$), obtaining
\begin{align}
\frac{2\ell^2}{N} \EE_k\nors{\gamma\kpo \C^{1/2}\xi\kpo^N}^2&=
\frac{2\ell^2}{N} \sum_{j=1}^N \lambda_j^2j^{2s}\EE_k^{\xi_j^-}
  (1\wedge e^{R_j})+\frac{a_{1,k}^N}{N} \nonumber \\
  &=\frac{2\ell^2}{N} \sum_{j=1}^N \lambda_j^2j^{2s}\EE_k^{\xi}
  (1\wedge e^{Z_{\ell,k}})+\frac{a_{1,k}^N}{N}+\frac{a_{2,k}^N}{N} \nonumber\\
  &=\frac{1}{N} \sum_{j=1}^N \lambda_j^2j^{2s}
  \Gamma(S_k^N)+\frac{a_{1,k}^N}{N}+\frac{a_{2,k}^N}{N} \label{useiniii2},
\end{align}
having used \eqref{approx4} and \eqref{GammaGamma} and having set
\be\label{a2kn}
a_{2,k}^N:=2\ell^2\sum_{j=1}^N \lambda_j^2j^{2s}\EE_k^{\xi}
\left( (1\wedge e^{R_j})- (1\wedge e^{Z_{\ell,k}})  \right).
\ee
Therefore 
\begin{align}\label{discrint}
\frac{2\ell^2}{N} \sum_{k=0}^{[TN]}\EE_k\nors{\gamma\kpo \C^{1/2}\xi\kpo^N}^2
= \frac{1}{N}\sum_{k=0}^{[TN]} \sum_{j=1}^N \lambda_j^2j^{2s}
  \Gamma(S_k^N)+\sum_{k=0}^{[TN]}\left(\frac{a_{1,k}^N}{N}+\frac{a_{2,k}^N}{N}
  \right).
\end{align}
If we prove that
\be\label{a1kna2kntozero}
\frac{1}{N}\exo\sum_{k=0}^{[TN]}\lv a_{1,k}^N \rv \ra 0 \quad
 \mbox{and} \quad 
\frac{1}{N}\exo\sum_{k=0}^{[TN]}\lv a_{2,k}^N \rv \ra 0, 
\ee
 then (\ref{m1goestoint}) follows from (\ref{diff+err}), (\ref{meanerrorterm}), (\ref{discrint}) and the above two limits. We therefore move on to proving   the limits 
in (\ref{a1kna2kntozero}). Let us start from the latter:
\begin{align}
\exo \lv a_{1,k}^N \rv &\les \exo \sum_{j=1}^N \lambda_j^2j^{2s}
\left(\EE_k
  \lv(1\wedge e^{Q})-(1\wedge e^{R_j})\rv^2\right)^{1/2} \nonumber \\
 & \les \exo\sum_{j=1}^N \lambda_j^2j^{2s}
\left(\EE_k \lv Q- R_j\rv^2\right)^{1/2}\nonumber \\
  &\les \exo \sum_{j=1}^N \lambda_j^2j^{2s}
\left(\EE_k \lv Q- R\rv^2\right)^{1/2}\label{1bit}\\
&+\exo\sum_{j=1}^N \lambda_j^2j^{2s}
\left(\EE_k \lv R- R_i\rv^2\right)^{1/2}.\label{2bit}
\end{align} 
The addend \eqref{1bit} tends to zero as $N \ra \infty$ by using \eqref{exprn} and \eqref{decompQvera}. For (\ref{2bit}) instead we have, by \eqref{defRi}, 
\begin{align*}
\eqref{2bit}&\les\sum_{j=1}^N \lambda_j^2j^{2s} \exo\left( \EE_k \lv R- R_i\rv^2\right)\\
&\les \sum_{j=1}^N \lambda_j^2j^{2s} \exo \left(\frac{1}{N}+ \frac{\lv \zeta\kjn\rv^2}{N}\right).
\end{align*}
The first limit in (\ref{a1kna2kntozero}) now follows from (\ref{lembound22}). The second limit in (\ref{a1kna2kntozero}) can be shown analogously, using this time  the bounds \eqref{wassrir} and  (\ref{wassrzlk}). 

Finally, to show (\ref{meanerrorterm}), observe that from (\ref{defen}), 
$$
\nors{\EE_k(x\kpo^N-x\kkn)}^2 \les \frac{\nors{\Theta(x\kkn, S\kkn)}^2}{N^2}+
\frac{\nors{e_k^N}^2}{N^2}.
$$
The desired result now follows from Lemma \ref{lem:stimaEkn}, \eqref{boundonmomentsofx}  and the bound
$$
\nors{\Theta(x,S)}\les 1+ \nors{x},
$$
(which is a consequence of the definition \eqref{Theta} and \eqref{C}).

\item[ii)] Condition {\bf ii)} of Lemma \ref{lem:finidimdistr+tightness} can be shown to hold with similar calculations, so we will not show the details. \item[iii)] It will suffice to show that 
$$
 \lim_{N \ra \infty} \frac{1}{N} \EE_{x_0}
\sum_{k=0}^{[TN]} \EE_k (\nors{M_k^{1,N}}^2 
{\bf 1}_{\{\nors{M_k^{1,N}}^2 > \epsilon N\}})=0\,.
$$
Using the Markov inequality, 
\begin{align*}
\EE_k (\nors{M_k^{1,N}}^2 
{\bf 1}_{\{\nors{M_k^{1,N}}^2 > \epsilon N\}})
&\leq \left( \EE_k \nors{M_k^{1,N}}^4 \right)^{1/2}
\left( \mathbb{P}\, 
\{\nors{M_k^{1,N}}^2 > \epsilon N \} \right)^{1/2}\\
& \leq \frac{1}{N \epsilon} \EE_k \nors{M_k^{1,N}}^4\,.
\end{align*}
By \eqref{useiniii},  \eqref{chaininh} and \eqref{c1/2xi}, 
$$
\frac{1}{N} \EE_k \nors{M_k^{1,N}}^4 \les \EE_k\nors{(x\kpo^N-x_k^N)}^4  \les \frac{1}{N^2}
\EE_k \nors{\gamma\kpo \C^{1/2}\xi\kpo^N}^4 \les \frac{1}{N^2}. 
$$
Therefore
\begin{align}
\frac{1}{N}  \EE_{x_0}
\sum_{k=0}^{[TN]} \EE_k (\nors{M_k^{1,N}}^2 
{\bf 1}_{\{\nors{M_k^{1,N}}^2 > \epsilon N\}})& \les 
\frac{1}{N^2 \epsilon}
\EE_{x_0}
\sum_{k=0}^{[TN]} \EE_k (\nors{M_k^{1,N}}^4) \les \frac{1}{N^2} \,.
\end{align}
\end{description} 
\end{proof}

\section*{Appendix A}
\begin{proof}[Proof of \eqref{c1/2xi}]
We just need to prove it for $p$ even. So let $q\geq 1$; then, by the weighted Jensen's inequality
$$
\EE\nors{\C^{1/2}\xi^N}^{2q} \leq \EE \left(\sum_{i=1}^{\infty} i^{2s}
\lambda_i^2 \lv\xi\inn\rv^2\right)^q  \leq
(Trace (\C_s))^{q-1} \sum_{i=1}^{\infty} i^{2s}
\lambda_i^2 \EE\lv\xi\inn\rv^{2q}< \infty.
$$
Alternatively, one can observe that \eqref{c1/2xi} is a consequence of Fernique's theorem.
\end{proof}
Before proving Lemma \ref{lem:propofDandGamma}, we recall the following fact, which has already been pointed out in Section \ref{subs:rellit}.
\begin{remark}\label{rem:DandG}\textup{
We recall that for $X\in \R_+$ and $b \in \R$,   $\dl(X)=\mathcal{G}_{\ell \sqrt{2}}(X,1)$, where $\mathcal{G}_{\ell}(X,b)$ is the drift function defined in \cite[(1.7)]{JLM12MF}. Analogously, our $\Gl(X)$ is $\Gamma_{\ell \sqrt{2}}(X,1)$, where 
$\Gl(X,b)$ is defined in  \cite[(1.6)]{JLM12MF}.} {\hfill$\Box$}
\end{remark}
\begin{proof}[Proof of Lemma \ref{lem:propofDandGamma}]
 The boundedness of $\dl$ and $\Gl$ follows from Remark \ref{rem:DandG} and  \cite[Lemma 2]{JLM12MF}.  Lipshitzianity follows simply observing that both functions have bounded derivative, indeed
\begin{align}
\frac{d}{dx}\dl (x) &=  \ell^2 \dl(x)-\frac{\ell^3}{\sqrt{\pi}}\left(  \frac{1}{\sqrt{x}}+\frac{1}{2 \, x^{3/2}} \right)
e^{-\frac{\ell^2}{4x}}  \nonumber \\
\frac{d}{dx}\Gl (x) &=  \ell^2 \dl(x) -  \frac{\ell^3}{\sqrt{\pi x}}  e^{-\frac{\ell^2}{4x}}.\label{derG}
\end{align}
Global Lipshitzianity of $\sqrt{\Gl}$ then follows after observing that $\Gl$ is bounded below away from zero (see \eqref{defGamma}). 

We now want to show that the derivative of $\al(x)$ is bounded. From the definition of $\al$ (equation \eqref{defAl}) we have 
\be\label{derAl}
\pa_x\al (x)= -2\dl(x)-2x\,\pa_x\dl(x)+\pa_x\Gl(x) \, . 
\ee
We will prove that 
\be\label{limal}
\lim_{x\ra +\infty }\pa_x\al(x) = 0\,.  \qquad  
\ee
 Because $\pa_x\al$ is a continuous function on $[0,+\infty)$, \eqref{limal} implies the  boundedness of $\pa_x\al(x)$. In order to prove \eqref{limal} we will prove that all the addends on the RHS of  \eqref{derAl} tend to zero (see also Figure \ref{fig:deral} below).
\begin{figure}[ht]
\centering
\includegraphics[height=7cm]{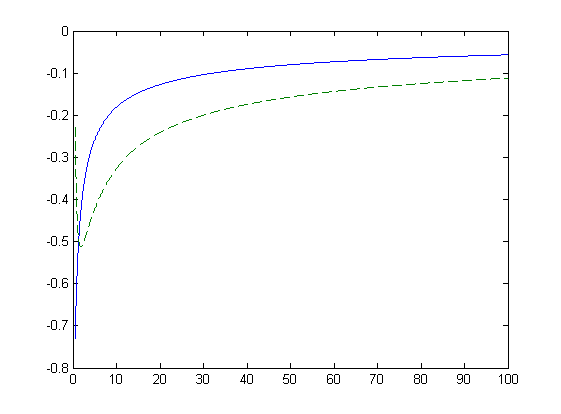}
\caption{Plots of the function $\pa_x\al(x)$ for $\ell=1$ and $\ell=2$ (dashed line).}
\label{fig:deral}
\end{figure}

\begin{itemize}
\item First of all, let us prove 
\be\label{limdl}
\lim_{x \ra + \infty}\dl(x)=0. 
\ee
The above limit follows from the definition of  $\dl$ \eqref{defD} by simply  applying de l'Hopital's rule:
\begin{align}
\lim_{x\ra +\infty}\frac{ \Phi\left( \frac{\ell (1-2x)}{\sqrt{2x}}\right)}{e^{-\ell^2 (x-1)}} & =
\lim_{x\ra +\infty} \frac{e^{-\ell^2(1-2x)^2/4x}}{\sqrt{2\pi}\,\ell^2 e^{-\ell^2 (x-1)}}
\left( \frac{\ell}{2\sqrt{2}x^{3/2}}+\frac{\ell}{\sqrt{2x}} \right) \label{frz} \\
& = \lim_{x\ra +\infty} e^{-\ell^2/4x}\frac{1}{\ell\sqrt{\pi}}\left( \frac{1}{4 x^{3/2}}+ \frac{1}{2\sqrt{x}}\right)=0.  \nonumber
\end{align}

\item From \eqref{derG} and \eqref{limdl}, also $\pa_x \Gl(x)\ra 0$ as $x \ra +\infty$. 
\item Now the second addend:
$$
\lim_{x\ra +\infty}-2x\, \pa_x \dl(x)=0.
$$
Indeed, 
$$
-2x\pa_x\dl = -2x \ell^2 \dl(x)+ 2\frac{\ell^3}{\sqrt{\pi}}\left( \sqrt{x}+\frac{1}{2\sqrt{x}}  \right)
e^{-\ell^2/4x}\,,
$$
therefore
\begin{align*}
\lim_{x\ra +\infty} -2x\pa_x\dl & =\lim_{x\ra +\infty} -4\ell^4\, \frac{x\, \Phi\left( \frac{\ell (1-2x)}{\sqrt{2x}} \right) }{e^{-\ell^2(x-1)}}  
+2\frac{\ell^3}{\sqrt{\pi}} \sqrt{x}e^{-\ell^2/4x}\\
&= \lim_{x\ra +\infty} -4\ell^4\left[ \frac{\Phi\left( \frac{\ell (1-2x)}{\sqrt{2x}} \right) - e^{-\ell^2 (1-2x)^2/4x}   \left( \frac{\ell}{4\sqrt{\pi x}}+ \frac{\ell \sqrt{x}}{2\sqrt{\pi}} \right)}
{-\ell^2 e^{-\ell^2(x-1)}}\right] \\
&+ \lim_{x\ra +\infty}       2\frac{\ell^3}{\sqrt{\pi}} \sqrt{x}e^{-\ell^2/4x}\\
&\stackrel{\eqref{frz}}{=} \lim_{x\ra +\infty} -2\frac{\ell^3}{\sqrt{\pi}} \sqrt{x}e^{-\ell^2/4x}
+2\frac{\ell^3}{\sqrt{\pi}} \sqrt{x}e^{-\ell^2/4x} =0
\end{align*}
\end{itemize}

Finally,  the sign of $\al(x)$ is studied in 
\cite[page 258]{MR2137324}.
\end{proof}

 We include here plots of the functions $\dl(x)$ and  $\Gl(x)$, Figure \ref{fig:dl} and Figure \ref{fig:gl} below. 
\begin{figure}[ht]
\centering
\includegraphics[height=7cm]{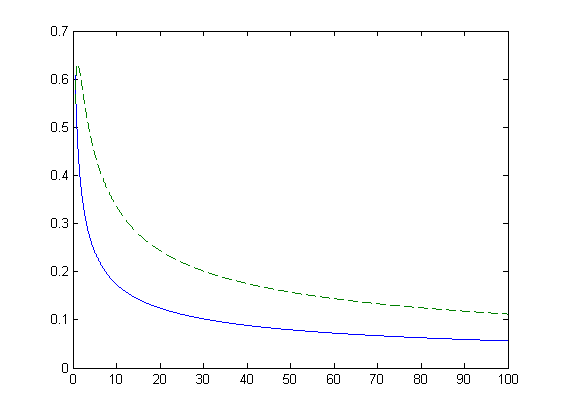}
\caption{Plots of the function $\dl(x)$ for $\ell=1$ and $\ell=2$ (dashed line).}
\label{fig:dl}
\end{figure}
\begin{figure}[ht]
\centering
\includegraphics[height=7cm]{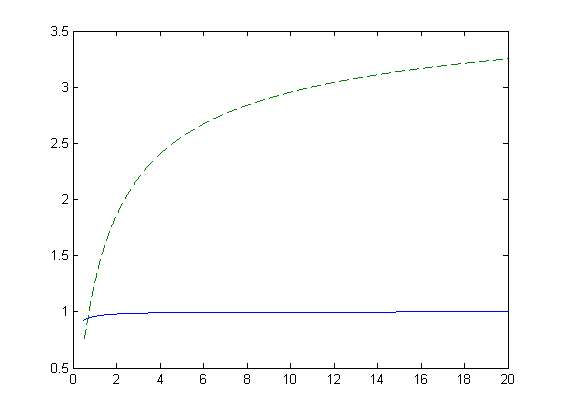}
\caption{Plots of the function $\Gl(x)$ for $\ell=1$ and $\ell=2$ (dashed line).}
\label{fig:gl}
\end{figure}

\begin{proof}[Proof of first equality in \eqref{appdr1}] We want to prove
$$\EE_k^{\xi, \gamma}(\gamma\kpo\, \xi\kpo\inn) 
=  \EE_k^{\xi}(\alpha\kpo\, \xi\kpo\inn).$$
Let $f_{\gamma\kpo, \xi\kpo}(\gamma, \xi)$ be the joint distribution (given $x_k$) of $\gamma\kpo$ and  $\xi\kpo$. Then 
\begin{align*}
\EE_k^{\xi, \gamma}(\gamma\kpo\, \xi\kpo\inn)
& = \iint \gamma\, \xi^i \, f_{\gamma\kpo, \xi\kpo}(\gamma, \xi) 
 =  \int \xi^i \int \gamma f_{\xi\kpo}(\xi) f_{\gamma\kpo\vert \xi\kpo}(\gamma\vert \xi) \nonumber\\
&= \int \xi^i f_{\xi\kpo}(\xi) \alpha\kpo(\xi)=
\EE_k^{\xi}(\alpha\kpo\, \xi\kpo\inn).
\end{align*}
\end{proof}

\section*{Appendix B}
Before starting the proofs of the various lemmata, we  derive equation \eqref{z=x+psi} below, which will be repeatedly used throughout this appendix. 

From \eqref{defzetakn}  and recalling that $\left[ \nabla \Psi^N(x_k) \right]^i$ denotes the $i$-th component of 
$\nabla \Psi^N(x_k)$, 
\be\label{zzzz1}
\zeta\kn = \frac{x\kn}{\lambda_i} + \lambda_i  \left[ \nabla \Psi^N(x_k) \right]^i \,.
\ee
Using the  bound \eqref{c12dpsi111} we have
\begin{align*}
\lv\lambda_i \left[ \nabla \Psi^N(x_k) \right]^i\rv^2 \leq 
\sum_{i=1}^{\infty}\lv\lambda_i \left[ \nabla \Psi^N(x_k) \right]^i\rv^2
\les \nors{x}^{2\varsigma}+ \nors{x}^{2} , 
\end{align*}
hence
$$
\lv\lambda_i \left[ \nabla \Psi^N(x_k) \right]^i\rv^p
\les \nors{x}^{\varsigma p}+ \nors{x}^{p}.
$$
Therefore for every $p \geq 0$
\be\label{z=x+psi}
\lv \zeta\kn \rv^{p} \les  \frac{\lv x\kn \rv^{p}}{\lambda_i^{p}}+ (\nors{x_k}^{\varsigma p}+ \nors{x_k}^{p}),
\ee
\begin{proof}[Proof of Lemma \ref{lem:bound1}] We will prove, in order,  the bounds
\eqref{boundonmomentsofsk}, \eqref{boundonmomentsofx}  and \eqref{lembound2ebounded}. 

$\bullet $  { \bf  Proof of \eqref{boundonmomentsofsk}}. We can act as in \cite[Proof of Lemma 9]{Pillai2014} (in comparing our proof with \cite[Proof of Lemma 9]{Pillai2014} set $\delta=N^{-1}$ in \cite{Pillai2014}).  Looking at \cite[Proof of Lemma 9]{Pillai2014}, all we need to show is
$$
\exo (S_{k+1}^N)^{2m} - \exo (S_{k}^N)^{2m} \les \frac{1}{N} (1+ \exo (S_{k}^N)^{2m}).
$$
A close inspection of the method of proof used in \cite{Pillai2014} reveals that  showing the above boils down to proving the following two bounds:
\be\label{v1}
\lv \EE_k \left[S_{k+1}^N-S_k^N \right] \rv \les \frac{1}{N} \left( 1+ S_k^N \right), \quad p\geq 1.
\ee
and
\be\label{v2}
\left( \EE_k\lv S_{k+1}^N-S_k^N \rv^p\right)^{1/p} \les \frac{1}{\sqrt{N}} (1+ S_k^N)\,.
\ee
Let us start with proving \eqref{v1}. To this end let us observe that by  \eqref{defEn} and \eqref{approximatedrifts}, one has 
\be\label{v11}
\EE_k \left[S_{k+1}^N-S_k^N \right]= \frac{E_k^N}{N}+ \frac{\al(S_k^N)}{N}\,.
\ee
Now  notice that 
\be\label{xnls}
\frac{\nors{x_k^N}^2}{N} = \frac{1}{N} \sum_{i=1}^N
i^{2s} \lv x\kn\rv^2 = \frac{1}{N} \sum_{i=1}^N
i^{2s} \lambda_i^2 \frac{\lv x\kn\rv^2}{\lambda_i^2} \les \frac{1}{N} \sum_{i=1}^N
 \frac{\lv x\kn\rv^2}{\lambda_i^2} = S_k^N\, ,
\ee
where in the last inequality we have used the fact that $\sum_i \lambda_i^2 i^{2s}$ is convergent and therefore $\lambda_i^2 i^{2s}$ is bounded. To bound the RHS of   \eqref{v11}, we recall that from the proof of Lemma \ref{lem:AlN-Al} one has $E_k^N= E_{1,k}^N+ E_{2,k}^N+\EE_k \hat{r}^N$ (see \eqref{decompee}).  Acting like we did to obtain  \eqref{uselater},  one has
\be\label{v12}
\lv E_{1,k}^N \rv \les \frac{1}{\sqrt{N}} \left( S_k^N+ \frac{\nors{x_k^N}^2}{N}\right)^{1/2}
\stackrel{\eqref{xnls}}{\les}\frac{1}{\sqrt{N}}(S_k^N+1).
\ee
With steps analogous to those used to obtain  \eqref{eek2}, one also has 
\be\label{v13}
\lv E_{2,k}^N\rv \les \frac{1}{\sqrt{N}} \left( 1+ \nors{x_k^N}^2 \right)^{1/2}
\les (1+S_k^N).
\ee
  Now \eqref{v1} follows from \eqref{decompee},  \eqref{v11}, \eqref{v12}, \eqref{v13}, \eqref{xnls}, \eqref{estrhat} and 
$$
\al(a) \les (1+a), \qquad a\geq 0\, .
$$
To prove \eqref{v2} one can instead just use \eqref{red},  \eqref{rknp} (together with $\gamma_{k+1} \leq 1$) and calculations analogous to those leading to \eqref{estrhat}. This concludes the proof of  \eqref{boundonmomentsofsk}.

$\bullet $ { \bf  Proof of \eqref{boundonmomentsofx}}.  For this bound we will use the same strategy of proof that we used to show \eqref{boundonmomentsofsk}. So we only need to prove
\be\label{mx1}
\nors{\EE_k (x\kpo^N-x_k^N)}\les \frac{1}{N}(1+\nors{x_k^N} )
\ee
and
\be\label{mx2}
\left(\EE_k \nors{x\kpo^N-x_k^N}^p\right)^{1/p}
\les \frac{1}{\sqrt{N}}(1+\nors{x_k^N} )\,.
\ee
Let us start with \eqref{mx1}:
\begin{align*}
\nors{\EE_k (x\kpo^N-x_k^N)} & =\sqrt{\frac{2\ell^2}{N}}
\nors{\EE_k (\gamma\kpo \C^{1/2}\xi^{N}\kpo)}\\
&= \sqrt{\frac{2\ell^2}{N}}
\left(\sum_{i=1}^N i^{2s} \lv \EE_k (\gamma\kpo \lambda_i\xi^{i,N}\kpo)\rv^2 \right)^{1/2}
\end{align*}
We therefore need to estimate $\lv \EE_k (\gamma\kpo \lambda_i\xi^{i,N}\kpo)\rv^2$. In order to do so,  we make the following preliminary observation: from \eqref{decompQvera} and \eqref{defRi} we have
$$
Q(x_k, \xi_{k+1}) = R^i(x_k, \xi\kpo)-
\frac{\ell^2}{N}  \lv \xi_{k+1}\inn\rv^2- \sqrt{\frac{2\ell^2}{N}}
 \zeta_k^{i,N} \xi\kpo\inn + r^N(x_k, \xi_{k+1}). 
$$
As we have already said, $R^i$ contains only terms that do not depend on the noise $\xi\inn\kpo$, therefore we can write
\begin{align*}
\lv \EE_k (\gamma\kpo \lambda_i\xi^{i,N}\kpo)\rv^2   & = 
\lv \EE_k \left[\left( 1 \wedge e^Q \right)   \lambda_i\xi^{i,N}\kpo \right]\rv^2\\
& = \lv \EE_k \left[\left( \left(1 \wedge e^Q\right)  - \left(1 \wedge e^{R^i}\right)\right)   \lambda_i\xi^{i,N}\kpo \right]\rv^2\\
& \leq \lambda_i^2 \lv \EE_k \left( \lv Q-R^i \rv  \lv \xi^{i,N}\kpo \rv \right)\rv^2\\
& \les \lambda_i^2 
\lv \EE_k \frac{\lv \xi^{i,N}\kpo \rv^3}{N} +\EE_k \frac{\lv \zeta_k^{i,N} \rv\lv \xi^{i,N}\kpo \rv^2}{\sqrt{N}} + \EE_k \lv \xi^{i,N}\kpo  r^N\rv    \rv^2 \,.
\end{align*}
Using \eqref{57p} and the fact that $\zeta_k$ depends only on $x_k$, we have
\begin{align}
\lv \EE_k (\gamma\kpo \lambda_i\xi^{i,N}\kpo)\rv^2   & \les 
 \lambda_i^2  \left( \frac{1}{N^2}+ \frac{\lv \zeta_k^{i,N} \rv^2}{{N}}\right)\nonumber\\
& \stackrel{\eqref{z=x+psi}}{\les}  \lambda_i^2 
\left( \frac{1}{N^2}+ \frac{\lv x_k^{i,N}\rv^2}{\lambda_i^2 N} + \frac{\nors{x_k}^2+1}{N}\right). 
\label{ekgxi}
\end{align}
\eqref{mx1} is  now a simple consequence of the above bound. 
For \eqref{mx2}, instead, we just use $\gamma\kpo\leq 1$ and 
$$
\left(\EE_k \nors{x\kpo^N-x_k^N}^p\right)^{1/p}\les 
\left( \frac{1}{N^{p/2}} \EE_k \nors{\C^{1/2}\xi\kpo}^p\right)^{1/p} \stackrel{\eqref{c1/2xi}}{\les}\frac{1}{\sqrt{N}}.
$$

$\bullet $ { \bf  Proof of \eqref{lembound2ebounded}}.   By acting as we do to obtain \eqref{z=x+psi} (with $p=2$), it is clear that we only need to show 
\be\label{expmomentsskn}
\exo e^{c S_k^N}< \infty \quad \mbox{ and } \quad \exo e^{\frac{c}{N} \| x_k^N\|^2_s}< \infty, 
\quad \mbox{uniformly over } 0 \leq k \leq [TN]+1,
\ee
for all $c>0$. However, by \eqref{xnls}, proving the second of the above bounds boils down to proving the first, which is therefore the only one we need to concentrate on. Such a bound  is a simple consequence of  \eqref{boundonmomentsofx}. Indeed, on inspection of the proof of \eqref{boundonmomentsofx}, one finds that the constants $\bar{c}$ appearing on the RHS of \eqref{boundonmomentsofx} grows at most like $d^m$, where $d>0$ is some positive constant independent of $m, N$ and $k$. Therefore
$$
\exo e^{\frac{c}{N} \| x_k^N\|^2_s} =  \sum_{m=0}^{\infty} c^m  \exo \frac{\| x_k^N\|^{2m}_s}{N^m m !} = e^{c  \, d^2  /N}\les 1. 
$$
This concludes the proof of the lemma.  
\end{proof}

\begin{proof}[Proof of Lemma \ref{lem:asconv}] This proof is in 2 steps.  The first step proves the first part of the statement, the second step proves the second part. 
\begin{itemize}
\item {\it Step 1}. For every fixed $t>0$ and for every  $\epsilon>0$, 
\be\label{bc}
\sum_{N=1}^{\infty}\mathbb{P}\left( \lv \hat{w}^N(t)\rv>\epsilon  \right)< \frac{1}{\epsilon^4}\sum_{N=1}^{\infty} \exo \left( \hat{w}^N(t)\right)^4 < \infty,
\ee
where $\hat{w}^N$ has been defined in \eqref{whatN}. 
Assuming for the moment that \eqref{bc} holds, by the Borel-Cantelli Lemma \eqref{bc} implies that $\hat{w}^N(t)$ converges to zero almost surely. Because almost sure convergence is preserved under continuous transformations, this means  that $S^{(N)}(t)$ converges almost surely to $S(t)$. We only sketch the proof of \eqref{bc}, as the calculations are completely analogous to those contained in the proof Theorem \ref{thm:weak conv of Skn}.   From \eqref{whatN}, we have 
\begin{align}
\exo \left( \hat{w}^N(t)\right)^4 &\les
\int_0^t \exo \left[ \al^N(\bar{x}\bn (v))-\al(\bar{S}\bn(v)) \right]^4 dv \nonumber\\ 
&+\int_0^t \exo \left[ \al (\bar{S}\bn (v))-\al(S\bn(v)) \right]^4 dv 
+\exo \lv w^N(t)\rv^4. \label{2add}
\end{align}
The estimate of the first and third addend on the right hand side of the above is done by proceeding analogously to what we have done for the proof 
of  Lemma
\ref{lem:AlN-Al} and Lemma \ref{lem:noise2}, respectively. The second addend can be studied with similar calculations (indeed, with calculations analogous to those in Step 2 of this proof).  Therefore we only show how to estimate the first addend, the others can be done with a similar procedure, and we leave it to the reader. 
With the notation introduced in the proof of Lemma 
\ref{lem:AlN-Al}, we have
$$
\exo \int_0^t  \left[ \al^N(\bar{x}\bn (v))-\al(\bar{S}\bn(v)) \right]^4 dv
= \exo \frac{1}{N}\sum_{k=0}^{[TN]}\lv E_k^N\rv^4
$$
and 
$
\lv E_k^N \rv^4 \les \lv E_{1,k}^N \rv^4 + \lv E_{2,k}^N \rv^4 + \lv  \EE_k\hat{r}^N\rv^4
$ (see \eqref{decompee}).
Acting as we did to obtain  \eqref{EEE} and \eqref{eek2}, we find 
\be\label{com1}
\exo\lv E_{1,k}^N \rv^4 \les  \frac{1}{N^2}\qquad
\mbox{and} \qquad  \exo \lv E_{2,k}^N \rv^4 \les \frac{1}{N^2}. 
\ee
Using \eqref{estrhat}, one finds that, overall, 
$$
\exo \int_0^t  \left[ \al^N(\bar{x}\bn (v))-\al(\bar{S}\bn(v)) \right]^4 dv \les 
\frac{1}{N^{2}},
$$
and the sequence $a_N= N^{-2}$ is summable. Similar estimates can be obtained for  the addends in \eqref{2add}. This concludes the proof of the almost sure convergence of $\hat{w}^N$ to zero. 
\item {\it Step 2}. For every $\epsilon>0$, 
\be\label{bc11}
\sum_{N=1}^{\infty}\mathbb{P}\left( \lv \bar{S}\bn(t) - S^{(N)}(t)\rv > \epsilon\right)<  \frac{1}{\epsilon^2}\sum_{N=1}^{\infty} \exo
\lv \bar{S}\bn(t)- S^{(N)}(t)\rv^2< \infty\, .
\ee
Again, if we prove the above, by the B-C Lemma we have almost sure convergence of $\bar{S}\bn (t)$ to $S^{(N)}(t)$ and, by Step 1, to $S(t)$.  From the definitions of  $\bar{S}\bn(t)$ and 
$S^{(N)}(t)$, equation \eqref{interpolantofsk}  and \eqref{piecewiseconstinterS}, respectively, for $(k/N)\leq t < (k+1)/N$, we have
\be\label{starprime}
 \bar{S}\bn(t) - S^{(N)}(t) =   (Nt-k) (S_{k+1}^N-S_k^N) ,
\ee
so that 
\be\label{starrina}
\exo \lv \bar{S}\bn(t) - S^{(N)}(t)\rv^2= \exo \lv (Nt-k) (S_{k+1}^N-S_k^N) \rv^2\leq 
\exo \lv S_{k+1}^N-S_k^N \rv^2 \stackrel{\eqref{diffs1s}}{\les} \frac{1}{N^2}.
\ee

This concludes the proof. 
\end{itemize}
\end{proof}

\begin{proof}[Proof of Lemma \ref{lem:Wassbound} ]
Using \eqref{wasspractical}, the bound \eqref{wassrigi} is a simple consequence of the definitions of $R$ and $G$, indeed

$$
\lv R - G \rv \les \lv \frac{1}{N} \sum_{i=1}^N \lv \xi^i \rv^2  -1 \rv, 
$$
hence
\begin{align}
\EE_k \lv R - G \rv & \leq \left( \EE_k  \lv R - G \rv^2 \right)^{1/2}  \les \left(\EE   \lv \frac{1}{N} \sum_{i=1}^N \lv \xi^i \rv^2  -1 \rv^2\right)^{1/2} \nonumber \\
& = 
\left( Var \left( \frac{1}{N} \sum_{i=1}^N \lv \xi^i \rv^2 \right)\right)^{1/2}\leq \frac{1}{\sqrt{N}}.
\label{RGM2}
\end{align}
We observe now (although we will use this only later) that a similar explicit calculation also shows
\be\label{R-G4M}
\EE_k \lv R - G \rv^4 \les \frac{1}{N^2}. 
\ee
Going back to the proof of the lemma,  the bound \eqref{wassrir} is a direct consequence of the definitions of $R$ and $R^i$.   The inequality 
\eqref{wassrzl} follows from (\ref{defGi}),  (\ref{approx4}), 
  \eqref{defzetakn} and the bound (\ref{c12dpsi}). Indeed, 
$$
 G -  Z_{\ell,k} = \sqrt{\frac{2 \ell^2}{N}} \sum_{j=1}^N \zeta\kjn \xi\kpo\jnn 
-  \sqrt{\frac{2 \ell^2}{N}}  \sum_{j=1}^N \frac{x\kjn}{\lambda_i} \xi \kpo\jnn .
$$
Therefore, by \eqref{defzetakn}, given $x_k$ we have
\be\label{Gzlknormdist}
G -  Z_{\ell,k} \sim  \cN \left( 0, \frac{2 \ell^2}{N} \| \C_N^{1/2}\nabla \Psi^N (x_k^N)\|^2  \right).
\ee
Using (\ref{c12dpsi}) one then has 
\be\label{highmomgz}
\EE_k \lv G -  Z_{\ell,k}\rv^p \les \frac{1+ \nors{x_k^N}^p}{N^{p/2}}, 
\ee
hence \eqref{wassrzl} follows.
Notice that from the above calculations we have 
\be\label{hmrz}
\EE_k \lv R - Z_{\ell,k}\rv^p \les  \frac{1+ \nors{x_k^N}^p}{N^{p/2}}, \quad p \in \{2,4\}. 
\ee
\end{proof}

\begin{proof}[Proof of Lemma \ref{lem:bound2}] We prove the three statements in the order in which they are presented. 

$\bullet$ Proof of the bound  \eqref{lembound22}. 
From \eqref{z=x+psi}, 
\be\label{splitting}
\exo \sum_{i=1}^{N} i^{2s} \lambda_i^2  \lv \zeta\kn \rv^{p} Y
\les \exo \sum_{i=1}^{N} i^{2s} \lambda_i^2  \frac{\lv x\kn \rv^{p}}{\lambda_i^{p}} Y
+   \sum_{i=1}^{N} i^{2s} \lambda_i^2 \exo [(1+ \nors{x_k}^{p}) Y]  \,.
\ee
The second addend is bounded thanks to the assumption on $Y$ and \eqref{boundonmomentsofx}.   As for the first addend, (by the weighted Jensen's inequality) this is bounded (for any $p\geq 0$) as soon as we can prove that 
$$
 v_k^{N}(p):=\sum_{i=1}^{N} i^{2s} \lambda_i^2  \frac{\lv x\kn \rv^{2p}}{\lambda_i^{2p}}
$$
has bounded first  moment (for every $p$), i.e. we want to prove $\exo v_k^N(p)< c$ where $c>0$ is a constant independent of $N$ and $k \in \{0,1 \dd [TN]\}$ (but possibly dependent on $p$). Observe that if $p=1$ then $v_k^N(1)=\nors{x_k^N}^2$, so the statement is a consequence of \eqref{boundonmomentsofx}. So we can restrict to $p\geq 2$. 
Denoting by $d$ a generic constant (that does not depend on $N$), the value of which will change from line to line, we write 
\begin{align*}
\exo v\kpo^N(p) &= \exo\sum_{i=1}^{N} i^{2s} \frac{\lambda_i^2}{\lambda_i^{2p}}  \left( x\kn + \sqrt{\frac{2\ell^2}{N}}\lambda_i \gamma\kpo\xi\kpo\inn\right)^{2p} \\
&\leq \exo v_k^N(p)+ \exo
d \sum_{i=1}^N  i^{2s} \frac{\lambda_i^2}{\lambda_i^{2p}} 
\sum_{m=0}^{2p-1} (x\kn)^m \EE_k  \left( 
 \sqrt{\frac{2 \ell^2}{N}} \lambda_i \gamma\kpo\xi\kpo\inn
  \right)^{2p-m}.
\end{align*} 
If, in the above summation, the index $m$ is smaller than $2p-2$, i.e. $0 \leq m \leq 2p-2$, then $p-m/2\geq 1$ and we have the estimate
\begin{align}
\exo i^{2s} \frac{\lambda_i^2}{\lambda_i^{2p}}  (x\kn)^m \EE_k  \left( 
 \sqrt{\frac{2 \ell^2}{N}} \lambda_i \gamma\kpo\xi\kpo\inn
  \right)^{2p-m} &\les \exo i^{2s} \frac{\lambda_i^2}{N} 
\frac{\lv x\kn\rv^m}{\lambda_i^m} \EE_k \lv \gamma\kpo \xi\kpo\inn\rv^{2p-m}\nonumber\\
& \les \exo i^{2s} \frac{\lambda_i^2}{N} 
\frac{\lv x\kn\rv^m}{\lambda_i^m}\\
&\les \frac{1}{N} i^{2s}\lambda_i^2 \exo \left( \frac{\lv x\kn\rv^{2p}}{\lambda_i^{2p}}+1\right).
 \label{estvm1}
\end{align}
If $m=2p-1$ we instead use \eqref{ekgxi} and obtain
\begin{align}
\exo i^{2s} \frac{\lambda_i^2}{\lambda_i^{2p}} \frac{1}{\sqrt{N}}\lv x\kn\rv^{2p-1} \lv \EE_k (\lambda_i \gamma_k \xi\kpo\inn)\rv &\les \exo 
i^{2s} \lambda_i^2  \frac{\lv x\kn\rv^{2p-1}}{\lambda_i^{2p-1}}
\frac{1}{\sqrt{N}} \left(\frac{\nors{x_k^N}}{\sqrt{N}}+
\frac{\lv x\kn\rv}{\lambda_i} \frac{1}{\sqrt{N}}\right)\nonumber\\
& \les \exo \frac{i^{2s}\lambda_i^2}{N} \left(\frac{\lv x\kn \rv^{2p}}{\lambda_i^{2p}} + \nors{x}^{2p}\right).\label{estvm2}
\end{align}
To obtain the last inequality we used Young's inequality with exponents $2p/(2p-1)$ and $2p$, as follows:
$$
\frac{\lv x\kn\rv^{2p-1}}{\lambda_i^{2p-1}} \cdot 
\nors{x_k^N} \leq   \frac{\lv x\kn \rv^{2p}}{\lambda_i^{2p}} +  \nors{x_k^N}^{2p}. 
$$
From \eqref{estvm1} and \eqref{estvm2} (and using \eqref{boundonmomentsofx}) we then have
$$
\exo v\kpo^N(p) \leq  \left( 1+\frac{d}{N}\right)\exo v_k^N(p) + \frac{d}{N}. 
$$
Iterating the above $[TN]$ times we get
$$
\exo v\kpo^N(p) \leq \left( 1+\frac{d}{N} \right)^{[TN]} v_0(p) + d < \infty,
$$
having denoted $v_0(p):=\sum_{i=1}^{\infty} i^{2s} \lambda_i^2  \frac{\lv x_0^i \rv^{2p}}{\lambda_i^{2p}}$. Notice that if $x_0 \in \hsintt$ then the series $v_0(p)$ is convergent for every $p\geq 0$.

$\bullet${ Proof of the bound \eqref{lembound26}. Set 
\be\label{dddd}
\tilde{\zeta}\kn:=\frac{\lv x\kn\rv }{\lambda_i}, 
\ee
 so that 
$$
S_k^N= \frac{1}{N}\sum_{i=1}^N \frac{\lv x\kn\rv^2}{\lambda_i^2}=
\frac{1}{N}\sum_{i=1}^N \lv \tilde{\zeta}\kn\rv^2.
$$
From \eqref{z=x+psi} we then have
\begin{align}
J_k^N &:= \sum_{i=1}^N i^{2s} \lambda_i^2\frac {\lv \zeta\kn\rv^{p}}{(N S_k^N)^{\alpha}} \nonumber \\
& \les \sum_{i=1}^N i^{2s} \lambda_i^2\frac{\lv \tilde{\zeta}\kn \rv^p}
{(N S_k^N)^{\alpha}} +  \sum_{i=1}^N i^{2s} \lambda_i^2\frac{\nors{x_k^N}^{p\varsigma}+ \nors{x_k^N}^p}{(N S_k^N)^{\alpha}}.\label{lastlast}
  \end{align}
  For the first addend in \eqref{lastlast} we have
\begin{align*}
\exo \sum_{i=1}^N i^{2s} \lambda_i^2\frac{\lv \tilde{\zeta}\kn \rv^p}
{(N S_k^N)^{\alpha}}& = \exo \sum_{i=1}^N  i^{2s} \lambda_i^2 \lv \tilde{\zeta}\kn\rv^{p-2\alpha} 
\left(\frac{\lv \tilde{\zeta}\kn\rv^2}{ \sum_{i=1}^N\lv \tilde{\zeta}\kn\rv^2}\right)^{\alpha}\\
&  
\les \exo v_k^N ((p-2\alpha)/2),
\end{align*}
thanks to the obvious estimate $ \left(\lv \tilde{\zeta}\kn\rv^2/\sum_{i=1}^N\lv \tilde{\zeta}\kn\rv^2\right) \leq 1$. Using \eqref{xnls}, one can easily see that also the expected value of the second addend is bounded if $\varsigma p \geq 2 \alpha$, as
$$
\exo \frac{\nors{x_k^N}^{p\varsigma}}{(N S_k^N)^{\alpha}} =\exo  \frac{\nors{x_k^N}^{2 \alpha}}{(N S_k^N)^{\alpha}} \nors{x_k^N}^{\varsigma p - 2 \alpha}  \les \exo \nors{x_k^N}^{\varsigma p - 2 \alpha}  < \infty.
$$
Therefore,  $\exo J_k^N < \infty$, uniformly over $k$ and $N$. From  the weighted Jensen inequality one can also see that $\exo \lv J_k^N \rv^q < \infty$ for all $q\geq 1$. 
To conclude, for $t_k \leq t < t_{k+1}$ we write
\begin{align*}
\exo J_k^N= \exo J_k^N {\bf 1}_{\{S_k^N\geq (S(t)/2)\}}+
\exo J_k^N {\bf 1}_{\{S_k^N<(S(t)/2)\}}.
\end{align*}
Now the first addend:
\begin{align*}
\exo J_k^N {\bf 1}_{\{S_k^N\geq (S(t)/2)\}} &=
\exo \frac{1}{N^{\alpha}}\sum_{i=1}^N i^{2s}\lambda_i^2\frac{\lv \zeta\kn\rv^p}{(S_k^N)^{\alpha}}{\bf 1}_{\{S_k^N\geq (S(t)/2)\}} \\
& \les \frac{1}{ (S(t))^{\alpha} N^{\alpha}}\exo \sum_{i=1}^N i^{2s}\lambda_i^2 \lv \zeta\kn\rv^p \stackrel{\eqref{lembound22}}{\lra} 0.
\end{align*}
The above limit follows from the assumption $S_0\geq \epsilon$ and 
\eqref{solofODEisbdd} (which, combined, guarantee  $\min\{\epsilon,1\}\leq S(t)$). 
The second addend: 
\begin{align*}
\exo J_k^N {\bf 1}_{\{S_k^N<(S(t)/2)\}} &\leq (\exo \lv J_k^N \rv^2)^{1/2}
\left( \mathbb{P}\left( (S_k^N-S(t))<-\frac{S(t)}{2} \right)\right)^{1/2}\\
&\les \left( \mathbb{P}\left( \lv S_k^N-S(t)\rv >\frac{S(t)}{2} \right) \right)^{1/2}
\les   \frac{1}{S(t)}\left( \exo \lv S_k^N-S(t)\rv^2 \right)^{1/2}.
\end{align*}
The statement now follows from Lemma \ref{corollarylemma}, \eqref{piecewiseconstinterS},  the assumption $S_0\geq \epsilon$ and 
\eqref{solofODEisbdd}.
}

$\bullet$ Finally, we turn to the proof of  \eqref{lembound24}. 
\be\label{dec}
\frac{1}{1+ \lv R\rv \sqrt{N}} = \eta + \frac{1}{1+ \lv Z_{\ell,k}\rv \sqrt{N}} , 
\ee
where
\begin{align*}
\lv \eta \rv & = \lv \frac{1}{1+ \lv R\rv \sqrt{N}} - \frac{1}{1+ \lv Z_{\ell,k}\rv \sqrt{N}}\rv\\
& \leq  \frac{\sqrt{N} \lv \lv R \rv - \lv Z_{\ell,k}\rv\rv}{1+ \lv Z_{\ell,k}\rv \sqrt{N}}
\leq \frac{\sqrt{N}\lv R-Z_{\ell,k} \rv}{1+ \lv Z_{\ell,k}\rv \sqrt{N}},
\end{align*}
having used $\lv \lv a \rv - \lv b \rv\rv\leq \lv a-b \rv$. Consequently, 
\begin{align}
\EE_k \lv \eta \rv^2 & \leq N \left( \EE_k \lv R-Z_{\ell,k} \rv^4 \right)^{1/2}
  \left( \EE_k \frac{1}{(1+ \lv Z_{\ell,k}\rv \sqrt{N})^4} \right)^{1/2}\\
& \stackrel{\eqref{hmrz}}{\les} {(1+\nors{x_k^N}^2)} 
\left( \EE_k \frac{1}{(1+ \lv Z_{\ell,k}\rv \sqrt{N})^4} \right)^{1/2}. 
\label{pen}
\end{align}
Also, from \eqref{dec},
\begin{align}\label{pen1}
\EE_k \frac{1}{(1+ \lv R\rv \sqrt{N})^2}& \les \EE_k \eta^2 + \EE_k \frac{1}{(1+ \lv Z_{\ell,k}\rv \sqrt{N})^2}\,.
\end{align}
Now notice that, given $x_k$,  $Z_{\ell, k}$ is Gaussian with variance $2 \ell^2 S_k^N$ and mean $-\ell^2$. Therefore, for every $p>1$, 
\begin{align*}
\EE_k \frac{1}{(1+ \lv Z_{\ell,k}\rv \sqrt{N})^p} &= \int_{\R}
\frac{1}{(1+ \lv x\rv \sqrt{N})^p} \frac{1 }{\sqrt{4 \pi \ell^2 S_k^N}} e^{-(x+ \ell^2)^2/ (2\ell^2 S_k^N)} dx\\
& \les \int_{\R}  \frac{1}{(1+ \lv x\rv \sqrt{N})^p} \frac{1 }{\sqrt{S_k^N}} dx\\
& = \frac{1}{\sqrt{N}}\int_{\R} \frac{1}{(1+ \lv y\rv )^p} \frac{1 }{\sqrt{S_k^N}} dy \les  \frac{1 }{ \sqrt{ N S_k^N}}. 
\end{align*}
The proof can now be concluded by combining \eqref{pen1}, \eqref{pen} and the above.
\end{proof}
\begin{remark}\label{rem810}
\textup{Notice that the proof of \eqref{lembound26} is the only place in which we actually use \eqref{der-s} instead of the slightly more general assumption $\|\C^{1/2} \nabla\Psi\|\leq \|\nabla \Psi\|_{-s}\les 1+ \nors{x}$. This is to avoid technicalities and streamline the proof. 
}\hfill{$\Box$}
\end{remark}

{\bf Acknowledgments}  We are extremely grateful  to the anonimous referee for his/her careful reading, for spotting  mistakes in an earlier version and for comments that helped improving  the paper.

\bibliographystyle{alpha}
\bibliography{difflim}

\end{document}